\documentclass[12pt]{article}
\usepackage{amssymb,amsmath,amsthm,mathrsfs}
\usepackage[usenames]{color}
\usepackage{hyperref}
\usepackage[all]{xy}
\usepackage{xypic}
\usepackage{graphicx}
\usepackage{amscd}
\usepackage{enumerate}

\baselineskip 18pt \textwidth 16cm  \theoremstyle{plain}

\newtheorem{theorem}{Theorem}[subsection]
\newtheorem{proposition}[theorem]{Proposition}
\newtheorem{corollary}[theorem]{Corollary}
\newtheorem{lemma}[theorem]{Lemma}
\newtheorem{definition}[theorem]{Definition}
\newtheorem{notation}[theorem]{Notation}

\newtheorem{example}[theorem]{Example}
\newtheorem{remark}[theorem]{Remark}

\newtheorem{claim}[theorem]{Claim}
\newtheorem{prop}[theorem]{Proposition}

\newtheorem{theoremABC}{Theorem}

\newtheorem{theoremRoman}{Theorem}

\newtheorem{propositionRoman}[theoremRoman]{Proposition}
\newtheorem{corollaryRoman}[theoremRoman]{Corollary}
\newtheorem{definitionRoman}[theoremRoman]{Definition}

\newcommand{\bmat}{\left( \begin{matrix}}
\newcommand{\emat}{\end{matrix}\right)}

\newcommand{\g}{\mathfrak g}
\newcommand{\Z}{\mathbb Z}

\newcommand{\nir}[1]{{{#1}}}
\newcommand{\RamiA}[1]{{{#1}}}
\newcommand{\RamiAlt}[1]{}
\newcommand{\lbl}[1]{\label{#1}}

\DeclareMathOperator{\SL}{SL}
\DeclareMathOperator{\GL}{GL}

\DeclareMathOperator{\Sp}{Sp}
\DeclareMathOperator{\SO}{SO}
\DeclareMathOperator{\val}{val}

\DeclareMathOperator{\tr}{tr}

\DeclareMathOperator{\Hom}{Hom}

\DeclareMathOperator{\Ad}{Ad}

\DeclareMathOperator{\Mat}{Mat}

\DeclareMathOperator{\Lie}{Lie}

\DeclareMathOperator{\rk}{rk}

\DeclareMathOperator{\Spec}{Spec}

\DeclareMathOperator{\gr}{gr}
\DeclareMathOperator{\Rees}{Rees}

\DeclareMathOperator{\Sym}{Sym}
\DeclareMathOperator{\linspan}{span}
\DeclareMathOperator{\Def}{Def}

\DeclareMathOperator{\Real}{Re}

\DeclareMathOperator{\supp}{supp}
\DeclareMathOperator{\VF}{VF}
\DeclareMathOperator{\VG}{VG}
\DeclareMathOperator{\Tr}{Tr}
\DeclareMathOperator{\Ker}{Ker}
\DeclareMathOperator{\Grass}{Grass}

\DeclareMathOperator{\Irr}{Irr}
\newcommand{\FRS}{(FRS)}
\newcommand{\A}{\mathbb A}
\newcommand{\Pg}{Polygraph}
\newcommand{\pg}{polygraph}

\usepackage[usenames,dvipsnames]{xcolor}

\definecolor{dblue}{RGB}{6,69,173}
\definecolor{lblue}{RGB}{11,0,128}

\newcommand{\colorlinks}{true}

\newcommand{\linkcolor}{lblue}
\newcommand{\citecolor}{lblue}
\newcommand{\urlcolor}{dblue}

\newcommand{\linkbordercolor}{red}
\newcommand{\citebordercolor}{green}
\newcommand{\urlbordercolor}{cyan}

\usepackage{hyperref}
\hypersetup{colorlinks=\colorlinks,linkbordercolor=\linkbordercolor, linkcolor=\linkcolor, citecolor=\citecolor, urlcolor=\urlcolor,urlbordercolor=\urlbordercolor,citebordercolor=\citebordercolor}%

\newcommand{\hrefHid}[2]{
\hypersetup{urlbordercolor={1 1 1}}%
\hypersetup{urlcolor=black}%
\href{#1}{#2}%
\hypersetup{urlbordercolor=\urlbordercolor}%
\hypersetup{urlcolor=\urlcolor}%
}

\newcommand{\inhref}[2]{\hyperref[#1]{#2}}

\newcommand{\inhrefHid}[2]{%
\hypersetup{linkbordercolor={1 1 1}}%
\hypersetup{linkcolor=black}%
\inhref{#1}{#2}%
\hypersetup{linkbordercolor=\linkbordercolor}%
\hypersetup{linkcolor=\linkcolor}%
}

\newcommand{\defHref}[3]{\newcommand{#1}[1][#3]{\href{#2}{##1}}}
\newcommand{\defInhref}[3]{\newcommand{#1}[1][#3]{\inhref{#2}{##1}}}
\newcommand{\defHrefHid}[3]{\newcommand{#1}[1][#3]{\hrefHid{#2}{##1}}}
\newcommand{\defInhrefHid}[3]{\newcommand{#1}[1][#3]{\inhrefHid{#2}{##1}}}

\newcommand{\defHrefBoth}[3]{%
\expandafter\defHrefHid \csname #3Hid\endcsname {#1}{#2}%
\expandafter\defHref \csname #3Vis\endcsname {#1}{#2}%
}
\newcommand{\defInhrefBoth}[3]{%
  \expandafter\defInhrefHid \csname #3Hid\endcsname {#1}{#2}%
  \expandafter\defInhref \csname #3Vis\endcsname {#1}{#2}%
}

\newcommand{\defHrefBothVis}[3]{%
\defHrefBoth{#1}{#2}{#3}%
\expandafter\defHref \csname #3\endcsname {#1}{#2}%
}
\newcommand{\defInhrefBothVis}[3]{%
  \defInhrefBoth{#1}{#2}{#3}%
  \expandafter\defInhref \csname #3\endcsname {#1}{#2}%
}

\newcommand{\defHrefBothHid}[3]{%
\defHrefBoth{#1}{#2}{#3}%
\expandafter\defHrefHid \csname #3\endcsname {#1}{#2}%
}
\newcommand{\defInhrefBothHid}[3]{%
  \defInhrefBoth{#1}{#2}{#3}%
  \expandafter\defInhrefHid \csname #3\endcsname {#1}{#2}%
}

\newcommand{\term}[2]{%
\label{#2}%
\emph{#1}%
 \globaldefs =1%
\defInhrefBothHid{#2}{#1}{#2}%
 \globaldefs =0%
}

\baselineskip 18pt \textwidth 16cm  \theoremstyle{plain}

\newtheorem{thm}[theorem]{Theorem}
\newtheorem{lem}[theorem]{Lemma}
\newtheorem{defn}[theorem]{Definition}
\newtheorem{notn}[theorem]{Notation}
\newtheorem{cor}[theorem]{Corollary}

\newtheorem{rem}[theorem]{Remark}

\newtheorem*{theorem*}{Theorem}
\newtheorem*{remark*}{Remark}
\newtheorem*{conjecture*}{Conjecture}

\newcommand{\Ga}{\Gamma}

\newcommand{\cF}{\mathcal F}
\newcommand{\cG}{\mathcal G}
\newcommand{\cI}{\mathcal I}

\newcommand{\cM}{\mathcal M}

\newcommand{\cO}{\mathcal O}

\newcommand{\cW}{\mathcal W}

\newcommand{\bD}{\mathbb D}

\newcommand{\fg}{\mathfrak g}

\newcommand{\C}{\mathbb C}

\DeclareMathOperator{\spec}{Spec}

\newcommand{\Sc}{{\mathcal S}}

\renewcommand{\dim}{{\operatorname{dim}}}

\renewcommand{\Hom}{{\operatorname{Hom}}}

\newcommand{\et}{{\'{e}tale }}

\title{Representation Growth and Rational Singularities of the Moduli Space of Local Systems}

\author{Avraham Aizenbud\thanks{Faculty of Mathematics and Computer Science, The Weizmann Institute of Science POB 26, Rehovot 76100, ISRAEL. E-mail: aizenr@gmail.com, URL: \url{http://www.wisdom.weizmann.ac.il/\~aizenr}.} \hspace{0.1cm} and  Nir Avni\thanks{ Department of Mathematics, Northwestern University, Evanston, IL 60201, USA.
E-mail: avni.nir@gmail.com. URL: \url{http://math.northwestern.edu/\~nir}.\smallskip
}}

\usepackage{varioref}

\newcommand{\simpAr}[2][r]{%
\ar@{}[#1]|-*[@]_{#2}%
}
\usepackage[normalem]{ulem}

\begin{document}

\maketitle

\begin{abstract}
Let $G$ be a semisimple algebraic group defined over $\mathbb{Q}_p$, and let $\Gamma$ be a compact open subgroup of $G(\mathbb{Q}_p)$. We relate the asymptotic representation theory of $\Gamma$ and the singularities of the moduli space of $G$-local systems on a smooth projective curve, proving new theorems about both: \begin{enumerate}

\item We prove that there is a constant $C$, independent of $G$, such that the number of $n$-dimensional representations of $\Gamma$ grows slower than $n^{C}$, confirming a conjecture of Larsen and Lubotzky. In fact, we can take $C=3\cdot \dim(E_8)+1=745$. We also prove the same bounds for groups over local fields of large enough characteristic.

\item We prove that the coarse moduli space of $G$-local systems on a smooth projective curve of genus at least $\lceil C/2\rceil+1=374$ has rational singularities.
\end{enumerate}
For the proof, we study the analytic properties of push forwards of smooth measures under algebraic maps. More precisely, we show that such push forwards have continuous density if the algebraic map is flat and all of its fibers have rational singularities.
\end{abstract}

\setcounter{tocdepth}{3}
\newpage
\tableofcontents

\section{Introduction}
\subsection{Summary of The Main Results}
 The results of this paper relate to three topics. The first is the asymptotic behavior of the number of irreducible complex representations of some pro-finite groups. We prove

\begin{theoremABC} \lbl{thm:intro.SLd.representation.growth} Let $F$ be  a non-archimedean local field of characteristic 0, let $G$ be a semi-simple algebraic group defined over $F$, and let $\Gamma$ be a compact open subgroup of $G(F)$. Then there exists a constant $C$ such that, for all integers $N$, the number of irreducible $N$-dimensional complex representations of $\Gamma$ is less than $C N^{745}$.
\end{theoremABC}
For most Lie types, our bounds are better; for example, for groups of type $A_n$, the bound is $CN^{22}$. For the precise bounds and a generalization of the result for local rings of large enough characteristic, see Theorem \ref{thm:intro.abscissa.SLd}.

Denote the set of irreducible characters of $\Gamma$ by $\Irr \Gamma$. The proof of Theorem \ref{thm:intro.SLd.representation.growth} shows that the sum $\sum_{\chi \in \Irr(\Gamma)} \chi(1)^{-2n}$ converges for $n$ large enough. We show that this sum is equal to the normalized symplectic volume of the moduli space of $\Gamma$-local systems on a closed orientable surface of Euler characteristic $-2n$; see Theorem \ref{thm:intro.volume.formula}.

The second topic of this article is the study of the singularities of the deformation variety of a surface group inside an algebraic group,
\[
\Def_{G,n}:=\Hom(\pi_1(\Sigma_n),G)= \left\{ (g_1,h_1,\ldots,g_n,h_n)\in G^{2n} \mid [g_1,h_1] \cdots [g_n,h_n]=1 \right\}
\]
where $\Sigma_n$ is a closed orientable surface of genus $n$ and $G$ is an algebraic group. We prove

\begin{theoremABC} \lbl{thm:intro.SLd.singularities.deformation} Let $G$ be a semi-simple complex algebraic group, and let $\Sigma$ be a closed orientable surface of genus greater than or equal to 374. The variety $\Hom(\pi_1(\Sigma),G)$ has rational singularities\footnote{For the notion of rational singularity, see Definition \ref{defn:rational.singularities}}.
\end{theoremABC}
See Theorem \ref{thm:intro.SLd.FRS} for a stronger statement.

We deduce from Theorem \ref{thm:intro.SLd.singularities.deformation} that the moduli space of $G$-local systems on any algebraic curve of genus at least 373 has rational singularities, see Theorem \ref{thm:RS.local.systems}.

The third topic of this article is the study of push forwards of measures under algebraic maps between smooth algebraic varieties over local fields. Such varieties look locally like $p$-adic balls, and, thus, we can define the notions of smooth\footnote{See \S\S \ref{subsubsec.direct.image} for the definition of smooth measure.} (or locally constant) measures or measures with continuous density with respect to some smooth measure. We prove

\begin{theoremABC} \lbl{thm:intro.push.forward} Let $X$ and $Y$ be smooth irreducible varieties over a local field $F$ of characteristic 0, and let $\phi : X \rightarrow Y$ be a flat map such that, for all $y\in Y$, the fiber $\phi ^{-1} (y)$ has rational singularities. For every smooth compactly supported measure $m$ on $X(F)$, the push forward $\phi _*m$ has continuous density.
\end{theoremABC}
See Theorem \ref{thm:push.forward}, which also includes a converse result.

\subsection{Longer Discussion of the Main Results}

Let $\Gamma$ be a group, and let $\pi$ be the fundamental group of a closed orientable surface of genus $n \geq 2$. The deformation space of $\pi$ inside $\Gamma$ is the set of homomorphisms from $\pi$ to $\Gamma$; we will denote it by $\Def_{\Gamma,n}$. When $\Gamma$ is an algebraic group (or a $p$-adic analytic group, or a topological group), $\Def_{\Gamma,n}$ is a variety (or a $p$-adic analytic variety, or a topological space, respectively).

Let $G$ be an algebraic group defined over a local field $F$. We study the geometry of $\Def_{G,n}$ and its connections to the representation theory of compact open subgroups of $G(F)$. These connections allow us to prove new theorems both on the singularities of $\Def_{G,n}$ and on the asymptotic representation theory of open compact subgroups of $G(F)$. A prototype of such connection is the following theorem of Frobenius:

\begin{theorem}[{see, e.g. \cite[Lemma 3.1]{LS2}}] \lbl{thm:Frobenius} Let $\Gamma$ be a finite group, and let $n \geq 1$ be an integer. For every $g\in \Gamma$, the number of solutions to the equation $[x_1,y_1] \cdots [x_n,y_n]=g$ is equal to
\[
| \Gamma |^{2n-1}\sum_{\chi \in \Irr(\Gamma)} \frac{\chi(g)}{\chi(1)^{2n-1}}.
\]
\end{theorem}

In order to extend Theorem \ref{thm:Frobenius} to pro-finite groups, we introduce some notation.
\begin{definition} \lbl{defn:Phi} Let $\Gamma$ be a group and let $n \geq 1$ be an integer. Define $\Phi_{\Gamma,n}: \Gamma^{2n} \rightarrow \Gamma$ to be the map
\[
\Phi_{\Gamma,n}(x_1,y_1,\ldots,x_n,y_n)=[x_1,y_1] \cdots [x_n,y_n] .
\]
\end{definition}

Note that $\Def_{\Gamma,n}$ is the fiber $\Phi_{\Gamma,n} ^{-1}(1)$ of $\Phi_{\Gamma,n}$.

\begin{definition} Let $X,Y$ be measurable spaces, let $m$ be a measure on $X$, and let $f:X \rightarrow Y$ be a measurable function. The \emph{push forward} of $m$ via $f$ is the measure $f_*m$ given by $f_*m(A)=m(f ^{-1} (A))$ for every $A \subset Y$.
\end{definition}

We will exploit the following extension of Frobenius' Theorem:

\begin{propositionRoman}[see \S\ref{subsec:special.values}] \lbl{prop:Frobenius.Formula}  Let $\Gamma$ be a finitely generated pro-finite group, and let $\Gamma(i)_{i\in \mathbb{N}}$ be a decreasing chain of open normal subgroups with trivial intersection. Let $n \geq 2$ be an integer. Denote the normalized Haar measure on $\Gamma$ by $\lambda_\Gamma$ and the normalized Haar measure on $\Gamma ^{2n}$ by $\lambda_{\Gamma ^{2n}}$. The following are equivalent: \begin{enumerate}
\item The measure $(\Phi_{\Gamma,n})_* \lambda_{\Gamma ^{2n}}$ has a continuous density with respect to $\lambda_\Gamma$.
\item There is a constant $C$ such that $\lambda_{\Gamma^{2n}} \left( \Phi_{\Gamma,n} ^{-1} (\Gamma(i))\right)\leq C \cdot \lambda_\Gamma (\Gamma(i))$ for all $i$.
\item \label{cond:Frob.Formula.3} The series $\sum_{\chi \in \Irr \Gamma} \chi(1)^{2-2n}$ converges.
\end{enumerate}
If these conditions hold, then the density of $(\Phi_{\Gamma,n})_* \lambda_{\Gamma ^{2n}}$ with respect to $\lambda_\Gamma$ is given by the function
\[
g \mapsto \sum_{\chi \in \Irr \Gamma} \frac{\chi(g)}{\chi(1)^{2n-1}}.
\]
\end{propositionRoman}

\subsubsection{Direct Image of Smooth Measures} \lbl{subsubsec.direct.image}

We concentrate on the first condition in Proposition \ref{prop:Frobenius.Formula}, namely, the continuity of the density of $(\Phi_{\Gamma,n})_* \lambda_{\Gamma^{2n}}$ with respect to $\lambda_\Gamma$. We treat this question for $p$-adic groups, and, more generally, for $p$-adic varieties.

\begin{definition} Let $X$ be a smooth $d$-dimensional algebraic variety over a non-archimedean local field $F$. Denote the ring of integers of $F$ by $O$. \begin{enumerate}
\item A measure $m$ on $X(F)$ is called {\emph{smooth}} if every point $x\in X(F)$ has an analytic neighborhood $U$ and a ($p$-adic analytic) diffeomorphism $f:U \rightarrow O^d$ such that $f_*m$ is some Haar measure on $O^d$.
\item A measure on $X(F)$ is called {\emph{Schwartz}} if it is smooth and compactly supported. We denote the space of all Schwartz measures on $X(F)$ by $\Sc(X(F))$.
\item We say that a measure $\mu$ on $X(F)$ {\emph{has continuous density}}, if there is a smooth measure $m$ and a continuous function $f:X(F) \rightarrow \mathbb{C}$ such that $\mu=f \cdot m$.
\end{enumerate}
\end{definition}

We are interested in finding conditions on a map $\phi: X \rightarrow Y$ between two smooth varieties that will imply that $\phi _*m$ has continuous density for any Schwartz measure $m$ on $X(F)$.

A sufficient condition for the continuity of direct images is that $\phi$ is a smooth map (i.e. a submersion); see Proposition \ref{prop:push.forward.smooth.map}. Recall that a map is smooth if and only if it is flat and all its fibers are smooth. We show that this last condition can be relaxed: it is enough to require that the map is flat and all fibers have rational singularities. We recall the notion of rational singularities in \ref{defn:rational.singularities}.

\begin{definitionRoman} Let $X$ and $Y$ be smooth irreducible varieties over a field $k$ of characteristic 0, and let $\phi : X \rightarrow Y$ be a morphism. We say that $\phi$ is \FRS\ 
if it is flat and, for any $y\in Y(\overline{k})$, the fiber $X \times_Y y$ is reduced and has rational singularities.
\end{definitionRoman}

\begin{theoremRoman} \lbl{thm:push.forward} Let $X$ and $Y$ be smooth irreducible varieties over a local field $F$ of characteristic 0, and let $\phi : X \rightarrow Y$. \begin{enumerate}
\item Assume that $\phi$ is \FRS. Then, for every Schwartz measure $m$ on $X(F)$, the push forward $\phi _*m$ has continuous density.
\item Conversely, assume that, for every finite extension $F'/F$ and every Schwartz measure $m'$ on $X(F')$, the measure $\phi_*m'$ has continuous density. Then $\phi$ is \FRS.
\end{enumerate}
\end{theoremRoman}

See Theorem \ref{thm:push.forward.detailed} for a generalization.

\subsubsection{Relations Between Representation Growth and Deformation Variety}
We now concentrate on Condition \eqref{cond:Frob.Formula.3} in Proposition \ref{prop:Frobenius.Formula}. For a topological group $\Gamma$, we denote the number of continuous irreducible complex representations of $\Gamma$ whose dimension is at most $n$ by $R_n(\Gamma)$. If $\Gamma$ is a finitely generated pro-finite group, a necessary and sufficient condition for $R_n(\Gamma)< \infty$ for all $n$ is that every finite-index subgroup of $\Gamma$ has a finite abelianization (see \cite{BLMM}). If this condition holds, we say that $\Gamma$ is FAb. Examples of FAb groups are compact open subgroups in semi-simple algebraic groups over local fields.

For a FAb group $\Gamma$, the representation growth of $\Gamma$ is the study of the asymptotic behavior of the sequence $R_n(\Gamma)$. In the case where the representation growth of $\Gamma$ is polynomially bounded, we introduce the following generating function:

\begin{definition} Let $\Gamma$ be a topological group, and assume that there is a constant $C$ such that, for any $n$, there are at most $Cn^C$ non-isomorphic irreducible complex continuous representations of dimension $n$ of $\Gamma$. The {\emph{representation zeta function}} of $\Gamma$ is the function
\[
\zeta_\Gamma(s)=\sum_{\chi \in \Irr \Gamma} \chi(1)^{-s},
\]
defined in $\left\{ s \in \mathbb{C} \mid \Real(s) > C+1 \right\}$.
\end{definition}

Suppose that $G$ is a semi-simple group over a non-archimedean local field $F$ of characteristic 0 and $\Gamma$ is a compact open subgroup of $G(F)$. The main theorem of \cite{Ja} implies that $\zeta_\Gamma(s)$ has meromorphic continuation to the whole complex plane, and there is a rational number $\alpha=\alpha(\Gamma)$ such that $R_n(\Gamma) = n^{\alpha+o(1)}$. This $\alpha$ is the maximum of the real values of the poles of $\zeta_\Gamma(s)$, but its exact value is still unknown. The results in \cite{Ja} use the orbit method for analytic pro-$p$ groups and, therefore, are restricted to characteristic 0.

We can now make a more explicit connection between the geometry of $\Def_{G,n}$ and representations of compact open subgroups of $G(F)$:

\begin{theoremRoman}[see \S\S \ref{subsec:special.values}] \lbl{thm:def.and.representation.growth}  Let $G$ be a semi-simple algebraic group defined over a finitely generated field $k$ of characteristic 0. The following are equivalent: \begin{enumerate}
\item The point $(1,\ldots,1)$ is a rational singularity of the deformation variety $\Def_{G,n}=(\Phi _{G,n}) ^{-1} (1)$.
\item $\Phi_{G,n}$ is \FRS.
\item For every non-archimedian local field $F$ containing $k$ and every compact open subgroup $\Gamma \subset G(F)$, we have $\alpha(\Gamma)<2n-2$.
\item For every finite extension $k'/k$, there is a local field $F'$ containing $k'$ and a compact open subgroup $\Gamma \subset G(F)$ such that $\alpha(\Gamma)<2n-2$.
\end{enumerate}
\end{theoremRoman}

Regarding groups over local fields of positive characteristic, we show

\begin{theoremRoman}[see \S\S \ref{subsec:pos.char}] \lbl{thm:intro.pos.zero.local.zeta}  Let $G$ be an affine group scheme over a localization of $\mathbb{Z}$ by finitely many primes. Assume that the generic fiber of $G$ is semi-simple, and let $n \geq 1$ be an integer. There is a constant $p_0$ such that, if $F_1,F_2$ are local fields with isomorphic residue fields of characteristic greater than $p_0$, and if $\Gamma_1,\Gamma_2$ are compact open subgroups of $G(F_1),G(F_2)$ respectively, then $\alpha(\Gamma_1)<2n$ if and only if $\alpha(\Gamma_2)<2n$.

Moreover, for such $n$, if $O_1$ and $O_2$ denote the rings of integers of $F_1$ and $F_2$, then, $\zeta_{G(O_1)}(2n)=\zeta_{G(O_2)}(2n)$.
\end{theoremRoman}

Our next result relates the value of $\zeta_\Gamma$ at an even positive integer with the volume of the space of $\Gamma$-local systems on a surface. Suppose that $G$ is a semi-simple algebraic group defined over a non-archimedean field $F$ of characteristic 0, that $\Gamma \subset G(F)$ is a compact open subgroup, and that $\Sigma$ is a compact orientable surface. The collection of all $\Gamma$-local systems on $\Sigma$ is in bijection with the quotient of $\Def_{\Gamma,n}$ by the conjugation action of $\Gamma$. We say that a homomorphism $\rho \in \Def_{\Gamma,n}$ is open if the closure of $\rho(\pi_1(\Sigma))$ is open in $\Gamma$. The set of open homomorphisms, which we denote by $\Def_{\Gamma,n}^{open}$, is open, dense, and $\Gamma$-invariant subset of $\Def_{\Gamma,n}$, and the quotient $\Def_{\Gamma,n}^{open} / \Gamma$ is a $p$-adic analytic manifold (see \S\S\ref{subsubsec:localization}). Atiyah, Bott, and Goldman defined a symplectic form on $\Def_{\Gamma,n}^{open} / \Gamma$ (see \S\S\ref{subsubsec:localization}), from which we get a top form $v_{ABG}$. Following \cite{Wi}, we show
{
\begin{theoremRoman}[see \S\S \ref{subsubsec:localization}] \lbl{thm:intro.volume.formula}  Let $F$ be a non-archimedean local field of characteristic 0 and residue characteristic different from 2, and let $G$ be a connected, simply connected, semi-simple algebraic group over $F$ with a Lie algebra $\mathfrak{g}$. \nir{The Killing form of $\mathfrak{g}$ gives rise to a Haar measure $\mu$ on $G(F)$ via the Gram determinant.} Suppose that $\Gamma \subset G(F)$ is a compact open subgroup and $n\in \mathbb{Z}_{>1}$ is such that $\Phi_{G,n}$ is \FRS. Denote the measure on $\Def_{\Gamma,n}^{open}/ \Gamma$ corresponding to $v_{ABG}$ by $| v_{ABG}|$ and the measure on $\Gamma$ corresponding to $\omega$ by $| \omega |$. Then
\[
\int_{\Def_{\Gamma,n}^{open} / \Gamma} | v_{ABG} | = |Z(\Gamma)| \cdot \left( \mu(\Gamma) \right )^{2n-2} \cdot\zeta_{\Gamma}(2n-2)
\]
\end{theoremRoman}
}

\subsubsection{Results on Representation Growth and Singularities of Deformation Varieties} \lbl{subsect:intro.SLd}

\begin{definitionRoman} \lbl{defn:bound.root} $ $\begin{enumerate}
\item Suppose that $\mathfrak{g}$ is a simple algebraic group over an algebraically closed field of characteristic 0. Let
\[
\mathcal{B}(\mathfrak{g})=\left\{ \begin{matrix} 22 & \textrm{$\mathfrak{g}=\mathfrak{sl}_d$ or $\mathfrak{g}=\mathfrak{so}_d$} \\ 40 & \mathfrak{g}=\mathfrak{sp}_{2d} \\ 3\dim(\mathfrak{g})+1 & \textrm{$\fg$ is exceptional} \end{matrix}\right.
\]
\item Suppose $G$ is a semi-simple algebraic group over a field $k$ of characteristic 0. Define $\mathcal{B}(G)$ to be the maximum of $\mathcal{B}(\fg)$, where $\fg$ is a simple factor of $\Lie(G) \otimes \overline{k}$.
\end{enumerate}
\end{definitionRoman}

\begin{theoremRoman}[see \S\S\ref{subsec:special.values}]  \lbl{thm:intro.SLd.FRS} Let $G$ be a semi-simple algebraic group over a field $k$ of characteristic 0, and let $n$ be an integer such that $n \geq \mathcal{B}(G)/2+1$.

Then, the map $\Phi_{G,n}: G^{2n} \rightarrow G$ is \FRS.
\end{theoremRoman}

In particular, the deformation variety $\Def_{G,n}=\Phi_{G,n} ^{-1}(1)$ has rational singularities if $n \geq \mathcal{B}(G)/2+1$. Using \cite{B}, we conclude that the categorical quotient $\Def_{G,n}/G$ has rational singularities. Consider the case $k=\C$. Fix a complex smooth projective curve $C$ of genus $n$. The Riemann--Hilbert correspondence gives rise to an analytic isomorphism between $\Def_{G,n}/G$ and the coarse moduli space of $G$-principal bundles together with a  connection on $C$ (see \cite[Proposition 7.8]{Si}). Since the notion of rational singularities depends only on the underlying analytic variety (see \cite{Bur}), we get

\begin{theoremRoman} \lbl{thm:RS.local.systems} Let $G$ be a semi-simple algebraic group over a field $k$ of characteristic 0, and let $C$ be   a smooth projective curve of genus at least $ \mathcal{B}(G)/2+1$.

Then the coarse moduli space of $G$-principal bundles with flat connections on $C$ has rational singularities.
\end{theoremRoman}

In a different direction, combining Theorems \ref{thm:def.and.representation.growth} and \ref{thm:intro.SLd.FRS}, we get

\begin{theoremRoman} \lbl{thm:intro.abscissa.SLd} Let $G$ be an affine group scheme over a localization of $\mathbb{Z}$ by finitely many primes. Assume that the generic fiber of $G$ is semi-simple. There is a number $p_0$ such that, if $O$ is the ring of integers in some local field $F$ of characteristic greater than $p_0$, then $$\alpha(G(O))<\mathcal{B}(G).$$
In particular, $R_n(G(O)) = o\left(n^{\mathcal{B}(G)}\right).$
\end{theoremRoman}

\subsection{Ideas of the Proofs}

The proofs of Theorems \ref{thm:push.forward} and \ref{thm:intro.SLd.FRS} are based on a theorem of Elkik which asserts that flat deformations of rational singularities are rational singularities (see \ref{thm:elk} for a precise statement).

\subsubsection{Proof of the \FRS\ Property (Theorem \ref{thm:intro.SLd.FRS})}
The notion of rational singularities is defined using a resolution of singularities. We avoid the hard problem of finding a resolution of $\Def_{G,n}$ by using Elkik's theorem. More precisely, in order to prove Theorem \ref{thm:intro.SLd.FRS}, we find degenerations of the variety $\Def_{G,n}$, i.e., we find a flat family of varieties whose generic member is $\Def_{G,n}$ and whose special member is a simpler variety which is easier to analyze. By Elkik's theorem, if the special member has rational singularity, so does the generic member. Our degenerations come from \nir{$\mathbb{G}_m$-}actions on affine varieties, or, equivalently, from filtrations on the coordinate algebras. We recall the notions of good filtration, stable points of the spectrum of a filtered ring, and prove the following

\begin{propositionRoman} (see Corollary \ref{cor:degeneration}) Let $A$ and $B$ be $k$-algebras with good filtrations, let $\varphi :A \rightarrow B$ be a filtration-preserving homomorphism, and let $p:A \rightarrow k$ be a stable $k$-point of $\Spec(A)$. If the associated graded $\gr(\varphi):\gr(A) \rightarrow \gr(B)$ is \FRS\ at $\gr(p):\gr(A) \rightarrow k$, then $\varphi$ is \FRS\ at $p$.
\end{propositionRoman}

In practice, it is easier to do the degeneration in several steps. The first step degenerates $\Def_{G,n}$ to its Lie algebra version
\[
\left\{ (X_1,Y_1,\ldots,X_n,Y_n) \in \mathfrak{g} ^{2n} \mid [X_1,Y_1]+ \cdots +[X_n,Y_n]=0 \right\}.
\]
Further steps degenerate this variety to a variety from the following class:

\begin{definitionRoman} \label{defn:symplectic.graph.variety} Let $\Upsilon=(V,E)$ be a (combinatorial) graph, and let $(W,\omega)$ be a symplectic vector space. The {\emph{symplectic graph variety}} of $\Upsilon$ and $W$ is the variety
\[
X_{\Upsilon,W}= \left\{ (w_v)_{v \in V}\in W^V \mid \textrm{$   \omega(w_v,w_u)=0$ for all $\left\{ u,v \right\} \in E$} \right\} .
\]
\end{definitionRoman}
Finally, we degenerate the symplectic graph variety to a different symplectic graph variety for which the graph $\Upsilon$ is a disjoint union of edges. This last variety is the product of varieties of the form
\[
\left\{ (v_1,v_2) \in W \times W  \mid \omega(v_1,v_2)=0 \right\},
\]
which are easy to desingularize explicitly and are easily shown to have rational singularities. As an intermediate step, we prove

\begin{theoremRoman} \label{thm:FRS.tree} Let $\Upsilon$ be a tree of maximal degree $d$, and let $W$ be a (non-zero) symplectic space of dimension greater than or equal to $4(d-1)$. Then $X_{\Upsilon,W}$ has rational singularities.
\end{theoremRoman}
See Theorem \ref{prop:frs_tree} for a stronger version.

The argument up to this point shows only that the point $(1,\ldots,1)$ is a rational singularity of $\Def_{G,n}$, because of the first degeneration. In order to bootstrap to the whole variety, we use the relation between representation growth and rational singularities described in Theorem \ref{thm:def.and.representation.growth}. More precisely, we show that if $(1,\ldots,1)$ is a rational singularity of $\Def_{G,n}$, then $\alpha(\Gamma) <2n-2$ for \emph{some} congruence subgroup $\Gamma \subset G(F)$ for any local field $F$. This implies, using a simple argument, that $\alpha(\Delta) < 2n-2$ for \emph{any} compact open subgroup of $G(F)$, for any local field $F$. This last statement implies that $\Def_{G,n}$ has rational singularities, by Theorem \ref{thm:def.and.representation.growth}.

We apply the above strategy for classical groups. For exceptional groups, we use instead the known bounds on representation growth and Theorem \ref{thm:def.and.representation.growth}.

\subsubsection{Continuity of Push Forward (Theorem \ref{thm:push.forward})}
Theorem \ref{thm:push.forward} is easy in the case where the map $\phi:X \rightarrow Y$ is smooth. Indeed, Schwartz measures on $X$ look locally like $f| \omega_X |$, where $f$ is a Schwartz function, $\omega_X$ is a top differential form, and $| \omega_X |$ is the measure corresponding to $\omega_X$, see \S\S \ref{subsec:measure.forms}. If $\phi$ is smooth and $\omega_Y$ is a non-vanishing top differential form on $Y$, then the density of the push forward $\phi_* f| \omega_X|$ with respect to $| \omega_Y |$ at a point $y$ is equal to $\int_{\phi ^{-1}(y)}f | \eta_y |$, for some top differential form $\eta_y$ on the fiber $\phi ^{-1}(y)$. Moreover, the forms $\eta_y$ vary in an algebraic manner with $y$---they are a section of the sheaf of relative differential top forms $\Omega_{X/Y}$. It follows that, in the case of smooth morphism, the push forward $\phi_* f| \omega_X|$ is again a Schwartz measure.

If $\phi$ is only \FRS, we show that the density is again given by integrating a certain top form on the smooth locus of the fiber. In order to show that these integrals converge, we use the following criterion:

\begin{propositionRoman} \lbl{prop:intro.integration.RS} Let $X$ be a variety over a non-archimedean local field $F$ of characteristic 0. Denote the smooth locus of $X$ by $X^{sm}$.
\begin{enumerate}
\item If $X$ has rational singularities, then, for any top differential form $\omega$ on $X^{sm}$ and any compact open subset $K \subset X(F)$, the integral $\int_{K \cap X^{sm}(F)}| \omega |$ converges.
\item Assume that $X$ is Cohen--Macaulay and that, for every finite field extension $F'/F$, every top differential form $\omega '$ on $X^{sm}$, and every compact open subset $K' \subset X(F')$, the integral $\int_{K' \cap X^{sm}(F')} | \omega '|$ converges. Then $X$ has rational singularities.
\end{enumerate}
\end{propositionRoman}

See \S\S \ref{subsec:integration.RS} for stronger statements.

Proposition \ref{prop:intro.integration.RS} and its proof do not imply that the density of $\phi_* f| \omega_X|$ is continuous, since, in general, we cannot find a simultaneous resolution of singularities for all fibers of $\phi$. This lack of simultaneous resolution of singularities is also the difficulty in proving Elkik's theorem. In fact, in view of Proposition \ref{prop:intro.integration.RS}, Theorem \ref{thm:push.forward} can be thought of as a quantitative version of Elkik's theorem.

The proof of the continuity in Theorem \ref{thm:push.forward} is done in two steps. In the first, we reduce the claim to the case where $Y$ is a curve. In order to do this, we show that the function
\[
y \mapsto \int_{\phi ^{-1} (y)^{sm}(F)} f | \eta_y |
\]
is constructible (in the sense of Model theory; see Appendix \ref{sec:cont.along.curves} for the definition), and then show that if a constructible function is continuous along all curves, then it is continuous. In the second step of the proof, we use embedded resolution of singularities as a substitute for a simultaneous resolution, and translate the question of continuity of push forward of Schwartz measures under general maps to the question of continuity of push forward of a measure of the form
\[
\left | x_1^{a_1} \cdots x_n^{a_n} dx _1 \wedge \cdots \wedge dx _n \right |
\]
under monomial maps
\[
\phi(x_1,\ldots,x_n)=x_1^{b_1} \cdots x_n^{b_n},
\]
where $x_1,\ldots,x_n$ are local coordinates. For general exponents $a_i,b_i$, the push forward need not be continuous. In our case, however, the assumption on rational singularity and a homological algebra argument imply an inequality on the exponents $a_i,b_i$, which implies that the push forward is continuous.

\subsection{Some Related Results}

\subsubsection{Representation Zeta Functions of FAb p-Adic Analytic Groups}

Let $G$ be a semi-simple algebraic group over a local non-archimedean field $F$ of characteristic 0 and residual characteristic $p$, and let $\Gamma$ be a compact open subgroup of $G(F)$. Jaikin--Zapirain proved in \cite{Ja} that there are natural numbers $n_i$ and rational functions $f_i\in \mathbb{Q}(x)$, for $i=1,\ldots,N$, such that
\[
\zeta_\Gamma(s)=\sum_1^N n_i^{-s} f_i(p^{-s}).
\]
Moreover, the denominators of the functions $f_i(x)$ have the form $\prod_j (1-p^{a_{i,j}}x^{b_{i,j}})$, where the numbers $a_{i,j}$ and $b_{i,j}$ are integers. This implies that $\zeta_\Gamma(s)$ has meromorphic continuation to the entire complex plane, that its poles have rational real parts, and that $R_n(\Gamma)=n^{\alpha(\Gamma) +o(1)}$, i.e., that $\log(R_n(\Gamma)) / \log n$ tends to $\alpha(\Gamma)$ when $n$ tends to infinity.

In \cite{LM}, the authors prove that $\alpha(\Gamma)\leq 3\dim\, G$. In \cite{LL}, the authors prove that $\alpha(\Gamma) \geq 1/15$. Theorem \ref{thm:intro.SLd.representation.growth} now implies that the sequence $d \mapsto \alpha(\SL_d(\mathbb{Z}_p))$ is bounded away from zero and infinity. It is still unknown whether this sequence has a limit.

If $\Gamma$ is either $\SL_2(O)$ or $\SL_3(O)$, where $O$ is the ring of integers of a local non-archimedean field of characteristic 0, the value $\alpha(\Gamma)$ were computed in \cite{Ja} and \cite{AKOV}. They are $\alpha(\SL_2(O))=1$ and $\alpha(\SL_3(O))=2/3$.

\subsubsection{Compact Lie Groups and Topological Quantum Field Theory}

Representation growth was also considered for compact simple Lie groups. In this case, much more is known. For example, Weyl's character formula implies that, if $L$ is a compact simple Lie group, then $\alpha(L)=\frac{\rk L}{| \Phi ^+ |}$, where $\rk L$ is the rank of $L$ and $|\Phi ^+|$ is the number of positive roots of $L$ (see \cite{LL}). In particular, $\alpha(L) \leq 1$ and $\alpha(L)$ tends to 0 as the dimension of $L$ tends to infinity. This stands in contrast to the uniform lower bound $1/15$ in the p-adic case.

The volume formula in Theorem \ref{thm:intro.volume.formula} is an analog of \cite[(4.72)]{Wi}, which deals with the case of a compact semisimple Lie group. Witten's result, as well as ours, are related to topological field theories. Explicitly, let $G$ be a group scheme over $\mathbb{Z}_p$ with semi-simple generic fiber. Consider the Dijkgraaf--Witten TQFTs $Z_{r}$ with gauge groups $G(\mathbb{Z} / p^r)$ and trivial Lagrangian (i.e., we choose the trivial cocycle in $H^2(B\left(G(\mathbb{Z} / p^r)\right),\mathbb{R} / \mathbb{Z})$); see \cite{FQ}. For any compact orientable surface $\Sigma$, we have $Z_r(\Sigma)=\zeta_{G(\mathbb{Z} / p^r)}(-\chi(\Sigma))$, where $\chi(\Sigma)$ is the Euler characteristic of $\Sigma$. If $\Sigma$ has genus greater than or equal to $\mathcal{B}(G \times_{\Spec \mathbb{Z}_p} \Spec \mathbb{Q}_p)$, then Theorem \ref{thm:intro.SLd.representation.growth} shows that the limit $\lim_{r \rightarrow \infty} Z_r(\Sigma)$ exists.

\subsubsection{Flatness of $\Phi_{G,n}$}

{Let $G$ be a group scheme over $\mathbb{Z}$ such that the generic fiber of $G$ is simple. Liebeck and Shalev (\cite{LS}) studied the limit of the sequence $\zeta_{G(\mathbb{Z} / p)}(s)$, where $p$ is a prime tending to infinity. }Namely, they showed that the limit $\lim_{p \rightarrow \infty}\zeta_{G(\mathbb{Z} / p)}(s)$ is equal to one for $s>\frac{\rk G}{| \Phi ^+ |}$. The proof of Corollary \ref{cor:Phi.flat} shows that, for any integer $n$, the existence of the limit
\begin{equation} \lbl{eq:limit.LS}
\lim_{p \rightarrow \infty} \zeta_{G(\mathbb{Z} / p)}(2n-2)
\end{equation}
is equivalent to the flatness of the map $\Phi_{G,n}$.

In this paper, we study the closely related limit
\begin{equation} \lbl{eq:limit.ours}
\lim_{r \rightarrow \infty}\zeta_{G(\mathbb{Z} / p^r)}(2n-2).
\end{equation}
Theorem \ref{thm:def.and.representation.growth} implies that the existence of the limit \eqref{eq:limit.ours} for any $p$ is equivalent to the map $\Phi_{G,n}$ being \FRS. Liebeck and Shalev study the limit \eqref{eq:limit.LS} (and, hence, prove that $\Phi_{G,n}$ is flat) using the description of the irreducible representations of finite groups of Lie type due to Deligne and Lusztig. No such description is known for the groups $G(\mathbb{Z} / p^r)$ for general $r$.

The flatness of $\Phi_{G,n}$ was also proved by J. Li (\cite{Li}). A different proof for the case $n \geq \mathcal{B}(G)$ also follows from our methods\footnote{This is a combination of Theorem \ref{thm:local.FRS} stating that $\Phi_{G,n}$ is \FRS\ at $(1,\ldots,1)$, and the proof of Theorem \ref{thm:def.and.representation.growth} where we show that if $\Phi_{G,n}$ is \FRS\ at $(1,\ldots,1)$, then it is \FRS. Unlike the rest of he proof of Theorem \ref{thm:def.and.representation.growth}, this part does not use \cite{LS}.}. This proof, like the one in \cite{Li} and unlike the proof in \cite{LS} works only in characteristic 0.

\subsubsection{Rational Singularities in Representation Theory}

Let $G$ be a reductive group over a field of characteristic 0, and let $\mathfrak{g}$ be its Lie algebra. Consider the quotient $\mathfrak{g}/G$ by the adjoint action. The quotient map $\pi_G:\mathfrak{g} \rightarrow \mathfrak{g}/G$ was shown to be flat in \cite{Ko}, and the fiber $\pi ^{-1}(0)$ (the nilpotent cone) was shown to have rational singularities in \cite{He}. A simple argument shows that these facts imply that all fibers of $\pi$ have rational singularities\footnote{For example, Elkik's theorem implies that all fibers in a neighborhood of $1$ have rational singularities, and one can use the actions of $\mathbb{G}_m$ on $\mathfrak{g}$ and $\mathfrak{g} / G$ to deduce that all fibers have rational singularities.}. Using Theorem \ref{thm:push.forward}, we get

\begin{corollaryRoman} Let $G$ is a reductive algebraic group over a local non-archimedean field of characteristic 0 with Lie algebra $\mathfrak{g}$. Denote the projection from $\mathfrak{g}$ to the categorical quotient $\mathfrak{g} / G$ by $\pi$. Let $m$ be a Schwartz measure on $\mathfrak{g}$, then the push forward $\pi_*m$ has continuous density.
\end{corollaryRoman}
This statement is probably known to experts, but we have been unable to find an exact reference.

Another example of rational singularities in representation theory is the following:

\begin{theorem*} [{\cite[Theorem 3.3]{Hi}}] Let $G$ be a reductive algebraic group over a local non-archimedean field of characteristic 0 with Lie algebra $\mathfrak{g}$, and let $\mathcal{O} \subset \mathfrak{g}$ be a nilpotent orbit. Then the normalization of the closure of $\mathcal{O}$ has rational singularities.
\end{theorem*}

Using Proposition \ref{prop:intro.integration.RS}, one can deduce the following theorem of Deligne and Ranga--Rao:

\begin{theorem*}[see \cite{RR}] Let $G$ be a reductive algebraic group over a local non-archimedean field of characteristic 0 with Lie algebra $\mathfrak{g}$, and let $\mathcal{O}$ be a nilpotent orbit. Let $\mu$ be a $G$-invariant measure on $\mathcal{O}$. Then, for each Schwartz function $f$ on $\mathfrak{g}$, the integral $\int_{\mathcal{O}} f \cdot \mu$ converges.
\end{theorem*}

\subsection{Future Work}

%

A sequel to this paper will deal with analogs of Theorems \ref{thm:intro.SLd.representation.growth} and \ref{thm:intro.abscissa.SLd} in the global case, for example, for groups of the form $G(\mathbb{Z})$.

\subsection{Structure of the Paper}

In Section \ref{sec:singularities.def}, we prove that the map $\Phi_{G,n}$ is \FRS\ at the point $(1,\ldots,1)$, assuming $G$ is classical and $n \geq \mathcal{B}(G)$. In Subsection \ref{subsec:methods}, we describe the degeneration method and give two examples of degenerations. The scheme of the proof is described in Subsection \ref{subsec:scheme}. In Subsection \ref{subsec:reduction.polygraph}, we reduce the claim to proving \FRS\ property of a map described by a combinatorial data, which we call a polygraph. The definition of this map is similar to Definition \ref{defn:symplectic.graph.variety}. In Section \ref{subsec:FRS.polygraph}, we translate the methods of Subsection \ref{subsec:methods} to polygraphs. In Subsection \ref{ssec:frs_tree}, we prove Theorem \ref{thm:FRS.tree}. In Subsections \ref{ssec:mult_FRS}--\ref{ssec:Comb.stat.sp}, we finish the proof for the classical groups, case by case.


In Section \ref{sec:push.forward}, we prove Theorem \ref{thm:push.forward}. Subsection \ref{subsec:measure.forms} describes the construction of measures out of differential forms. In Subsection \ref{subsec:strong.push.forward} we formulate a strong version of Theorem \ref{thm:push.forward}. In Subsection \ref{subsec:generalities} we prove some general results about push forwards of Schwartz measures. In Subsection \ref{subsec:integration.RS}, we  prove Proposition \ref{prop:intro.integration.RS}. In Subsections \ref{subsec:FRS.to.continuous} and \ref{subsec:continuous.to.FRS}, we prove a stronger version of Theorem \ref{thm:push.forward}. Part 1 is proved in Subsection \ref{subsec:FRS.to.continuous}, and part 2 is proved in Subsection \ref{subsec:continuous.to.FRS}.

In Section \ref{sec:rep.growth} we apply our previous results to representation growth and the geometry of deformation varieties. Subsection \ref{subsec:special.values} concerns the relation between representation growth and the map $\Phi_{G,n}$. We prove Proposition \ref{prop:Frobenius.Formula}, prove that the map $\Phi_{G,n}$ is always flat, prove Theorem \ref{thm:def.and.representation.growth}, and finish the proof of Theorem \ref{thm:intro.SLd.FRS}. In Subsection \ref{subsec:pos.char}, we apply our characteristic 0 results to representation growth in positive characteristic and prove Theorem \ref{thm:intro.pos.zero.local.zeta} (which implies Theorem \ref{thm:intro.abscissa.SLd}). In Subsection \ref{subsubsec:localization} we describe the Atiyah--Bott--Goldman form on the space of $\Gamma$-local systems, and prove Theorem \ref{thm:intro.volume.formula}.

In Appendix \ref{sec:cont.along.curves} we study integrals of Schwartz measures with parameters,  i.e. functions of the form
\[
y \mapsto \int_{\phi ^{-1} (y)^{sm}(F)} f | \eta_y |,
\]
where $\phi : X \rightarrow Y$ is \FRS. We show that such functions are continuous if and only if their restrictions to any curve in $Y$ is continuous. This result is used in Section \ref{sec:push.forward}. Our methods are taken from Model Theory, and, in particular, the theory of Motivic Integration.

In Appendix \ref{sec:alg.geom}, we summarize the definitions and facts of Algebraic Geometry that we use in this paper. Subsection \ref{ssec:flat} discusses flat maps. Subsection \ref{ssec:G.duality} contains a summary of results that we use from the theory of Grothendieck duality. Subsections \ref{ssec:LCI} and \ref{ssec:CM}-\ref{ssec:RS} discuss singularity theory.

In Appendix \ref{app:ex}, we provide some illustrations for the graphs obtained in \S\S\ref{ssec:mult_FRS}--\S\S\ref{ssec:Comb.stat.sp}.

\subsection{Conventions}
We will use the following conventions:
\begin{itemize}
\item  $k$ is a field of characteristic $0$.
\item  $F$ is a local field. Unless stated otherwise, it is non-archimedean and of characteristic $0$. We denote the ring of integers of $F$ by $O:=O_{F}$.
\item All the algebras that we consider are  commutative  and, unless stated otherwise, unital and finitely generated over some  base field (usually $k$ or $F$).
\item Unless stated otherwise, all the schemes that we consider are of finite type over the base field.
\item We use the term algebraic variety as a synonym for a reduced scheme.
\item A morphism of algebraic varieties or schemes means a morphism over the base field.
\item The smooth locus of an algebraic variety $X$ will be denoted by $X^{sm}$.
\item Given a field extension $k \subset F$ and a variety $X$ defined over $k$, we denote $X_F=X \times_{\Spec k}\Spec F$.
 \item For a $k$-scheme $X$, we denote its ring of regular functions by $k[X]:=O_X(X)$.  We use similar notation for $F$-schemes.
\item Unless stated otherwise, a point in algebraic variety will mean a point over the base field. For a variety $X$, we denote the set of such points by $X(k)$ or $X(F)$.
\item We will consider the (Hausdorff) analytic topology on $X(F)$ and the Zarizki topology on $X$ and $X(k)$. For a point $x \in X(k)$ the expression ``a neighborhood of $x$" will mean (depending on the context) an open subscheme $U \subset X$ such that $x \in U(k)$ or the set of $k$ points of such open subscheme.
\item A $p$-adic manifold is a Hausdorff space $X$ with a sheaf of functions that is locally isomorphic to the space $\mathbb{Z}_p^N$ together with the sheaf of functions that are locally given by convergent power series; see \cite{Se}. We will not use notions from rigid analytic geometry.
\item By point of an algebra $A$, we mean a point of its spectrum, i.e. a morphism from  $A$ to the base field.
\end{itemize}

\subsection{Acknowledgements}
We thank Karl Schwede, Angelo Vistoli, Sandor Kovacs, and Laurent Moret-Baily for answering questions related to rational singularities and algebraic geometry on MathOverflow as well as the site's administrators for the platform. We thank Vladimir Hinich for answering many questions about Grothendieck duality and rational singularities. We benefitted from conversations with Joseph Bernstein, Roman Bezrukavnikov, Alexander Braverman, Vladimir Drinfeld, Pavel Etingoff, Victor Ginzburg, David Kazhdan, Michael Larsen, Alex Lubotzky, Tony Pantev, and Yakov Varshavsky. We thank them all.

A.A. was partially supported by NSF grant DMS-1100943 and ISF grant 687/13; N.A. was partially supported by NSF grants DMS-0901638 and DMS-1303205. Both authors were also partially supported by BSF grant 2012247.

\section{Singularities of Deformation Varieties} \lbl{sec:singularities.def}
\setcounter{theorem}{0}
Our aim in this section is to prove the following theorem.

\begin{theorem} \lbl{thm:local.FRS} Let $d \geq 1$ be an integer and let $G$ be either $\SL_d,\SO_d$ or $\Sp_{2d}$. Let $n$ be an integer such that $n \geq \mathcal{B}(G)/2+1$. Then, the map $\Phi_{G,n}$ is \FRS\ at the point $(1,\ldots,1)$.
\end{theorem}

In Subsection \ref{subsec:methods} we introduce our method and in Subsection \ref{subsec:scheme} we describe the scheme of the proof. The rest of the section contains the proof itself. Throughout the section, we fix $G$ and $n$ as above and denote the Lie algebra of $G$ by $\fg$.

%

\subsection{Methods} \lbl{subsec:methods}

\subsubsection{Degeneration} \lbl{subsec:degeneration.method}

Given algebraic varieties $X,S$ and a morphism $\pi_X : X \rightarrow S$, we say that $X$ is an $S$-variety, and we call $\pi_X$ the structure map. We will denote the fiber $\pi_X ^{-1} (s)$ by $X_s$. We say that $X$ is a flat $S$-variety if $\pi_X$ is flat. A morphism $f:X \rightarrow Y$ between two $S$-varieties is called an $S$-morphism if $\pi_X=\pi_Y \circ f$. The restriction of $f$ to a morphism between $\pi_X ^{-1} ( s)$ and $\pi_Y ^{-1} (s)$ will be denoted by $f_s$.

Elkik's theorem (Theorem \ref{thm:elk}) has the following corollary:

\begin{corollary} \lbl{cor:elkik} Let $X$, $Y$ be flat $S$-varieties with structure maps $\pi_X,\pi_Y$, and let $f:X \rightarrow Y$ be an $S$-map. Assume that $Y$ is smooth. The set of points $x\in X(k)$ for which the fiber map $f_{\pi_X(x)}: X_{\pi_X(x)} \rightarrow Y _{\pi_X(x)}$ is \FRS\ at $x$ is open.
\end{corollary}

\begin{proof} By Theorem \ref{thm:flat.locus}, $f_{\pi_X(x)}$ is flat at $x$ if and only if $f$ is flat at $x$, and this is an open condition. Restricting to the open set on which $f$ is flat, we can assume, without loss of generality, that $f$ is flat. Elkik's theorem (Theorem \ref{thm:elk}), applied with $S=Y$, implies that the set of points $x$ at which $f_{\pi_X(x)} ^{-1} (f_{\pi_X(x)}(x))=f ^{-1} (f(x))$ has rational singularity is open.
\end{proof}

\begin{corollary}[Geometric Degenerarion]\label{cor:deg.geom} Let $X$ and $Y$ be flat $\mathbb{A} ^1$-varieties equipped with an action of $\mathbb{G}_m$ such that the structure maps intertwine the $\mathbb{G}_m$ action with the standard action of $\mathbb{G}_m$ on $\mathbb{A} ^1$. Let $\phi :X \rightarrow Y$ be an $\mathbb{A} ^1$-morphism that is also $\mathbb{G}_m$-equivariant, and Let $s:\A^1 \to X$ be a  $\mathbb G_m$-equivariant section of the structure map.
\RamiAlt{Let
$$\xymatrix{
X \ar@{->}^\phi[rr] \ar@{->}_p[dr]&& Y \ar@{->}^q[dl]\\
&\A^1 &
}
$$
Be a commutative diagram  of algebraic varieties equpt with an action of $\mathbb G_m$ such that  $p$ and $q$ are flat and.
 the action on $\A^1$ is the standard one. Let $s:\A^1\to X$  be an equivariant section. For a point $a \in \A^1$ we denote $X_a=p^{-1}(a),Y_a=q^{-1}(a)$ and  $\phi_a:= \phi|_{X_a}:X_a \to Y_a$.} Suppose that $\phi_0$ is \FRS\ at $s(0)$. Then $\phi_1$ is \FRS\ at $s(1)$
\end{corollary}
\begin{proof}
By Corollary \ref{cor:elkik}, the set of $t \in \A^1(k)$ such that $\phi_t$ is \FRS\ at $s(t)$ is open. In particular, the map $\phi_t$ is \FRS\ at $s(t)$, for some $t \neq 0$. Using the $\mathbb G_m$-action, $\phi_1$ is \FRS\ at $s(1)$.
\end{proof}

\begin{definition}
In the situation above, we call $\phi_0:X_0 \to Y_0$ a {\emph{degeneration}} of $\phi_1:X_1 \to Y_1$.
\end{definition}

Instead of talking about families with a $\mathbb{G}_m$ action, we will use the equivalent language of filtrations, which is more suitable for computations.

\begin{definition} \label{defn:filtration}
{Let $A$ be an algebra. By a \emph{filtration} on $A$ we mean}
an increasing sequence $\left( F^i A \right)_{i\in \mathbb{Z}}$ of subspaces of $A$ such that the following hold:
\begin{enumerate}
\item $F^iA \cdot F^jA \subset F^{i+j}A$.
\item $\cap F^iA=\left\{ 0 \right\}$.
\item $\cup F^iA=A$.
\end{enumerate}
\end{definition}

\begin{remark} One usually considers two kinds of filtrations, ascending and descending; both with non-negative indices. Definition \ref{defn:filtration} unites both kinds, since one can replace a descending filtration $\left( G_iA \right) _{i \geq 0}$ with the filtration
\[
F^iA=\left\{ \begin{matrix} G_{-i}A & i \leq 0 \\ A & i \geq 0 \end{matrix} \right. .
\]
\end{remark}

\begin{example}
On the field $k$, we will consider the standard filtration
$$F^i(k)=\begin{cases}0 & i<0 \\
k & i \geq 0 \\
\end{cases}$$
\end{example}

\begin{definition}
$ $

\begin{enumerate}
\item For a reduced algebra $A$ with a filtration $(F^iA)$, we define the {\emph{Rees algebra}} of $A$ to be the graded $k[t]$-algebra $\Rees(A):=\oplus t^{i} F^iA$. Note that, in general, $\Rees(A)$ is reduced, but need not be unital nor finitely generated.
\item For a filtration-preserving morphism $\phi:A \to B$ between reduced filtrated algebras, we define $\Rees(\phi):\Rees(A) \to \Rees(B)$ by $\Rees(\phi)(t^i a)=t^i \phi(a)$. The assignment $\Rees$ is a covariant functor from the category of reduced filtrated algebras to the category of reduced algebras.
\item We call a filtration on a reduced algebra $A$ {\emph{good}} if $\Rees(A)$ is unital and finitely generated.
\item The {\emph{Rees variety}} of a reduced algebra $A$ with good filtration is $\mathcal R(A):= \Spec(\Rees(A))$. Since the Rees algebra is graded, the Rees variety has a natural action of $\mathbb G_m$. Since the Rees algebra is a $k[t]$-algebra, we have a natural map $\mathcal R(A)\to \A^1$ which is $\mathbb{G}_m$-equivariant.
\item $\mathcal R$ gives rise to a contravariant functor from the category of reduced algebras with good filtration to the category of $\mathbb G_m$-equivariant maps $X \to \A^1$.
\end{enumerate}
\end{definition}

\begin{example} If $A$ is a reduced $\mathbb{Z}$-graded algebra, then the filtration induced by the grading is good. In this case, $\mathcal{R}(A)=\Spec(A) \times \mathbb{A} ^1$.
\end{example}

The following proposition is standard:
\begin{proposition}  [Rees Equivalence] \label{prop:eq.cat}
$ $
\begin{enumerate}
\item  \label{prop:eq.cat:1}
The functor $\mathcal{R}$ defines an (anti-)equivalence of categories between the category of reduced algebras with good filtrations and the category of diagrams $X \overset{\phi}{\to} \A^1$, where $X$ is an affine variety with a $\mathbb{G}_m$ action, and $\phi$ is flat and $\mathbb{G}_m$-equivariant.
\item The fiber at $1$ of $\mathcal R(A)$ is $\Spec (A)$
\item The fiber at $0$ of $\mathcal R(A)$ is $\Spec(\gr(A))$
\item $\mathcal R(k)=(\A^1\overset{Id}{\to} \A^1) $
\end{enumerate}
\end{proposition}

\begin{proof}
The only non-trivial part is assertion \ref{prop:eq.cat:1}. We construct an inverse functor for $\mathcal R$ as follows. Given $X \overset{\phi}{\to} \A^1$ as above, we consider the algebra $A':=O(X)$ with the grading $A'=\bigoplus A^i$ coming from the $\mathbb G_m$ action. Let $t$ be the coordinate of $\A^1$, and consider it as an element of $A'$. Since $\phi$ is $\mathbb G_m$-equivariant, multiplication by $t$ gives a map  $u_i:A^i \to A^{i+1}$.  Since $\phi$ is flat, $t$ is not zero divisor, so all $u_i$s are embeddings. Define $\mathcal R^{-1}(X)$ to be the direct limit of $\cdots \stackrel{u_{i-1}}{\rightarrow} A^i \stackrel{u_i}{\rightarrow} \cdots $, with the natural filtration. The action of $\mathcal{R} ^{-1}$ on morphisms is evident, and the verification that $\mathcal{R} ^{-1}$ is an inverse of $\mathcal{R}$ is straightforward.
 \end{proof}

\begin{remark}
The last proposition implies that a section as in Corollary \ref{cor:deg.geom} corresponds, under $\mathcal R ^{-1}$, to a map $\phi: A \to k$ such that $\phi(F^{-1}A)=0$. We call such maps stable points of $A$.
\end{remark}

The following is a consequence of \ref{cor:deg.geom}
\begin{corollary}[Degenerarion] \lbl{cor:degeneration} Let $A,B$ be rings with good filtrations, let $\varphi :A \rightarrow B$ is a filtration-preserving map, and let $p:A \to k$ be a stable point.
Suppose that $\gr(\varphi)$ is \FRS\ at $\gr({p})$. Then $\varphi$ is \FRS\ at ${p}$.
\end{corollary}

\subsubsection{Linearization} \lbl{sssec:linearization}
As a first application of the degeneration method (Corollary \ref{cor:degeneration}), suppose that $\phi:X \rightarrow Y$ is a map between two affine varieties, and let  $\phi^{*}:k[Y] \to k[X]$ be the corresponding map of algebras.

 Let $x\in X$. Denote the maximal ideal of $k[X]$ corresponding to $x$ by $\mathfrak{m}_{X,x}$, and the maximal ideal of $k[Y]$ corresponding to $\phi(x)$ by $\mathfrak{m}_{Y,\phi(y)}$. Choose a natural number $r$ such that
$\phi^{*}(\mathfrak{m}_{Y,f(y)}) \subset \mathfrak{m}_{X,x}^r$, and define the following filtrations:
\[
F^i k[X]= \left\{ \begin{matrix} k[X] & i \geq 0\\ \mathfrak{m}_{X,x} ^{-i}  & i<0 \end{matrix} \right.
\]
and
\[
F^i k[Y]= \left\{ \begin{matrix} k[Y] & i \geq 0\\ \left( \mathfrak{m}_{Y,\phi(y)}\right)  ^{\left \lceil \frac{-i}{r} \right \rceil}  & i<0 \end{matrix} \right. .
\]
By the definition of $r$, the map $\phi$ is filtration-preserving. The degenerations of $X$ and $Y$ are the tangent cones $C_x(X)$ and $C_{\phi(y)}Y$ of $X$ and $Y$, which are equal to $\Spec \gr(k[X])$ and $\Spec \gr(k[Y])$ respectively.

This leads us to the following notation

\begin{notation} Let $\phi:X \rightarrow Y$ be a map between two affine varieties, let $x\in X$ and let $r$ be such that $\phi^{*}(\mathfrak{m}_{Y,f(y)}) \subset \mathfrak{m}_{X,x}^r$. Define $D^{r-1}_x \phi:=\gr \phi:C_x(X) \to C_{f(x)}(Y)$. Note that $D^0$ is the usual differential.
\end{notation}

From the degeneration method (Corollary \ref{cor:degeneration}) we get

\begin{corollary} [Linearization] \lbl{cor:lin}
Let $\phi: X \to Y$ be a morphism of affine varieties, and let $x\in X(k)$. Suppose that $D^r_x \phi$ is  \FRS\ at $x$ for some $r$ for which $D^r_x$ is defined. Then $\phi$ is  \FRS\ at $x$.
\end{corollary}

In order to compute the operation $D^r_x$ for our case we will use the following obvious lemma and example
\begin{lemma}[Functoriality of $D^r$]\label{lem:func_Dr}
The operation $D^r_x$ is functorial in the following sense:
Let
$$\xymatrix{
 x\simpAr[r]{\mapsto} \simpAr[d]{\in} & y\simpAr[d]{\in}\simpAr[r]{\mapsto} & z\simpAr[d]{\in} \\
 X\ar@{->}[r]_{\pi} & Y \ar@{->}_{\tau}[r] & Z\\
}$$
be a commutative diagram of algebraic varieties (and points). Assume that $\pi^{*}(\mathfrak{m}_{Y,y}) \subset \mathfrak{m}_{X,x}^{s+1}$ and $\tau^{*}(\mathfrak{m}_{Z,z}) \subset \mathfrak{m}_{Y,y}^{r+1}$. Then $D^{r+s}_x(\tau \circ \pi)=D^s_y(\tau) \circ D^r_x(\pi)$
\end{lemma}

\begin{example}[Computation of $D^r$ for affine spaces]\label{exm:comp_Dr}
Let $U \subset \A^n$ be a Zariski open subset of an affine space that contains $0$. Let $p=(p_i)_{i=1\dots m}:U \to \A^m$. Assume that all the derivatives of $p$ of order $\leq r$ vanish. Let $q=(q_i)_{i=1\dots m}:\A^n \to \A^m$  be the $(r+1)$st Taylor expansion of $p$, meaning that $q_i$  are homogeneous polynomials of order $r$ and the order-$(r+1)$ derivatives of all $p_i-q_i$ vanish. Then $$D^r(p)=q,$$ under the standard identifications $$C_{0}(U)=T_0(U)\cong \A^n \text{ and }
C_{p(0)}(\A^m)=C_{0}(\A^m)=T_{0}(\A^m)\cong \A^m$$
\end{example}

\subsubsection{Elimination} \lbl{sssec:elimination}
The linearization method (Corollary \ref{cor:lin}) allows us  to reduce our problem to showing that some map between affine spaces is \FRS. For an affine space $\mathbb{A} ^I$, we will use filtrations coming from weights, i.e. vectors in $\mathbb{Z} ^I$. A weight $w\in \mathbb{Z} ^I$ gives a good filtration $F_w$ on $k[\mathbb{A} ^I]$ as follows: define the degree of a monomial $x_1^{a_1}\cdots x_I^{a_I}$ to be $\sum a_iw(i)$, and let $F_w^nk[\mathbb{A} ^I]$ be the span of all monomials of degree at most $n$. Note that this filtration depends on the choice of coordinates $x_i$ on $\mathbb{A} ^I$. Given a weight $w$ and a polynomial $f\in k[\mathbb{A} ^I]$, the symbol of $f$, denoted by $\sigma_w(f)$ is the sum of the monomials of $f$ with the highest $w$-degree. The degeneration method (Corollary \ref{cor:degeneration}) implies:

\begin{corollary}[elimination]\label{cor:elm} Let $I,J$ be finite sets and let $\psi=(\psi_j)_{j\in J}:\A^I \to \A^J$ be a morphism such that $\psi(0)=0$, let $w \in \Z^I$ be a weight, and let $\sigma_w(\psi)=(\sigma_w(\psi_j))_{j\in J}$ be the symbol of $\psi$. If $\sigma_w(\psi)$ is \FRS\ at $0$, then so is $\psi$.
\end{corollary}

\begin{proof} Let $x_i$ be the coordinates on $\A^I$ and $y_j$ be coordinates on $\A^J$. Let $w'\in \mathbb{Z} ^J$ be the weight defined by $w'(j)=\deg_w(\psi_j)$. It is easy to see that $F_w$ is a good filtration on $k[\mathbb{A} ^I]$, $F_{w'}$ is a good filtration on $k[\mathbb{A} ^J]$, $\psi$ is filtration preserving, and the point $0$ is stable. Finally, $\gr(\psi)$ is given by the symbol of $\psi$.
\end{proof}

\begin{remark} If $\psi: \mathbb{A} ^I \rightarrow \mathbb{A} ^J$ is as above, then $\psi ^{-1} (0)$ is an affine scheme whose coordinate ring has a filtration induced by the filtration on the domain. Denote the spectrum of the corresponding graded ring by $\gr( \psi ^{-1} (0))$. In general, we have $\gr ( \psi ^{-1}(0)) \subset (\gr \psi) ^{-1}(0)$ as subschemes of $\mathbb{A} ^I$, but the inclusion can be strict. As an example, consider the map $\psi(x,y)=(x^2,y+x^2)$, where the both coordinate rings $k[x,y]$ are given the filtration according to degree. In this case, $(\gr \psi) ^{-1} (0)$ is the zero locus of $x^2$, whereas $\gr (\psi ^{-1} (0))$ is the zero locus of $(x^2,y)$.

However, if $\gr \psi$ is \FRS, then $\gr ( \psi ^{-1}(0))$ is equal to $(\gr \psi) ^{-1}(0)$. Indeed, the flatness implies that $\dim \gr (\psi ^{-1} (0))=\# I -\#J=\dim (\gr \psi) ^{-1} (0)$. Since $(\gr \psi) ^{-1} (0)$ is conic, it is connected. Since $(\gr \psi) ^{-1} (0)$ has rational singularities, it must be irreducible, and, hence, equal to $\gr (\psi ^{-1} (0))$. We will use this fact in the sequel, since it is usually easier to compute $(\gr \psi) ^{-1} (0)$ than $\gr (\psi ^{-1} (0))$.
\end{remark}

\subsubsection{Explicit Resolution} \lbl{sssec:explicit.resolution}
After we use the above methods, we reduce the problem to proving \FRS\ property of rather simple map. We do this using the following

\begin{proposition} \label{prop:exp_res} Let $\phi :X \rightarrow Y$ be a map of irreducible smooth affine varieties, let $x\in X$, and denote $Z=\phi ^{-1} (\phi(x))$. Suppose \begin{enumerate}
\item \label{prop:exp_res:1} $\dim Z \leq \dim X -\dim Y$.
\item \label{prop:exp_res:2} $Z$ is reduced and smooth in co-dimension 1.
\item \label{prop:exp_res:4} There is a resolution of singularities $\pi : E \rightarrow Z$ such that $H^i(E,O_E)=0$ for all $i>0$.
\end{enumerate}
Then $\phi$ is \FRS\ at $x$.
\end{proposition}

\begin{proof} By Theorem \ref{thm:flat_dim}, the first condition implies that $\phi$ is flat at $x$. Next, we show that $x$ is a rational singularity of $Z$, i.e., that $O_Z \cong R \pi_*O_E$.
Let $p:Z\to \spec(k)$ be the projection to a point. Since $Z$ is affine, the functor  $p_*$ is exact. Thus $$H^{\bullet}(E)= R (p \circ \pi)_*O_E= R p_*(R\pi_*O_E)=p_*(R\pi_*O_E).$$
Together with the third condition, this implies that $R \pi_* O_E= \pi_* O_E$. It remains to show that $Z$ is normal. Since $\phi$ is flat, $Z$ is complete intersection, and hence Cohen Macaulay. Thus by Corollary \ref{cor:codim.2.in.CM} \eqref{cor:codim.2.in.CM:1}, the second condition implies that $X$ is normal.
\end{proof}

\subsection{Scheme of the proof of Theorem \ref{thm:local.FRS}} \label{subsec:scheme}
We first use the linearization method in order to pass to an analogous problem for the Lie algebra $\g$ of $G$ and the map $\Psi:\g^{2n}\to \g$ given by $$\Psi(X_1,Y_1,\ldots,X_n,Y_n)=[X_1,Y_1]+\ldots+[X_n,Y_n],$$(see \S\S\S\ref{sssec:reduction.Lie.algebra}). This problem has an additional symmetry. Namely, the map $\Psi:\g^{2n}\to \g$, can be interpreted as a map $\Psi:\g\otimes W \to \g$, where $W$ is a $2n$-dimensional symplectic space (see \S\S\S\ref{sssec:symplectic.polygraphs}).

Next, we choose a basis of $\g$ for which the structure matrix will be sparse. The components of the map $\Psi$ in this basis have small and simple set of terms. The map $\Psi$ is now essentially determined by some combinatorial data, which we call \pg, see Definition \ref{def:polygraph}.

To proceed, we assign weights to the basis elements of $\g$ and use the elimination method in order to leave only one term in each component of $\Psi$. We obtain a map $\Psi_{\Gamma,W}:W^V \to k^E$ which is described by  a graph $\Gamma=(V,E)$  and the symplectic space $W$, where each component of $\Psi_{\Gamma,W}$ is the symplectic pairing of the corresponding coordinates. We call such maps symplectic graph maps, see Definition \ref{def:graph}.

In order to simplify the map $\Psi_{\Gamma,W}$ further, we decompose $W$ to a direct sum of symplectic subspaces and assign different weights to the summands. Applying the elimination method, we get a new symplectic graph map, for which the graph has more vertices, but the same number of edges. The `price' we pay is decreasing the dimension of $W$ (see Corollary \ref{cor:col}).

We use this trick in order to break $\Gamma$ into a forest (see \S\S\S\ref{sssec:reduction.to.forest}), and then we show how to break any forest into a disjoint union of edges (see \S\S\S\ref{sssec:red_edge}). This process decreases the dimension of $W$ by  a constant factor and this is the reason why $n$ should be large enough.

Finally, we deal with a single edge by finding an explicit resolution, see proposition \ref{sssec:pf.edge}.

\subsection{Reduction to \FRS\ Property of a \Pg} \lbl{subsec:reduction.polygraph}
\subsubsection{Reduction to the Lie Algebra} \lbl{sssec:reduction.Lie.algebra}
Let $H$ be an arbitrary linear algebraic group with Lie algebra $\mathfrak{h}$. We apply Corollary \ref{cor:lin} to the map $\Phi_{H,n}: H^{2n} \rightarrow H$, the point $x=(1,\ldots,1)$, and $r=2$. The cones of $H^{2n}$ and of $H$ are equal to $\mathfrak{h} ^{2n}$ and $\mathfrak{h}$ respectively.

\begin{lemma} $\Psi_{\mathfrak{h},n}:= D^1_{(1,\dots,1)}\Phi_{H,n}$ is given by
\[
\Psi_{\mathfrak{h},n} (X_1,Y_1,\ldots,X_n,Y_n)=[X_1,Y_1]+\ldots+[X_n,Y_n].
\]
\end{lemma}

\begin{proof}
$ $
\begin{enumerate}
 \item [case 1] $H=\GL_N$ is the general linear group.\\ This case follows from the computation of $D^r$ for affine spaces (Example \ref{exm:comp_Dr}). Namely, Let $g_i, h_i \in \mathfrak{h}$ such that $1+g_i,1+h_i \in H$. We have
  $$\Phi_{H,n}(1+g_1,1+h_1,\dots,1+g_n,1+h_n)= [X_1,Y_1]+\ldots+[X_n,Y_n]+f(X_1,Y_1,\ldots,X_n,Y_n),$$ where the all the partial derivatives of $f$ of order $\leq 2$  vanish at $0$. So Example \ref{exm:comp_Dr} implies the assertion.
\item [case 2] the general case.\\ This case follows from functoriality of $D^{r}$ (Lemma \ref{lem:func_Dr}) and the previous case.
 Indeed, let $i:H \to \GL_N$ be an embedding of $H$ into a general linear group. We have the following commutative diagram:
$$\xymatrix{
 H^{2n}\ar@{->}[r]^{i^{2n}} \ar@{->}[d]_{\Phi_{H,n}}  & \GL_N^{2n} \ar@{->}[d]^{\Phi_{\GL_N,n}} \\
 H\ar@{->}[r]_{i}   & \GL_N\\
 }$$
Thus, by Lemma \ref{lem:func_Dr},  $$ D^0 i \circ D^1\Phi_{H,n}=D^1\Phi_{\GL_N,n} \circ D^0 (i^{2n})= D^1\Phi_{\GL_N,n} \circ (D^0 i)^{2n}.$$
Since $D^0i$ is the inclusion of $\mathfrak{h}$ in $\mathfrak{gl}_n$, the assertion is proved.

\end{enumerate}

\end{proof}

Thus, we reduce Theorem \ref{thm:local.FRS} to the following theorem:

\begin{theorem}
The map $\Psi_{\mathfrak{g},n}$ is \FRS\ at the point $0$.
\end{theorem}

\subsubsection{Symplectic Interpretation and \FRS\ \Pg s} \lbl{sssec:symplectic.polygraphs}

We give an alternative description of the map $\Psi_{\mathfrak{g},n}$. Let $B:=\{e_i\}_{i\in I}$ be a basis of $\mathfrak g$, and let $C=\{\alpha_j\}_{j \in J}$ be a coordinate system on $\mathfrak g$. We do not require any relation between $B$ and $C$. The coordinate system $C$ gives an identification of $\mathfrak g$ with $k^J$.

Let $W$ be a $2n$-dimensional symplectic space with a standard basis $p_1,q_1,\dots p_n,q_n$. Using the basis $B$, we get an identification of $\mathfrak g^{2n}\cong \mathfrak{g} \otimes W$ with $W^I$. Let $\left( a_{ijl}:=\alpha_l([e_i,e_j]) \right)_{i,j \in I, l \in J}$ be the structure coefficients of $\mathfrak g$ with respect to $B$ and $C$. Under these identifications, the map $\Psi_{\mathfrak{g},n}:W^I \to k^J$ is given by $(w_i)_{i \in I} \mapsto \left( \sum_{ij}a_{ijl} \langle w_i,w_j\rangle \right)_{l \in J}$. Choosing an ordering on $I$ allows us to express $\Psi_{\mathfrak{g},n}$ as $(w_i)_{i \in I} \mapsto \left( \sum_{i,j, \in I, i<j}2a_{ijl} \langle w_i,w_j\rangle \right)_{l \in J}$

\begin{definition}\label{def:polygraph}$ $
\begin{enumerate}
 \item A {\emph{\pg\ }}is a triple $(I,J,S)$ consisting of finite sets $I,J$ and a subset $S\subset I^{(2)} \times J$, where $I^{(2)}$ denote the set of subset of size $2$ of $I$.
\item Given a \pg\ $\Gamma= (I,J,S)$, an ordering $<$ of $I$, a function $a:S \to k^{\times}$, and a symplectic space $W$, we define a map  $\Psi_{\Gamma,<,a,W}:W^I \to \mathbb{A} ^J$
by $$(w_i)_{i \in I} \mapsto \left( \sum_{i,j, \in I, i<j} a(\{i,j\},l) \langle w_i,w_j\rangle \right)_{l \in J}.$$
\item Let $\Gamma$ be a \pg\ and $W$ be a symplectic space. We say that the pair $(\Gamma,W)$ is \FRS\ if, for all $<,a$ as above, the map $\Psi_{\Gamma,<,a,W}$ is \FRS\ at $0$
\end{enumerate}
\end{definition}

\begin{example} A basis $B=\{e_i\}_{i\in I}$ and a coordinate system $C=\{ \alpha_j\}_{j\in J}$ on $\mathfrak g$ give us a \pg\ $\Gamma_{B,C}=(I,J,S)$ by setting $S=\{(\{i,j\},l)|a_{ijl} \neq 0\}$, where  $\{a_{ijl}\}_{i,j\in I,l \in J}$ are the structure coefficients of $\g$ with respect to $B$ and $C$.
\end{example}

Thus, Theorem \ref{thm:local.FRS} will follow from the following
\begin{proposition}[The combinatorial statement]\label{prop:mult_FRS}
Let $W$ be a $2n$-dimensional symplectic space. Then there exists a basis $B$ and a coordinate system $C$ on $\g$ such that $(\Gamma_{B,C},W)$ is \FRS.
\end{proposition}
We will prove this proposition in \S\S\ref{ssec:mult_FRS} and \S\S\ref{ssec:frs_tree} after some preparation in the next subsection.

\subsection{How to Prove \FRS\ for \Pg s} \label{subsec:FRS.polygraph}
\subsubsection{ \Pg s} \lbl{sssec:polygraphs}
For technical reasons we slightly generalize Definition \ref{def:polygraph}:

\begin{definition}
$ $
\begin{enumerate}
\item Let $\Gamma=(I,J,S)$ be a \pg. A \term{symplectic assignment}{SmpAsg}\ for $\Gamma$
 is an assignment $W$ of a symplectic vector space $W(i)$ for each vertex $i\in I$ such that $W(i)=W(k)$ whenever $(\{i,k\},j)\in S$ for some $j\in J$.
 \item We define the map $\Psi_{\Gamma,<,a,W}:\bigoplus_{i\in I}W(i) \to \mathbb{A} ^J$  and the notion of \FRS\ for the pair $(\Gamma,W)$ in the same way as in Definition \ref{def:polygraph}.
\end{enumerate}
\end{definition}

Our main tool for proving the \FRS\ property of \pg s is the elimination method (Corollary \ref{cor:elm}). In order to apply it, we study two kinds of modifications of pairs $(\Gamma,W)$.

\begin{definition} Let $\Gamma:= (I,J,S)$ be a \pg. We call a vector $w\in \mathbb{Z} ^I$ an $I$-weight. An $I$-weight $w$ induces an $I^{(2)}$-weight $\tilde w:I^{(2)}\to\Z$ by $\tilde w(\{i,j\})=w(i)+w(j)$. We define $\gr_w \Gamma:=(I,J,\gr_w S)$ by
$$\gr_w S=\{(\alpha,l) \in S|\  \forall (\beta,l) \in S \text{ we have } \tilde w(\alpha) \geq \tilde w(\beta)\}.$$
\end{definition}

\begin{remark} Practically, we often choose the values of $w$ to be exponent of a large enough integer. In this case, $\tilde w(\{i,j\})$ behaves like the maximum of $w(i)$ and $w(j)$.
\end{remark}

\begin{corollary}[Elimination for polygraphs] \label{cor:pg_elm}
Let $\Gamma:= (I,J,S)$ be a \pg\  and $W$ be {a \SmpAsg\ for it.}
\begin{enumerate}
\item Let $w\in\Z^I$ be a weight. If $(\gr_w \Gamma ,W)$ is \FRS, then so is $(\Gamma ,W)$.
\item Let $W'$ {be another \SmpAsg\ such that $W'(i) \subset W(i)$ for all vertices $i$}. If $(\Gamma ,W')$ is \FRS\, then so is $(\Gamma ,W)$
\end{enumerate}
\end{corollary}

\begin{proof} The first part follows from the elimination method (Corollary \ref{cor:elm}). For the second part, choose  a {\SmpAsg\ $W''$ such that $W'(i) \oplus W''(i)=W(i)$ for all $i$}, put weight 1 on $(W')^I$ and 0 on $(W'')^I$, and apply the elimination method.
\end{proof}

The following Lemma is obvious but important

\begin{lemma}[Level splitting]\label{lem:pg_lev_sp}
Let $\Gamma:= (I,J,S)$ be a \pg, $M$ be a finite set. Let $\mathcal{L}(\Gamma):=(I \times M,J,S \times M)$, where we consider the embedding  $S \times M \hookrightarrow (I\times M)^{(2)} \times J $ given by $$((\{x,y\},j),m) \mapsto (\{(x,m),(y,m)\},j).$$
{Let $W_{i}$, $i\in M$, be symplectic vector spaces and denote $W=\bigoplus W_i$. Let $\cW$ be a symplectic assignment for $\mathcal L(\Gamma)$ defined by $\cW((i,m))=W_m.$}

{If $(\mathcal{L}(\Gamma),\cW)$ is \FRS, then so is $(\Gamma,W)$.}
\end{lemma}
\begin{remark}
We think of the procedure $\Gamma \mapsto \mathcal{L}(\Gamma)$ as splitting of $\Gamma$ into different levels that are indexed by $M$. We duplicate each vertex and edge to $M$ levels but do not change the set $J$.
\end{remark}

\subsubsection{Graphs} \lbl{sssec:graphs}

We will use the above methods in order to degenerate the \pg\ $\Gamma_{B,C}$ into a simpler \pg\ which is, in fact, induced from a graph. Let us describe how to construct a \pg\ from a graph, and how the results above translate to graphs.

\begin{definition}\label{def:graph}
 Let $\Gamma=(V,E)$ be a graph. 
\begin{enumerate}
\item
Set $\mathcal{P}(\Gamma):=(V,E,\Delta E)$ where  where $\Delta E \subset V^{(2)}\ \times E$ is the diagonal.
\item
We will say that $(\Gamma,W)$ is \FRS\ if $(\mathcal{P}(\Gamma),W)$ is \FRS.
\item
Note that the (isomorphism class of the) map
$\Psi_{\Gamma,<,a,W}$ does not depend on $<$ and $a$.
We will refer to $\Psi_{\Gamma,W}:=\Psi_{\mathcal{P}(\Gamma),<,a,W}$ as the {\emph{graph map}} of $(\Gamma,W)$ and to $X_{\Gamma,W}:=\Psi_{\Gamma,W}^{-1}(0)$ as the {\emph{graph variety}} of $(\Gamma,W)$.
\end{enumerate}
\end{definition}

Given a graph $\Gamma$, a symplectic space $W$, and a finite set $M$, consider the level splitting $(\mathcal{L}(\mathcal{P} (\Gamma)),W)$. Given a weight $w$ for $\mathcal{L}(\mathcal{P} (\Gamma))$ (i.e., a function $w: V \rightarrow \mathbb{Z} ^M$), we can degenerate the graph map as in Corollary \ref{cor:pg_elm}. The following describes the result of such a procedure:

\begin{definition}\label{def:col} Let $\Gamma=(V,E)$ be a graph, let $M$ be a finite set, and let $w:V \rightarrow \mathbb{Z} ^M$. \begin{enumerate}
\item Define $\tilde w:V^{(2)} \to \Z^{M}$ by $\tilde w(\{v_1,v_2\}):= w(v_1)+w(v_2)$.
\item Suppose that, for any edge $\alpha \in E$, the tuple $\tilde w(\alpha) \in \Z^M$ has a unique maximum. Define
$$\gr_{w,i} E=\{\alpha\in E |\forall j \in M \text{ we have }(\tilde w(\alpha))_i \geq (\tilde w(\alpha))_j\} \subset E.$$
\item Set $\gr_{w,i}\Gamma:=(V,\gr_{w,i}E)$
\end{enumerate}
\end{definition}

Using level splitting (lemma \ref{lem:pg_lev_sp}) and elimination for \pg s (Corollary \ref{cor:pg_elm}), we get the following tool for proving \FRS\ property of graphs
\begin{corollary}[Coloring]\label{cor:col} Let $\Gamma:= (V,E)$ be a graph,  and $M$ be a finite set. Let $w:V \to \Z^{M}$ be a function such that, for any edge $\alpha \in E$, the tuple $\tilde w(\alpha)$ has a unique maximum.
{ Let $W_{i},$ for $i\in M$, be  symplectic spaces and $W=\bigoplus W_i.$

Suppose that $(\gr_{w,i}\Gamma,W_{i})$ is \FRS. Then so is $(\Gamma,W)$}
\end{corollary}

\begin{remark}\label{rem:col}
We think of the decomposition $E= \bigcup \gr_{w,i}E$ as a coloring of $\Gamma$.
The {graphs $\gr_{w,i}\Gamma$ are obtained by considering only one color}.
\end{remark}


\subsection{\FRS\ Trees}
\label{ssec:frs_tree}
In this subsection we prove the following theorem:
\begin{thm}\label{prop:frs_tree} Let $T=(V,E)$ be a tree with maximal degree $d$ and $W$ be a (non-zero) symplectic space of dimension at least $4(d-1)$. Then $(T,W)$ is \FRS.
\end{thm}

\subsubsection{Reduction to an Edge} \label{sssec:red_edge}
{First, we apply the coloring method (Corollary \ref{cor:col}) and reduce the theorem to the case where $T$ consists of one edge.}

Assume that $T$ has more then one edge. Choose a vertex $v_{0}$ of $T$ which does not have a maximal degree (for example, we can choose $v_0$ to be a leaf of $T$). Let $\delta:V \to \Z_{\geq 0}$ be the distance function from $v_0$, and let $M:=\Z/2\Z$. Define $w:V \to \Z^{M}$ by
 $$w(v)(m)=
 (1+(-1)^{\delta(v)+m}){\delta(v)} =\begin{cases} 2 \delta(v) & \text{if } \delta(v) \in m  \\
0 & \text{if } \delta (v)  \notin m \\
\end{cases}$$

The following lemma is obvious.

\begin{lemma}
The graph $\gr_{w,m}T$ is a union of isolated vertices and the graph
$$ \bigsqcup_{v\in V,\delta(v)\notin m}T_{v},$$ where $T_{v}=(V_v,E_v)$ is the full subgraph of $T$ that consist of $v$ and its children (i.e. vertices $v_i$ of $T$ which are connected with $v$ and satisfy $\delta(v_i)=\delta(v)+1$).
\end{lemma}

Applying the coloring method (Corollary \ref{cor:col}) and the last lemma, Theorem \ref{prop:frs_tree} follows form the claim that, for any symplectic space $W$ of dimension $\geq 2m$, the pair $(T_v,W)$ is \FRS. Here $m$ is the number of children of $v$ which is evidently  not larger then $d-1$.
To prove this claim, define $w':V_v \to \Z^m$, by $w(v)=0,(w(v_i))_j=\delta_{ij}.$
 It is easy to see that $\gr_{w'} T_v$ is a disjoint union of isolated vertices and isolated edges. Thus we reduce the claim to the case where $T$ is an edge, which is the following proposition:

\begin{proposition} \label{prop:edge}
Let $W$ be a symplectic plane. The symplectic form $\omega:W \times W\to \A^1$ is \FRS\ at 0.
\end{proposition}
This is rather standard proposition; for completeness we include a proof in the next subsubsection.

\subsubsection{Proof for an Edge} \lbl{sssec:pf.edge}

The proof is based on proposition \ref{prop:exp_res}.
Let $$Z:=\omega^{-1}(0)=\{x,y \mid \text{ x is parallel to }y\}.$$
$\dim Z=3 \leq 4-1=\dim(W\times W)-\dim \mathbb{A} ^1$, so assumption  \eqref{prop:exp_res:1} of Proposition \ref{prop:exp_res} holds.

The map $\omega$ is regular on $W \times W \smallsetminus (W \times \left\{ 0 \right\} \cup \left\{ 0 \right\} \times W)$, so $Z\smallsetminus  (W \times \left\{ 0 \right\} \cup \left\{ 0 \right\} \times W)$ is reduced. Since it is Zariski dense in $Z$, we get that $Z$ is reduced.

Let $\mathcal T=O(-1)$ be the tautological bundle of $\mathbb P^1$, let $\mathcal E=\mathcal T\oplus \mathcal T$ and let $E$ be the total space of $\mathcal E$. We have a natural resolution of singularities $\pi:E\to Z$ which is an isomorphism outside $(0,0)\in Z$. In particular, $Z$ is smooth outside a codimension 2 subvariety. This proves assumption   \eqref{prop:exp_res:2} of Proposition \ref{prop:exp_res}. \

In order to prove assumption   \eqref{prop:exp_res:4} of Proposition \ref{prop:exp_res} we will use the following lemma
\begin{lemma}
Let $\mathcal E$ vector bundle over a variety $X$ and $E$ be its total space. Then the following are equivalent
\begin{enumerate}
 \item $H^i(X,\Sym^j(\mathcal E^*))=0,\forall i>0,j\geq 0$.
 \item $H^i(E,O_E)=0, \forall i>0$.
\end{enumerate}
\end{lemma}
\begin{proof}
Let $\pi_1:E\to X$ and let $\pi_2:X\to {\spec(k)}$ be the projection {to} the point.
Since the map $\pi_1$ is affine, the functor $(\pi_1)_{*}$ is exact. Thus, we have $R((\pi_1)_{*})=(\pi_1)_{*}.$  Therefore,

\begin{multline*}
H^ \bullet (E,O_E) = \left( R(\pi_2 \circ \pi_1)_{*}(O_{E}) \right) \cong \left( R(\pi_2)_{*}\left( R(\pi_1)_{*} (O_{E})\right) \right) \cong \left( R(\pi_2)_{*} ( (\pi_1)_{*} (O_{E})) \right) \cong \\ \cong \left( R(\pi_2)_{*}\left(\bigoplus_{j \geq 0} \Sym^j(\mathcal E^*)\right) \right) \cong \bigoplus_{j \geq 0} \left( R(\pi_2)_{*}\left( \Sym^j(\mathcal E^*)\right) \right) = \bigoplus_{j \geq 0}H^ \bullet (X,\Sym^j(\mathcal E^*)).
\end{multline*}
Here we identify the derived category of sheaves over a point with the category of graded vector spaces.

This clearly implies the assertion
\end{proof}

Assumption  \eqref{prop:exp_res:4} follows now from the fact that $H^i(\mathbb P^1,O(j))=0,\forall i>0,j\geq 0$.

\subsection{Proof of the Combinatorial Statement (Proposition \ref{prop:mult_FRS}) for $G=\SL_d$.}\label{ssec:mult_FRS}

\subsubsection{Coordinates on $\g$} \lbl{sssec:coordinates}

Let $L=\{1,\dots, d\}$ and $I=J=L \times L\smallsetminus \{(d,d)\}$. Let $e'_{(i,j)} \in \mathfrak{gl}_n$ be the matrix whose $(i,j)$th entry is equal to one and its other entries are equal to zero, and let $e_{(i,j)} = e'_{(i,j)}-\frac{\tr e'_{(i,j)}}{d} Id$. The set $B=\{e_{(i,j)}\}_{(i,j)\in I}$ is a basis of $\g$.
For $(i,j)\in J$, let $\alpha_{(i,j)}$ be the functional $\alpha_{(i,j)}(X)=X_{i,j}$. The set $C=\{\alpha_{(i,j)}\}_{(i,j)\in J}$ is a coordinate system on $\mathfrak{g}$.

Let $\Gamma_0=(I,J,S_0):=\Gamma_{B,C}$. It is easy to see that
$$S_0= \{(\{(i,j),(j,{l})\},(i,{l})) \in I^{(2)} \times J\}.$$

\subsubsection{Reduction to a Graph} \lbl{sssec:reduction.to.graph}

Define $w_0:I \to \Z$ by $w_{0}((i,j))=-3^{|i-j|}$. Let  $\Gamma_1=(I,J,S_1):=\gr_{w_{0}}(\Gamma_0)$. It is easy to see that the set $S_1$ consists of all triples $(\{(i,j),(j,{l})\},(i,{l}))$ such that $(i,{l}) \in J$ {is arbitrary} and $j$ is as close as possible to the average of $i$ and ${l}$. Namely,
\begin{multline*}
S_1= \left\{(\{(i,j),(j,{l})\},(i,{l})) \in  I^{(2)} \times J|\left |j-\frac{i+{l}}{2} \right | <1+\delta_{i,{l}} \right \}=\\
=\left\{(\{(i,j),(j,{l})\},(i,{l})) \in  I^{(2)} \times J| i\neq {l} \text{ and }\left |j-\frac{i+{l}}{2} \right | <1, \text{ or } i={l} \text{ and }\left |j-i\right|=1 \right \}.
\end{multline*}
Define $w_1:I \to \Z_{\geq 0}$ by $w_{1}((i,j))={i}$. Let  $\Gamma_2:=(I,J,S_2):=\gr_{w_{1}}(\Gamma_1)$. It is easy to see that
$$S_2= \left\{(\{(i,j),(j,{l})\},(i,{l})) \in  I^{(2)} \times J |\,   j=  \left \lceil \frac{i+{l}}{2}\right \rceil +\delta_{i,{l}} \right \}.$$
Note that $\Gamma_2$ is the \pg\ attached to the graph $\Gamma_3=(I,E)$, where
$$E= \left\{\{(i,j),(j,{l})\} \in I^{(2)}|\,   j=  \left \lceil \frac{i+{l}}{2}\right \rceil +\delta_{i,{l}} \right \}.$$
For the convenience of the reader we provide a picture of this graph for the case {d=8 in appendix \ref{app:ex}}.

By the elimination method for \pg s (Corollary \ref{cor:pg_elm}), we reduce Proposition \ref{prop:mult_FRS} to the following proposition
\begin{proposition}\label{prop:gr_FRS} Let $W$ be a symplectic space of dimension $\geq 24$.
Than $(\Gamma_3,W)$ is \FRS.
\end{proposition}

\subsubsection{Reduction to a Forest} \lbl{sssec:reduction.to.forest}

Let $M=\Z/3\Z$ and let $w_3:I \to \Z_{\geq 0}^M$ be the function $$w_3((i,j))(m)=
\begin{cases}5^{|i-j|} & \text{ if }m\equiv i-j\textrm{ (mod 3)}\\
3 \cdot 5^{|i-j|-1}  & \text{ if }m\equiv (i-j-sign(i-j+1/2))\textrm{ (mod 3)}\\
0 & \text{ if }m\equiv (i-j+sign(i-j+1/2))\textrm{ (mod 3)}\\
\end{cases}$$

We provide an illustration of $w_{3}$ and corresponding coloring of $\Gamma_3$ for the {cases $d=6$ and $d=8$, in Appendix \ref{app:ex}}.

Applying the coloring method (Corollary \ref{cor:col}), Proposition \ref{prop:gr_FRS} follows now from the \FRS\ property of trees (Theoerm \ref{prop:frs_tree})  and  the following simple lemma.
\begin{lemma}
{For all $m \in M$, the graph $\Gamma_{4}^{m}=(V^{m},E^{m}):=\gr_{w_{3},m}(\Gamma_3)$ is a forest with maximal degree $\leq 3$.}
\end{lemma}
\begin{proof} We first give an intuitive explanation based on the pictures in Appendix \ref{app:ex}. As one can see, each connected component of $\Gamma_{4}^{m}$ is supported on at most 2 adjacent diagonals (except for $n-1$ separated intervals around the main diagonal). Each of those component is a `comb', i.e. an interval (in one diagonal) with some leaves (from the other diagonal) attached to some of its vertices (not more than one leaf for each vertex of the interval). Such a `comb' is evidently a tree of maximal degree $\leq 3$.

We repeat the above considerations {rigorously} for general $n$.
Decompose $\Gamma_4^m$ in the following way: Decompose $I=\bigcup _{l=-n,\dots, n}I_l$, where $$I_l=\{(i,j)\in I | i-j=l\}.$$
 Define  $\Delta_l=(V_l,E_l) \subset \Gamma_3$, by

 $$V_l=\begin{cases}I_l \cup I_{l+1} & l<0 \\
 I_l \cup I_{l-1} & l>0 \\
 I_1 \cup I_{-1} & l=0 \\
\end{cases}$$

and

 $$E_l=\begin{cases}  E \cap (I_l^{(2)} \cup I_l \times I_{l+1} ) & l<0 \\
E \cap (I_l^{(2)} \cup I_l \times I_{l-1} ) & l>0\\
 E\cap (I_1 \times I_{-1}) & l=0 \\
\end{cases}.$$
Here we consider the product of 2 disjoint subsets of $I$ as a subset of $I^{(2)}.$

It is easy to see that $\Gamma_{4}^{m}$ is a union of isolated vertices and the graph $$\bigsqcup_{l \equiv m \textrm{ (mod 3)}}\Delta_l,$$ so it is enough to show that  $\Delta_l$ are forests of maximal degree $3$. The case $l=0$ is obvious since $\Delta_0$ is a disjoint union of edges. For the other cases, the degree estimate is also easy.  The forest property follows from the facts that the restriction of $\Delta_l$ to $V_l$ is a union of segments and that all other vertices have degree at most $1$.
\end{proof}

This concludes the proof of Theorem \ref{thm:local.FRS} modulo Theorem \ref{prop:frs_tree}.

{
\subsection{Proof of the Combinatorial Statement (Proposition \ref{prop:mult_FRS}){ for $G=\SO_d$}.}\label{ssec:Comb.stat.so}

The proof is similar to the proof for the $\SL_d$ case. We give here the main steps and, as before, we give a picture in Appendix \ref{app:ex}.
\begin{itemize}
\item  Coordinates on $\g$:
\begin{itemize}
\item Let $L=\{1,\dots, d\}$ be as before.
\item Let $I=J=L^{(2)}$.
\item For any $i>j\in L$, let $e_{\{i,j\}}:=e'_{i,j}-e'_{j,i}\in \mathfrak{o}_n$, where $e'_{i,j}$ is as before.
\item The set $B=\{e_{s}\}_{s\in I}$ is a basis of $\g$.
\item Let $C:=\{\alpha_{s}\}_{s\in J}$ be the dual basis of $\mathfrak{g}^*$.
\item Let $\Gamma_0=(I,J,S_0):=\Gamma_{B,C}$. We have $$S_0= \{(\{s_{1},s_{2}\},s_{1} \triangle s_{2}) \mid |s_1 \cap s_2|=1\}\subset I^{(2)} \times J,$$
where $\triangle$ is the symmetric difference.
\end{itemize}
\item Reduction to a Graph
\begin{itemize}
\item Define a weighting $w_{0}:I \to \Z$ by $$w_{0}(\{i,j\}))=-3^{|i-j|}.$$
\item  Let $\Gamma_1=(I,J,S_1):=\gr_{w_{0}}(\Gamma_0)$. Then
$$
S_1= \left\{(\{\{i,j\},\{j,{l}\}\},\{i,{l}\}) \in  I^{(2)} \times J|\left |j-\frac{i+{l}}{2} \right | <1+\delta_{|i-l|,1}) \right \}.
$$
\item Define a weighting $w_{1}:I \to \Z$ by $w_{1}(\{i,j\}))=\max \left\{ i,j \right\}$.
\item  Let $\Gamma_2=(I,J,S_2):=\gr_{w_{1}}(\Gamma_1)$. We get that $\Gamma_2$ is the \pg\ attached to the graph $\Gamma_3=(I,E)$, where
\[
E=\left\{ \left\{ \left\{i,j\right\} , \left\{j,l \right\} \right\} \in I^{(2)} \mid j= \left\{ \begin{matrix} d-2 & \left\{ i,l \right\} = \left\{ d,d-1 \right\} \\ \max \left\{ i,l \right\}+1 & |i-l|=1, \left\{ i,l \right\} \neq \left\{ d,d-1 \right\} \\ \left \lceil \frac{i+{l}}{2}\right \rceil & |i-l|>1 \end{matrix} \right. \begin{matrix} \\ \\ \\ \\ \\ \end{matrix}\right\} .
\]

\end{itemize}
\item Reduction to a Forest:
\begin{itemize}

\item  Let $M:=\{1,2,3\}$ and let $w_3:I \to \Z_{\geq 0}^M$ be the function
{
$$w_3(\{i,j\})=
\begin{cases}(0,3^{d}+2j,\delta_{i,d-2}\cdot 2 \cdot 3^{d-1}) & \text{ if } i=j-1\\
(3^{d}+2j-1,0,\delta_{i,d-2}\cdot 2 \cdot 3^{d-1}) & \text{ if } i=j-2\\
(0,0,3^{d-|i-j|})  & \text{ if }i<j-2 \text{ and } i-j \text{ is odd}\\
(3^{d-|i-j|},0,0)  & \text{ if }i<j-2 \text{ and } i-j \text{ is even}\\
\end{cases}$$
\item It is easy to see that {$\Gamma_{4}^m:=\gr_{w_{3}}^m(\Gamma_3)$} is a forest of maximal degree $\leq 3$.
}
\end{itemize}

\nir{We draw the graph $\Gamma_3$, the weighting $w_3$, and the graphs $\Gamma_4^m$ for $d=8$ in Appendix \ref{app:ex}.}
\end{itemize}

}
\subsection{Proof of the Combinatorial Statement (Proposition \ref{prop:mult_FRS}){ for $G=\Sp_{2d}$}.}\label{ssec:Comb.stat.sp}

\nir{Again, we only give a sketch of the argument.} We have $$\g=\mathfrak{sp}_{2d}=
\left\{\begin{pmatrix}
A & B \\
C & -A^t \\
\end{pmatrix} \in \Mat_{2d} \mid B=B^t,C=C^t\right\}.$$
\begin{itemize}
\item  Coordinates on $\g$:
\begin{itemize}
\item For a set $F$, we denote the collection of multisets of size 2 of $F$ by $F^{[2]}$. We have that $\left| F^{[2]} \right| = \binom{|F|+1}{2}$.
\item Let $L=\{1,\dots d\}$, and let $e'_{i,j}$ be the standard basis of $\Mat_{d}$,  as before. Let $I_0=L \times L\smallsetminus \{(d,d)\}$, let $I_{1}=\{(d,d)\}$, let $I_2=L^{[2]}\times \{-1\}$, and let $I_3=L^{[2]}\times \{1\}$.
\item Let $I:=J:=\bigcup_{n=0}^3 I_{n}$ and $I_{2,3}=I_{2} \cup I_3$.
\item For all $i,j\in L,$ let:
\begin{itemize}
\item $e^0_{(i,j)}={e'}_{(i,j)} - \frac{\tr e'_{(i,j)}}{d} I,$ if $(i,j) \in I_0$ and $e^0_{(d,d)}=I$.

\item $e_{(i,j)}:=
 \begin{pmatrix}
e^0_{(i,j)} & 0 \\
0 & -(e^0_{(i,j)})^t \\
\end{pmatrix}$
\item $e_{([i,j],-1)}:=
 \begin{pmatrix}
0 & e'_{(i,j)}+e'_{(j,i)} \\
0 & 0 \\
\end{pmatrix}$, where $[i,j]$ is the multiset consisting of $i$ and $j$.
\item $e_{([i,j],1)}:=e_{([i,j],-1)}^{t}$

\end{itemize}
\item The set $B=\{e_{s}\}_{s\in I}$ is a basis of $\g$.
\item
Let
$\{{\alpha'}_{s}\}_{s\in J}$ be the dual basis of $\mathfrak{g}^*$.
Define $\alpha_{(i,j)} \in \g^*$  by $$\alpha_{(i,j)}\left( \begin{pmatrix}
A & B \\
C & -A^{t} \\
\end{pmatrix} \right) =A_{i,j},$$ for $(i,j)\in I_0$  and
 $$\alpha_{(d,d)}\left(  \begin{pmatrix}
A & B \\
C & -A^{t} \\
\end{pmatrix}\right) =\tr(A).$$
Define also $\alpha_{s}:=\alpha'_{s}$ for $s \in I_{2,3}$.
The set $C:=\{\alpha_{s}\}_{s\in J}$ is a co-ordinate system on $\g.$
\item Let $\Gamma_0=(I,J,S_0):=\Gamma_{B,C}$. We have
$$S_0= \bigcup_{i=1}^6 S_0^i,$$  where
\begin{itemize}
\item $S_0^1 \subset I_0^{(2)} \times I_0$ is as $S_0$ in the type $A$ case.
\item $S_0^2=\{(\{(s,1),(s,-1)\},(d,d))|s \in L^{[2]}\}$.
\item $S_0^3=\{(\{(i,j),([j,l],-1)\},([i,l],-1)) \mid i,j,l \in L,(i,j) \in I_0\}$.
\item $S_0^4=\{(\{(i,j),([i,l],1)\},([j,l],1)) \mid i,j,l \in L,(i,j) \in I_0\}$.
\item $S_0^5=\{(\{(d,d),s\},s)|s\in I_{2,3}\}$.
\item $S_0^6\subset (I_{2,3})^{(2)}\times I_0$.
\end{itemize}
\end{itemize}
\item Splitting to simpler cases.
\begin{itemize}
\item
Define a weighting $w_{0}:I \to \Z$ by $w_0|_{I_0}=1, w_0|_{I-I_0}=0. $ Let $\Gamma_1=(I,J,S_1):=\gr_{w_{0}}(\Gamma_0)$. Since, for every element $(\left\{ \alpha,\beta \right\}, \gamma)\in S_0^6$, there is an element $(\left\{ \alpha ',\beta '\right\},\gamma)\in S_0^1$ and higher weight, all edges in $S_0^6$ disappear in $\Gamma_1$. Using similar arguments for $S_0^5$, we get that
$$S_1= \bigcup_{i=1}^4 S_0^i.$$
\item Let $M=\{0,1,2,3\}$, and let $\Gamma_2= \mathcal{L}(\Gamma_{1}):=(I \times M,J,S_{1} \times M)$. Define $w_2: I\to \Z_{\geq 0}^M$ by: $$w_2(s)(i)=
\begin{cases}
\delta_{i,0} & \text{if } s\in L\times L \\
3\delta_{i,1}+2\delta_{i,3} & \text{if } s\in I_{2}  \\
3\delta_{i,2}+2\delta_{i,3} & \text{if } s\in I_{3}\\
\end{cases}.$$
Let $\Gamma_3:=\gr_{w_{2}}(\Gamma_2).$ We can decompse $\Gamma_3$ into a disjoint union\footnote{by ``disjoint union" of \pg s $(I_1,J_1,S_1)$ and $(I_2,J_2,S_2)$ we mean the triple of disjoint unions $(I_1 \cup I_2,J_1 \cup J_2,S_1 \cup S_2)$ } of \pg s $\Gamma_4^1,\Gamma_4^2,\Gamma_4^3,\Gamma_4^4$ such that $\Gamma^{i}_4\cong(I,J,S^{i}_0)$, for $i=1,2,3,4$.
So, based on the type $A$ case and the fact that $\Gamma^{3}_4 \cong \Gamma^{4}_4$, it is enough to show that $(\Gamma^{2}_4,W)$ is \FRS\ whenever $\dim\, W \geq 2$ and $(\Gamma^{3}_4,W)$ is \FRS\ whenever $\dim\, W \geq 8.$  The first statement is obvious, so it remains to prove the second.

\item Identify $I_2$ with $L^{[2]}$. Let $I_{0,2}= I_0\cup I_2$ and  $S_5:=S_4^{3}$. We can replace $\Gamma^{3}_4$ with $\Gamma_5:=(I_{0,2},I_2,S_5).$
\end{itemize}

\item Reduction to a tree
\begin{itemize}
\item Define a weighting $w_{5}:I_{0,2} \to \Z$ by $$w_{5}(({i,j}))=w_{5}([{i,j]})=-3^{|i-j|}.$$
\item  Let $\Gamma_6=(I_{0,2},I_2,S_6):=\gr_{w_{5}}(\Gamma_5)$. We get
$$S_6=  \left\{ \left( \{(i,j),[j,l]\},[i,l] \right) \mid  i,j,l \in L,  \left |j-\frac{i+{l}}{2} \right| <1+\delta_{i,d}\delta_{l,d} \right\} .$$
\item Define a weighting $w_{6}:I_{0,2} \to \Z$ by $w_6((i,j))=-i-j$ and $w_{6}([i,j])=0$.
\item  Let $\Gamma_7=(I_{0,2},I_2,S_6):=\gr_{w_{6}}(\Gamma_6)$. We get  that $\Gamma_7$ is the \pg\ attached to the (bipartite) graph $\Gamma_8=(I_{0,2},E)$, where
$$E=  \left\{ \{(i,j),[j,l]\} \mid i,j,l \in L,  j=\left\lfloor\frac{i+{l}}{2} \right\rfloor-\delta_{i,n}\delta_{l,n},i\leq l \right\}. $$
\item It is easy to see that $\Gamma_{8}$ is a forest of maximal degree $\leq 3$.
\end{itemize}

\end{itemize}

We draw the graph $\Gamma_8$ for $d=7$ in Appendix \ref{app:ex}.

\section{Push Forward of Smooth Measures} \lbl{sec:push.forward}
\setcounter{theorem}{0}

In this section, $k$ will denote a finitely generated field of characteristic 0, and $F$ will denote a local non-archimedean field of characteristic 0. If $X$ is a smooth variety, we denote the line bundle of top differential forms by $\Omega_X$. Similarly, if $f:X \rightarrow Y$ is a smooth map, we denote the line bundle on $X$ of relative top differential forms by $\Omega_{X/Y}$. {More generally, for singular varieties and non-smooth maps, we use $\Omega_X$ (or $\Omega_{X/Y}$) for the shifted (relative) dualizing complex; see Appendix \ref{ssec:G.duality}}

\subsection{Measures and Forms} \lbl{subsec:measure.forms}

Recall that a measure $\mu$ on a Borel space $X$ is said to be absolutely continuous with respect to another measure $\nu$ on $X$ if, for any Borel subset $A \subset X$ such that $\nu(A)=0$, we have $\mu(A)=0$. In this case, the Radon--Nikodym theorem says that there is $f\in L_1(X,\nu)$ such that, for every subset $B \subset X$, we have $\mu(B)=\int_B \mu$; such $f$ is called the density of $\mu$ with respect to $\nu$. If $\nu$ is absolutely continuous with respect to $\mu$ and $\mu$ is absolutely continuous with respect to $\nu$, we say that $\mu$ and $\nu$ are in the same measure class. We first construct, for every smooth algebraic variety $X$ over a non-archimedean local field $F$, a canonical measure class.

\begin{definition} Suppose that $X$ is a smooth algebraic variety over a non-archimedean local field $F$, and that $\omega$ is a rational top differential form on $X$. We define a measure {$| \omega |_F$} on the analytic variety $X(F)$ as follows. Given a compact open set $U \subset X(F)$ and an analytic diffeomorphism $\phi$ between an open subset $W \subset F^n$ and $U$, write $$\phi^* \omega=g dx_1 \wedge \ldots \wedge dx_n,$$ for some $g: W \rightarrow F$, and define
\[
| \omega |_F (U)=\int_{W} |g|_F d \lambda,
\]
where $|g |_F$ is the normalized absolute value on $F$ and $\lambda$ is the standard additive Haar measure on $F^n$. Note that this definition is independent of the diffeomorphism $\phi$. There is a unique extension of the assignment $| \omega |_F$ to a non-negative (possibly infinite) Borel measure on $X(F)$, which we also call $| \omega |_F$. If the field $F$ is fixed, we omit it from the notation.
\end{definition}

The following Lemma and Proposition are evident:

\begin{lemma} \lbl{lem:measures.out.of.forms} Suppose that $X$ is a smooth variety over a non-archimedean local field $F$, and that $\omega_1,\omega_2$ are two top forms on $X$. Then \begin{enumerate}
\item If $\omega_1$ is regular, then $| \omega_1 |$ is a Radon measure, i.e., for any compact subset $A \subset X(F)$, we have $| \omega_1 |(A) < \infty$.
\item If both $\omega_1$ and $\omega_2$ are regular, then the measures $| \omega_1 |$ and $| \omega_2 |$ are absolutely continuous with respect to each other.
\item If $\omega_1$ and $\omega_2$ are regular and nowhere vanishing, then the density of $| \omega_2 |$ with respect to $| \omega_1 |$ is a locally constant function.
\item If $\omega_1$ has a pole at $x \in X(F)$, then $| \omega_1 |(A)=\infty$ for every open set $A$ containing $x$.
\end{enumerate}
\end{lemma}

\begin{proposition} Let $X$ be a smooth variety over a non-archimedean local field. \begin{enumerate}
\item A measure $m$ on $X(F)$ is Schwartz if and only if it is a linear combination of measures of the form $f | \omega |$, where $f$ is a locally constant and compactly supported function on $X(F)$, and $\omega$ is a rational top differential form on $X$ with no zero or pole in the support of $f$.
\item A measure $m$ on $X(F)$ has continuous density if and only if for every point $x\in X(F)$ there is a neighborhood $U$ of $x$, a continuous function $f:U \rightarrow \mathbb{C}$, and a rational top differential form $\omega$ with no poles in $U$ such that $m= f | \omega |$.
\end{enumerate}
\end{proposition}

If $m$ is a Schwartz measure on $X(F)$ and $x\in X(F)$, we say that $m$ vanishes at $x$ if the restriction of $m$ to some neighborhood of $x$ is the zero measure.

\subsection{Main Theorem} \lbl{subsec:strong.push.forward}
If $\varphi : X \rightarrow Y$ is a smooth map between two smooth $k$-varieties, then $\Omega_{X/Y}$ is an invertible sheaf on $X$ and there is an isomorphism\footnote{There is some ambiguity about the sign in the definition we give. This will be irrelevant for us, since we will be interested only in the absolute values of the forms.} $\Omega_X \rightarrow \varphi ^* \Omega_Y \otimes \Omega_{X/Y}$ such that, for each extension $k \subset F$ and every point $x\in X(F)$, the isomorphism of fibers $\wedge ^{\dim X}T^*_xX \rightarrow \left( \wedge^{\dim Y} T^*_{\varphi(x)} Y \right)  \otimes  \left( \wedge^{\dim \varphi ^{-1} (\varphi(x))} T^*_x \varphi ^{-1} ( \varphi (x)) \right) $ is obtained from the short exact sequence of vector spaces
\[
0 \rightarrow T_x \varphi ^{-1} ( \varphi (x)) \rightarrow T_x X \stackrel{d \varphi}{\rightarrow} T_{\varphi (x)} Y \rightarrow 0.
\]
Thus, if $\omega_X\in \Gamma(X,\Omega_X)$ is a top form on $X$ and $\omega_Y\in \Gamma(Y,\Omega_Y)$ is a nowhere vanishing top form on $Y$, there is a unique element $\eta \in \Gamma(X,\Omega_{X/Y})$ such that the image of $\eta \otimes \varphi ^*\omega_Y$ under the isomorphism $\Omega_{X/Y} \otimes \varphi ^* \Omega_Y \rightarrow \Omega_X$ is equal to $\omega_X$. We denote this element $\eta$ by $\frac{\omega_X}{\varphi ^* \omega_Y}$. If $X \rightarrow Y$ is not smooth, we define a relative top form $\frac{\omega_X}{\varphi ^* \omega_Y}$ on the smooth locus $X^S$ of $\varphi$ to be $\frac{\omega_X | _{X^S}}{\psi ^* \omega_Y}$, where $\psi$ is the restriction of $\varphi$ to $X^S$.

Our goal in this section is to prove the following

\begin{theorem} \lbl{thm:push.forward.detailed} Let $\varphi : X \rightarrow Y$ be a map between smooth algebraic varieties defined over a finitely generated field $k$ of characteristic $0$, and let $x \in X(k)$. Then \begin{enumerate}
\item The following conditions are equivalent:
\begin{enumerate}
\item $\varphi$ is \FRS\ at $x$.
\item There exists a Zariski open neighborhood $x \in U \subset X$ such that, for any local field $F \supset k$ and any Schwartz measure $m$ on $U(F)$, the measure $(\varphi|_{U(F)})_*(m) $ has continuous density.
\item For any finite extension $k'/k$, there exists a local field $F' \supset k'$ and a non negative Schwartz measure $m$ on $X(F')$ that does not vanish at $x$ such that $(\varphi|_{X(F')})_*(m) $ has continuous density.
\end{enumerate}

\item Let $F \supset k$ be a local field. Denote the smooth locus of $\varphi$ by $X^S$. If $\varphi$ is \FRS, $\omega_X$ is a nowhere-vanishing regular top differential form on $X_F$, $\omega_Y$ is a regular and nowhere-vanishing top differential form on $Y_F$, and $f:X(F) \rightarrow \mathbb{C}$ is a Schwartz function, then the density of $\varphi_* \left( f| \omega_X| \right)$ with respect to $| \omega_Y|$ is given by
\[
\frac{\varphi_*(f| \omega_X|)}{| \omega_Y|}(y)=\int_{(\varphi ^{-1} (y) \cap X^S)(F)}f\cdot \left|\frac{\omega_X}{\varphi ^* \omega_Y} \arrowvert _{\varphi ^{-1} (y) \cap X^S}\right|.
\]
In particular, the integral in the right hand side converges.
\end{enumerate}
\end{theorem}

This Section is organized as follows: Subsections \ref{subsec:generalities} and \ref{subsec:integration.RS} contain general results about integration and rational singularities; the rest of the section is devoted to the proof of Theorem \ref{thm:push.forward.detailed}. The implication $(a) \implies (b)$, as well as part 2 of the theorem are proved in Subsection \ref{subsec:FRS.to.continuous}. The implication $(c) \implies (a)$ is proved in Subsection \ref{subsec:continuous.to.FRS}. The implication $(b) \implies (c)$ follows from the fact that any finitely generated field is contained in a local field.

\subsection{Generalities on push forwards} \lbl{subsec:generalities}

\begin{proposition} \lbl{prop:push.forward.smooth.map} Let $\varphi: X \rightarrow Y$ be a smooth map between smooth varieties defined over a non-archimedean local field $F$. \begin{enumerate}
\item \label{prop:push.forward.smooth.map:1} If $m$ is a Schwartz measure on $X(F)$, then $\varphi_*m$ is a Schwartz measure on $Y(F)$.
\item  \label{prop:push.forward.smooth.map:2} Assume that $\omega_X$ and $\omega_Y$ are top forms on $X$ and $Y$ respectively, that $\omega_Y$ is nowhere vanishing, and that $f$ is a Schwartz function on $X(F)$. Then the measure $\varphi_*( f| \omega_X |)$ is absolutely continuous with respect to $| \omega_Y |$, and its density at a point $y\in Y(F)$ is equal to $\int_{\varphi ^{-1} (y)(F)} f \left| \frac{\omega_X}{\varphi ^* \omega_Y} \vert_{\varphi ^{-1}(y)}\right|$.
\end{enumerate}
\end{proposition}

\begin{proof} Using a partition of unity and an analytic change of coordinates, we can assume that $X$ and $Y$ are open subsets in affine spaces $\mathbb{A} ^n$ and $\mathbb{A} ^d$ respectively, and the map is a linear projection. In this case, the claim follows from Fubini's theorem.
\end{proof}

\begin{corollary} \lbl{cor:push.forward.dominant.map} Let $\varphi : X \rightarrow Y$ be a locally dominant map between smooth varieties defined over a non-archimedean local field $F$ of characteristic 0, and denote the smooth locus of $\varphi$ by $X^S$. \begin{enumerate}
\item\label{cor:push.forward.dominant.map:1} If $m_X$ is a Schwartz measure on $X(F)$ and $m_Y$ is a smooth nowhere vanishing measure on $Y(F)$, then $\varphi_*m_X$ is absolutely continuous with respect to $m_Y$.
\item Assume that $\omega_X$ and $\omega_Y$ are top forms on $X$ and $Y$ respectively, that $\omega_Y$ is nowhere vanishing, and that $f \geq 0$ is a Schwartz function on $X(F)$. Then the push forward $\varphi_*( f| \omega_X |)$ is absolutely continuous with respect to $| \omega_Y |$ and its density is represented by the function
\[
y \mapsto \int_{(X^S \cap \varphi ^{-1} (y))(F)} f \left| \frac{\omega_X}{\varphi ^* \omega_Y} \vert_{X^S \cap \varphi ^{-1}(y)}\right|
\]
from $Y(F)$ to $\mathbb{R}_{\geq 0}\cup \left\{ \infty \right\}$.
\end{enumerate}
\end{corollary}

\begin{proof} Without loss of generality, we can assume that $X$ is affine. Let $d$ be the metric on $X(F)$ induced from the $p$-adic metric on the affine space, let $X(F)^{S,n}$ be the (closed and open) set of points of $X^S(F)$ whose distance to $\left( X\smallsetminus X^S \right) (F)$ is greater than or equal to $\frac{1}{n}$, and let $g_n$ be the characteristic function of $X(F)^{S,n}$.

Since the measure $g_n m_X$ is a Schwartz measure on $X^S(F)$, Proposition \ref{prop:push.forward.smooth.map} \eqref{prop:push.forward.smooth.map:1} implies that $\varphi_*(g_n m_X)$ is absolutely continuous with respect to $m_Y$, so it has a density, which we denote by $h_n$. Note that $h_n$ is monotone non-decreasing in $n$ and that $\int_Y h_n m_Y \leq m_X(X(F)) < \infty$. By Lebesgue's Monotone Convergence Theorem, the limit $h=\lim_{n \rightarrow \infty} h_n$ exists and belongs to $L_1(Y(F),m_Y)$. Since $\varphi$ is locally dominant, $\varphi_*(m_X)=(\varphi |_{X^S})_*(m_X|_{X^S})$. Integrating against continuous functions and applying Lebesgue's Monotone Convergence Theorem again, we find that $\varphi_*(m_X)=(\varphi |_{X^S})_*(m_X|_{X^S})=h m_Y$, and the latter is absolutely continuous with respect to $m_Y$. The second statement follows from Proposition \ref{prop:push.forward.smooth.map} \eqref{prop:push.forward.smooth.map:2}.
\end{proof}

\subsection{Integration and Rational Singularities} \lbl{subsec:integration.RS}

\begin{lemma} \lbl{lem:RS.implies.finite.integral} Suppose that $V \subset \mathbb{A} ^n$ is a variety with rational singularities, that $U \subset V$ is an open smooth subset whose complement has codimension at least two, and that $\omega$ is a top form on $U$. Then the integral $\int_{U(F)\cap O^n} | \omega |$ is finite.
\end{lemma}

\begin{proof} Let $\widetilde{V} \rightarrow V$ be a strong resolution of singularities (see Subsection \ref{sssec:res_sing}), and let $i: U \rightarrow V$ be the inclusion. Using Proposition \ref{prop:equviv_rat} (\ref{prop:equviv_rat:5}') and the assumption, we find that there is a top form $\eta \in \Gamma(\widetilde{V},\Omega_{\widetilde{V}})$ such that $\pi_*\eta |_{\pi ^{-1} (U)}=\omega$. In particular,

\begin{equation} \lbl{eq:int.resolution}
\int_{U(F)\cap O^n} d| \omega |=\int_{\pi ^{-1} (U(F)\cap O^n)}d| \eta | = \int_{\pi ^{-1} (V(F)\cap O^n)} d| \eta |,
\end{equation}
where the second equality is because $\pi ^{-1} (V(F) \smallsetminus U(F))$ has positive codimension. But the last integral in (\ref{eq:int.resolution}) is over a compact set, so it is finite.
\end{proof}

\begin{corollary} \label{cor:RS.implies.finite.integral} Let $X$ be a Gorenstein variety with rational singularities over a non-archimedean local field of characteristic 0, and let $\omega$ be a top differential form on $X^{sm}$. For any $A \subset X(F)$, define
\[
m(A)= \int_{A\cap X^{sm}(F)} | \omega |.
\]
Then $m$ is a Radon measure on $X(F)$, i.e., the measure of each compact subset is finite.
\end{corollary}

Using Corollary \ref{cor:RS.implies.finite.integral}, we can generalize the notion of a Schwartz measure to Gorenstein varieties with rational singularities.

\begin{definition} Let $X$ be a Gorenstein variety with rational singularities over a non-archimedean local field $F$ of characteristic 0. \begin{enumerate}
\item If $\omega$ is a top differential form on $X^{sm}$, denote the measure $m$ from Corollary \ref{cor:RS.implies.finite.integral} by $| \omega |$.
\item A measure $m$ on $X(F)$ is called {\emph{smooth}} if, for any point $x\in X(F)$, there is a Zariski neighborhood $U$ of $x$ such that the restriction of $m$ to $U(F)$ is equal to $f| \omega |$, where $f:X(F) \rightarrow \mathbb{C}$ is a locally constant function, and $\omega$ is an invertible top differential form on $U\cap X^{sm}$.
\item A {\emph{Schwartz measure}} on $X(F)$ is a smooth measure with compact support.
\end{enumerate}
\end{definition}

\begin{lemma} \label{lem:push.forward.dominant.RS} Let $X$ be a Gorenstein variety with rational singularities over a non-archimedean local field $F$ of characteristic 0, let $Y$ be a smooth variety over $F$, and let $f:X \rightarrow Y$ be a locally dominant map. Denote the smooth locus of $\varphi$ by $X^S$. \begin{enumerate}
\item If $m_X$ is a Schwartz measure on $X(F)$ and $m_Y$ is a smooth nowhere vanishing measure on $Y(F)$, then $\varphi_*m_X$ is absolutely continuous with respect to $m_Y$.
\item Assume that $\omega_X$ and $\omega_Y$ are top forms on $X^{sm}$ and $Y$ respectively, that $\omega_Y$ is nowhere vanishing, and that $f \geq 0$ is a Schwartz function on $X(F)$. Then the push forward $\varphi_*( f| \omega_X |)$ is absolutely continuous with respect to $| \omega_Y |$ and its density is represented by the function
\[
y \mapsto \int_{(X^S \cap \varphi ^{-1} (y))(F)} f \left| \frac{\omega_X}{\varphi ^* \omega_Y} \vert_{X^S \cap \varphi ^{-1}(y)}\right|
\]
from $Y(F)$ to $\mathbb{R}_{\geq 0}\cup \left\{ \infty \right\}$.
\end{enumerate}
\end{lemma}

\begin{proof} The first statement follows from the second. In order to prove the second statement, choose a resolution of singularities $\pi : \widetilde{X} \rightarrow X$, denote $\widetilde{\varphi}=\varphi \circ \pi$, and let $\omega_{\widetilde{X}}$ be a top form that coincides with $\pi ^* \omega_X$ on an open dense set. Applying Corollary \ref{cor:push.forward.dominant.map} to the map $\widetilde{\varphi}$ and the forms $\omega_{\widetilde{X}}$ and $\omega_Y$, we get that the measure $\widetilde{\varphi}_* \left( (f \circ \pi) | \omega_{\widetilde{X}} | \right) =\varphi_*(f | \omega_ X|)$ is absolutely continuous with respect to $| \omega_Y|$ and its density at $y\in Y(F)$ is given by
\[
\int_{\widetilde{X}^S \cap \widetilde{\varphi} ^{-1} (y)} f \left| \frac{\omega_{\widetilde{X}}}{\widetilde{\varphi} ^* \omega_Y} \vert_{\widetilde{X}^S \cap \widetilde{\varphi} ^{-1}(y)}\right|.
\]

Let $U \subset X$ be an open dense set over which $\pi$ is an isomorphism. It follows that $\pi ^{-1}(U)$ is open and dense in $\widetilde{X}$, and that there is an open dense set $V \subset Y$ such that, for any $y\in V$, the set $U\cap \varphi ^{-1}(y)$ is dense in $\varphi ^{-1} (y)$ and the set $\pi ^{-1}(U) \cap \widetilde{\varphi} ^{-1} (y)$ is dense in $\widetilde{\varphi} ^{-1} (y)$. We get

\[
\int_{X^S \cap \varphi ^{-1} (y)} f \left| \frac{\omega_X}{\varphi ^* \omega_Y} \vert_{X^S \cap \varphi ^{-1}(y)}\right| = \int_{U \cap X^S \cap \varphi ^{-1} (y)} f \left| \frac{\omega_X}{\varphi ^* \omega_Y} \vert_{X^S \cap \varphi ^{-1}(y)}\right| =
\]

\[
=\int_{\pi ^{-1} \left( U \cap X^S \cap \varphi ^{-1} (y) \right) } (f\circ \pi) \left| \frac{\omega_{\widetilde{X}}}{\widetilde{\varphi} ^* \omega_Y} \vert_{\pi ^{-1} (U \cap X^S \cap \varphi ^{-1}(y))}\right| =
\]
\[
= \int_{\widetilde{X}^S \cap \widetilde{\varphi} ^{-1} (y)} (f\circ \pi) \left| \frac{\omega_{\widetilde{X}}}{\widetilde{\varphi} ^* \omega_Y} \vert_{\widetilde{X}^S \cap \widetilde{\varphi} ^{-1}(y)}\right|,
\]
proving (2).
\end{proof}

For the next lemma, recall that if $X$ is a Cohen--Macaulay variety then it has a (shifted) dualizing sheaf $\Omega_X$. The restriction of $\Omega_X$ to the smooth locus of $X$ is identified with the sheaf of top differential forms (see Appendix \ref{sec:alg.geom}).

\begin{lemma} \lbl{lem:CM.finite.integral.implies.RS} Let $X$ be a Cohen--Macaulay variety defined over a finitely generated field $k$, let $x\in X(k)$, and denote the smooth locus of $X$ by $X^{sm}$. Suppose that, for any finite extension $k'/k$ and any section $\omega$ of $\Omega_X$, there is a local field $F$ containing $k'$ and a non-negative Schwartz function $f\in \Sc(X(F))$ with $f(x) \neq 0$ such that the integral $\int_{X^{sm}(F)}f | \omega |$ is finite. Then $x$ is a rational singularity of $X$.
\end{lemma}

\begin{proof} Let $\pi :\widetilde{X} \rightarrow X$ be a resolution of singularities, and let $\omega$ be a section of $\Omega_X$ that is regular at $x$. By passing to a Zariski neighborhood of $x$, it is enough to prove that the pullback of $\omega|_{X^{sm}}$ to $\widetilde{X}$ extends to a regular top form (see Proposition \ref{prop:equviv_rat} (\ref{prop:equviv_rat:4}')). Assuming it does not, it must have a pole. Let $k'$ be a finite extension of $k$ such that the $k'$-points of the pole locus of $\pi ^* \omega$ are Zariski dense, let $F$ be a local field containing $k'$, and let $f$ be a Schwartz function as in the conditions of the lemma. Choose a point $p\in \widetilde{X}(F)$ in the pole locus of $\pi ^* \omega$ such that $f(\pi(p))\neq 0$, and let $A \subset \widetilde{X}(F)$ be a neighborhood of $p$ on which $f\circ \pi$ is constant. Since
\[
f(\pi(p))\int_{A}| \pi ^* \omega | \leq \int_{\pi ^{-1}(X^{sm}(F))} (f\circ \pi) | \pi ^* \omega |=\int_{X^{sm}(F)} f | \omega | < \infty,
\]
we get a contradiction to Lemma \ref{lem:measures.out.of.forms}.
\end{proof}

\begin{corollary}\lbl{lem:Gorenstein.finite.integral.implies.RS} Let  $X$ be a Gorenstein variety defined over a finitely generated field $k$, let $x\in X(k)$, and denote the smooth locus of $X$ by $X^{sm}$. Let $\omega$ be a section of $\Omega_X$ that does not vanish at $x$. Suppose that, for any finite extension $k'/k$, there exists a local field $F$ containing $k'$ and a non-negative Schwartz function $f\in \Sc(X(F))$ with $f(x) \neq 0$ such that the integral $\int_{X^{sm}(F)}f | \omega |$ is finite. Then $x$ is a rational singularity of $X$.
\end{corollary}

\subsection{\FRS\ Implies Continuous push forward} \lbl{subsec:FRS.to.continuous}

In this subsection, we prove the implication $(a) \implies (b)$ of the first part of Theorem \ref{thm:push.forward.detailed}, as well as the second part of that theorem. After a base change to $F$ and using the fact that the claims are local, it is enough to prove the following stronger theorem:

\begin{theorem} \label{thm:yet.another} Let $\varphi:X \rightarrow Y$ be an \FRS\ map of affine varieties over a local field $F$ of characteristic 0. Assume that $Y$ is smooth (and therefore, by Elkik's Theorem \ref{thm:elk} (2), $X$ has rational singularities) and that $X$ is Gorenstein. Let $\omega_X$ is a regular nowhere vanishing top differential form on the smooth locus $X^{sm}$ of $X$, let $\omega_Y$ is a regular nowhere vanishing top differential form on $Y$, and let $f:X(F) \rightarrow \mathbb{C}$ be a Schwatz function. Denote the smooth locus of $\varphi$ by $X^S$. Then the measure $\varphi_*(f| \omega_X |)$ has continuous density with respect to $| \omega_Y |$, which is given by
\[
\frac{\varphi_*(f| \omega_X|)}{| \omega_Y|}(y)=\int_{(\varphi ^{-1} (y) \cap X^S)(F)}f\cdot \left|\frac{\omega_X}{\varphi ^* \omega_Y} \arrowvert _{\varphi ^{-1} (y) \cap X^S}\right|.
\]
\end{theorem}

\subsubsection{Reduction to a Curve}

Let $m=f | \omega_X |$. Embedding $X$ in $\mathbb{A} ^n$, we can require that $f$ is the indicator function of $X(F) \cap O^n$. By Lemma \ref{lem:push.forward.dominant.RS}, $\varphi_*m$ has $L_1$-density with respect to $| \omega_Y |$; we need to show that this density is continuous. \\

Let $X^S$ be the smooth locus of $\varphi$. By assumption, all fibers of $\varphi$ are reduced, and, hence, generically smooth. Proposition \ref{prop:sm_crit} says that, for any $y\in Y(F)$, the smooth locus of $\varphi ^{-1} (y)$ is equal to $X^S \cap \varphi ^{-1} (y)$.

Since the restriction of $f$ to $\varphi ^{-1} (y)(F)$ is a Schwartz function, Lemma \ref{lem:RS.implies.finite.integral} implies that the integrals $\int_{X^S\cap \varphi ^{-1} (y)(F)} f \left|\frac{\omega_X}{\varphi^* \omega_Y} \right|$ are all convergent. We denote the function $y \mapsto \int_{X^S\cap \varphi ^{-1} (y)(F)} f \left|\frac{\omega_X}{\varphi^* \omega_Y} \right|$ by $\varphi_* f$. Note that $\varphi_* f$ depends on the choices of $\omega_X$ and $\omega_Y$, although our notation omits them. By Lemma \ref{lem:push.forward.dominant.RS}, $\varphi_*f$ is a function representing the density of $\varphi_*(f | \omega_X |)$ with respect to $| \omega_Y |$.

Thus, in order to prove Theorem \ref{thm:yet.another}, it is enough to prove that $\varphi_*f$ is continuous. By Theorem \ref{thm:continuity.along.curves}, the map $\varphi _*f$ is continuous if and only if its restriction to any curve $C \subset Y$ is continuous. Given $C \subset Y$, let $\widetilde{C} \stackrel{\eta}{\rightarrow} C$ be the normalization of $C$, and consider the pullback diagram
\[
\xymatrix{ \widetilde{X} \ar[r]^{\widetilde{\eta}} \ar[d]^{\widetilde{\varphi}} & X \ar[d]^{\varphi} \\ \widetilde{C} \ar[r]^{\eta} & Y} .
\]
If $\widetilde{f}: \widetilde{X}(F) \rightarrow \mathbb{C}$ is the composition $f\circ \widetilde{\eta}$, then $\widetilde{f}$ is Schwatz, and the function $\varphi_*f$ is continuous if and only if $\widetilde{\varphi}_* \widetilde{f}$ is continuous. Thus, we may assume that $Y$ is a smooth curve. A similar argument starting from an etale map from a neighborhood of $\varphi(x)$ to $\mathbb{A} ^1$ shows that we can assume that $Y$ is equal to $\mathbb{A} ^1$. We need to prove, under these assumptions, that $\varphi_*f$ at $0$. \\

\subsubsection{Reduction to a Local Model}
Let $X_0=\varphi ^{-1}(0)$ and let $\widetilde{X} \stackrel{\pi}{\rightarrow} X$ be a strong resolution of singularities of the pair $(X,X_0)$ (see \ref{sssec:res_sing}). By Sard's theorem, there are only finitely many singular values of $\widetilde{\varphi}=\varphi \circ \pi$. Hence, after (Zariski) base change, we can assume that the only singular value is 0. In particular, for every $t \neq 0$, the variety $\widetilde{\varphi} ^{-1} (t)$ is non-singular and the map $\widetilde{\varphi} ^{-1} (t) \rightarrow \varphi ^{-1} (t)$ is a resolution of singularities. By Proposition \ref{prop:strict.transform}, the strict transform $X_0'$ of $X_0$ in $\widetilde{X}_0:=\widetilde{\varphi} ^{-1} (0)$ is a resolution of singularities of $X_0$.

{Let $W \subset X$ be an open dense set over which $\pi$ is an isomorphism. Let $\omega_{\widetilde{X}} \in \Gamma(\widetilde{X},\Omega_{\widetilde{X}})$ be a top differential form such that $\omega_{\widetilde{X}}|_{\pi ^{-1} W}=\pi ^*\left( \omega_X|_W \right)$.} We have $\varphi_*(1_{X(O)}| \omega_X |)=\widetilde{\varphi}_*(1_{\widetilde{X}(O)}| \omega_{\widetilde{X}} |)$. Since $\widetilde{X}_0$ is a divisor with strict normal crossings inside a smooth variety, for every $\widetilde{p} \in \widetilde{X}_0(F)$ there is a coordinate system $x_1,\ldots,x_n$ in a neighborhood of $\widetilde{p}$ on which both $\widetilde{\varphi}$ and $\omega_{\widetilde{X}}$ are monomial, say
\[
\widetilde{\varphi}=\alpha x_1^{a_1}\cdot\ldots\cdot x_n^{a_n}
\]
and
\[
\omega_{\widetilde{X}}=\beta x_1^{b_1}\cdot\ldots\cdot x_n^{b_n} d x_1 \wedge \cdots \wedge d x_n,
\]
{where $\alpha$ and $\beta$ are invertible. Since $X_0' \cap \pi ^{-1}(W)$ is isomorphic to $W\cap X_0$, which is reduced, we get that $X_0'$ is reduced. If $\widetilde{p} \in X_0'$, then $X_0'$ is locally the zero locus of one of the $x_i$, which we can assume to be $x_1$. This means that $a_1$ is equal to 1. The following key proposition implies that if $\widetilde{p} \in X_0'$, then $a_i \leq b_i$ for $i \geq 2$, and, if $\widetilde{p} \notin X_0'$, then $a_i \leq b_i$ for all $i$.}

\begin{proposition} \lbl{prop:vanishing.form} If $\widetilde{\eta}$ is a rational top form on $\widetilde{X}$ such that $\widetilde{\varphi} \widetilde{\eta}$ is regular, then $\widetilde{\eta}$ is regular on $\widetilde{X}_0\smallsetminus X_0'$.
\end{proposition}

Before giving a formal proof, we sketch the argument. Suppose that $\widetilde{\eta}$ is a rational top form on $\widetilde{X}$ such that $\widetilde{\varphi}\widetilde{\eta}$ is regular. Since $X$ has rational singularities, there is an element $\zeta\in \Gamma(X,\Omega_X)$ that coincides with $\widetilde{\varphi}\widetilde{\eta}$ on an open dense set. The element $\eta=\frac{\zeta}{\varphi}$ is a rational section of $\Omega_X$ with (at most) a simple pole along $X_0$. Let $\omega_0\in \Gamma(X_0,\Omega_{X_0})$ be the residue\footnote{$X$ and $X_0$ are not a smooth varieties, so we need to use the formalism of Grothendieck Duality in order to manipulate residues.} of $\eta$ along $X_0$. Since $X_0$ has rational singularities, there is an element $\widetilde{\omega_0}\in \Gamma(X_0',\Omega_{X_0'})$ that coincides with the pull back of $\omega_0$ on an open dense set in $X_0'$. Let $\widetilde{\omega}$ be a rational top form on $\widetilde{X}$ that is regular outside $X_0'$, and has a simple pole along $X_0'$ with residue $\widetilde{\omega_0}$ (we use the fact that $\Omega_{\widetilde{X}}$ is acyclic in order to prove the existence of $\widetilde{\omega}$; see below).

Consider the regular top form $\widetilde{\delta}:=\widetilde{\varphi}(\widetilde{\eta}-\widetilde{\omega})$. Since $X$ has rational singularities, there is an element $\delta\in \Gamma(X,\Omega_X)$ that coincides with $\widetilde{\delta}$ on an open dense set. Computing residues, we see that $\delta$ vanishes on $X_0$, so it is divisible by $\varphi$. It follows that $\widetilde{\delta}$ is divisible by $\widetilde{\varphi}$, so $\widetilde{\eta}-\widetilde{\omega}$ is regular. Since the only poles of $\widetilde{\omega}$ are along $X_0'$, the same is true for $\widetilde{\eta}$.

\begin{proof}[Proof of \ref{prop:vanishing.form}] In the following, if $f: Y \rightarrow Z$ is a birational map (respectively, a closed immersion of a codimension 1 subvariety), we identify $f^! \Omega_Z$ with $\Omega_Y$ (respectively, $\Omega_Y[1]$).

Let $\Omega_{\widetilde{X}}([\widetilde{X_0}])$ be the sheaf of rational top forms $\widetilde{\eta}$ on $\widetilde{X}$ such that $\widetilde{\varphi}\widetilde{\eta}$ is regular. Similarly, let $\Omega_{\widetilde{X}}([X_0'])$ be the sheaf of rational top forms that have at most a simple pole along $X_0'$. It is enough to show that the natural inclusion $\Omega_{\widetilde{X}}([X_0']) \rightarrow \Omega_{\widetilde{X}}([\widetilde{X_0}])$ becomes an isomorphism after applying $\pi_*$.

Consider the commutative diagram of varieties
\[
\xymatrix{X_0' \ar[r]^{\widetilde{i}} \ar[dr]^k \ar[d]_{\pi_0} & \widetilde{X} \ar[d]^{\pi} \\ X_0 \ar[r]^i & X}
\]
where $i$ and $\widetilde{i}$ are the inclusions and $\pi_0$ is the restriction of $\pi$. Applying Proposition \ref{prop:shrik_prop} part \ref{prop:shrik_prop:42} to both triangles, we get the following commutative diagram of objects in the derived category $D^+(X)$
\begin{equation} \label{eq:cd.objects.main.lemma}
\xymatrixcolsep{5pc} \xymatrix{Ri_*R (\pi_0)_* \Omega_{X_0'}[1] \ar[r]^{Ri_* \Tr_{\pi_0}[1]} \ar[d] & Ri_* \Omega_{X_0}[1] \ar[d]_{\Tr_i} \\ Rk_* \Omega_{X_0'}[1] \ar[r]^{\Tr_k} & \Omega_X \\ R \pi_* R \widetilde{i}_* \Omega_{X_0'}[1] \ar[u] \ar[r]_{R \pi_* \Tr_{\widetilde{i}}} & R \pi_* \Omega_{\widetilde{X}} \ar[u]_{\Tr_{\pi}}}
\end{equation}
where the two left vertical maps are the natural isomorphisms. Note that all derived direct images are actually the usual direct images: For $i$ and $\widetilde{i}$, this is clear since they are affine. For $\pi$ and $\pi_0$, it follows from the Grauert--Riemannschneider Theorem (\ref{thm:GR}), since they are resolutions of singularities. In all other cases, this follows using the left vertical isomorphisms.

Consider first the top square of \eqref{eq:cd.objects.main.lemma}. The map $\Tr_i$ is represented by an extension
\[
0 \rightarrow \Omega_X \rightarrow \Omega_X([X_0])\rightarrow i_* \Omega_{X_0} \rightarrow 0
\]
(see Proposition {\ref{prop:shrik_prop}\eqref{prop:shrik_prop:6c}}), where the map $\Omega_X \rightarrow \Omega_X([X_0])$ is the obvious inclusion. Composing $\Tr_i$ with the map $i_*\Tr_{\pi_0}: i_* (\pi_0)_* \Omega_{X_0'} \rightarrow i_* \Omega_{X_0}$, we get an extension
\begin{equation} \label{eq:ext.bottom}
0 \rightarrow \Omega_X \rightarrow \Omega_X([X_0])\rightarrow i_*  (\pi_0)_*\Omega_{X_0'} \rightarrow 0.
\end{equation}
Next, consider the lower square of \eqref{eq:cd.objects.main.lemma}. The map $\Tr_{\widetilde{i}}$ is represented by an extension
\begin{equation} \label{eq:ext.temp}
0 \rightarrow \Omega_{\widetilde{X}}\rightarrow \Omega_{\widetilde{X}}([X_0'])\rightarrow \widetilde{i}_*\Omega_{X_0'} \rightarrow 0 ,
\end{equation}
where the map $\Omega_{\widetilde{X}} \rightarrow \Omega_{\widetilde{X}}([X_0'])$ is the obvious inclusion. Applying $\pi_*$ and the trace map $\Tr_\pi : \pi_* \Omega_{\widetilde{X}} \rightarrow \Omega_X$ to the extension \eqref{eq:ext.temp}, we get an extension
\begin{equation} \label{eq:ext.top}
0 \rightarrow \Omega_X \rightarrow \pi_*(\Omega_{\widetilde{X}}([X_0']))\rightarrow \pi_* \widetilde{i}_*\Omega_{X_0'} \rightarrow 0.
\end{equation}
By the commutativity of \eqref{eq:cd.objects.main.lemma}, these two extensions are isomorphic, meaning that there is a commutative diagram
\begin{equation} \label{eq:cd.isom}
\xymatrix{ 0 \ar[r] & \Omega_X \ar[r] & \Omega_X([X_0]) \ar[r] & Ri_* R (\pi_0)_*\Omega_{X_0'} \ar[r] & 0 \\ 0 \ar[r] & \Omega_X \ar[r] \ar[u] & \pi_*(\Omega_{\widetilde{X}}([X_0'])) \ar[r] \ar[u] & R \pi_* R \widetilde{i}_*\Omega_{X_0'} \ar[r] \ar[u] & 0
}
\end{equation}
where the horizontal sequences are \eqref{eq:ext.bottom} and \eqref{eq:ext.top}, the leftmost vertical map is the identity, and the other vertical maps are isomorphisms. In particular, we get an isomorphism $f_1:\pi_*(\Omega_{\widetilde{X}}([X_0'])) \rightarrow \Omega_X([X_0])$ that restricts (after the natural identification) to the identity on $\Omega_{X \smallsetminus X_0}$.

On the other hand, the isomorphism $\Tr_{\pi}:\pi_* \Omega_{\widetilde{X}}=R\pi_* \Omega_{\widetilde{X}} \rightarrow \Omega_X$ gives an isomorphism $f_2:\pi_* \Omega_{\widetilde{X}}([\widetilde{X}_0]) \rightarrow \Omega_X([X_0])$ that restricts, similarly, to the identity on $\Omega_{X\smallsetminus X_0}$. The composition $f_2 ^{-1} \circ f_1$ agrees with the obvious inclusion $\pi_*\Omega_{\widetilde{X}}([X_0']) \rightarrow \pi_* \Omega_{\widetilde{X}}([\widetilde{X_0}])$ on $X\smallsetminus X_0$, so they must be equal. Hence, the inclusion $\pi_*\Omega_{\widetilde{X}}([X_0']) \rightarrow \pi_* \Omega_{\widetilde{X}}([\widetilde{X_0}])$ is an isomorphism.

\end{proof}

\subsubsection{Push Forward of Monomial Measures Under Monomial Maps}

From the argument above it follows that, in order to prove that $\widetilde{\varphi}_*(1_{\widetilde{X}}| \omega_{\widetilde{X}} |)$ has continuous density (and, thus, complete the proof of Theorem \ref{thm:yet.another}), it is enough to show the following:
{

\begin{lemma} Denote the normalized Haar measure of $O$ by $\lambda$ and the normalized Haar measure of $O^n$ by $\lambda ^{\boxtimes n}$. Let $A=(a_1,\ldots,a_n)\in \mathbb{Z}_{\geq 0}$ and $B=(b_1,\ldots,b_n)\in \mathbb{Z}_{\geq 0}$. Assume that one of the following holds: \begin{enumerate}
\item $a_1=1$ and $a_i \leq b_i$ for $i \geq 2$. or
\item Not all $a_i$ are 0 and $a_i \leq b_i$ for all $i$.
\end{enumerate}
Then the push forward of the measure $x^B \lambda^{\boxtimes n}$ with respect to the map $x^A: O^n \rightarrow O$ has continuous density with respect to $\lambda$.
\end{lemma}

\begin{proof} By integrating away the variables $x_i$ for which $a_i=0$, we can assume that $a_1 \neq 0$. Denote the size of the residue field of $O$ by $q$. The group $O^ \times$ acts on $O^n$ by $\alpha (x_1,\ldots,x_n)=(\alpha x_1,x_2,\ldots,x_n)$, preserving the measure $x^B \lambda ^{\boxtimes n}$, and the map $x^A$ intertwines this action with the action of $O^ \times$ on $O$ given by $\alpha \cdot x=\alpha ^{a_1}x$. Hence, the measure $x^A_*(x^B \lambda ^{\boxtimes n})$ is invariant to this action of $O^ \times$. In particular, it is smooth outside 0. In order to prove that the push-forward has continuous density, { we will analyze its density outside $0$ and show that it extends continuously to $0$.

We first compute the measure $x^A_*(x^B \lambda ^{\boxtimes n})$  of the annulus $\mathcal{A}_r$ of radius $q^{-r}.$}
For every $r\in \mathbb{Z}_{\geq 0}$, the pre-image of {$\mathcal{A}_r$} is the union of the sets $$X_{r_1,\ldots,r_n}:=\left\{ (x_1,\ldots,x_n) \in O^n \mid |x_i|=q^{r_i} \right\},$$ where $(r_1,\ldots,r_n)$ satisfies $\sum a_ir_i = r$. {Denote $c=\frac{q-1}{q} $.} Since $(x^B \lambda ^{\boxtimes n})(X_{r_1,\ldots,r_n})=c^nq^{-\sum (b_i+1)r_i}$, we get that
{
\[
(x^A)_*\left( x^B \lambda ^{\boxtimes n}\right) (\mathcal{A}_r)= c^n\sum_{\underset{}{r_1,\ldots,r_n}} q^{-\sum(b_i+1)r_i},
\]

where the outer sum is over all $r_2,\ldots,r_n\in \mathbb{Z}_{\geq 0}$ such that $\sum a_ir_i = r$.

We consider the following cases:
\begin{enumerate}[{Case} 1.]
\item $b_1=0.$\\
In this case $a_1=1$, so $r_1=r-\sum_{i \geq 2} a_i r_i$. We get
 \[
 (x^A)_*\left( x^B \lambda ^{\boxtimes n}\right) (\mathcal{A}_r)= c^n\sum_{r_2,\ldots,r_n} q^{-\sum(b_i+1)r_{i}-(r-\sum a_ir_i)}=c^nq^{-r}\sum_{r_2,\ldots,r_n} q^{-\sum (b_i+1-a_i)r_i},
 \]
where the outer sums are over all $r_2,\ldots,r_n\in \mathbb{Z}_{\geq 0}$ such that $\sum a_ir_i \leq r$.

Since $a_1=1$, the density of $(x^A)_*\left( x^B \lambda ^{\boxtimes n}\right) $ is constant on each $\mathcal{A}_r$. This density is equal to
$$\frac{(x^A)_*\left( x^B \lambda ^{\boxtimes n}\right) (\mathcal{A} _r)}{\lambda(\mathcal{A} _r)}
={c^{n-1}}\sum_{r_2,\ldots,r_n} q^{-\sum (b_i+1-a_i)s_i},$$
Where the sum is as before. The assertion follows now from the fact that the right hand side converges when $r \to \infty$, by the assumptions.

{
\item $a_i \leq b_i$.

After relabeling, we can assume that $\frac{b_1+1}{a_1}\leq \frac{b_i+1}{a_i}$ for all $i$. We have
\begin{multline*}
(x^A)_* \left( x^B \lambda ^{\boxtimes n} \right) (\mathcal{A} _r) \leq c^n\sum_{r_2,\ldots,r_n} q^{-\sum_{i=2}^n (b_i+1)r_i-(b_1+1)\frac{\left( r-\sum_{i=2}^n a_i r_i \right)}{a_1} }=\\=c^nq^{-\frac{b_1+1}{a_1}r}\sum_{r_2,\ldots,r_n} q^{\sum_{i=2}^n \left( \frac{(b_1+1)a_i}{a_1}-(b_i+1) \right) r_i},
\end{multline*}
where the outer sums are over all non-negative integers $r_2,\ldots,r_n$ such that $\sum_2^na_ir_i \leq r$. Since $\frac{(b_1+1)a_i}{a_1}-(b_i+1)\leq 0$ and since $b_1 \geq a_1$, we get
\[
(x^A)_* \left( x^B \lambda ^{\boxtimes n} \right) (\mathcal{A}_r) \leq c^nq^{-r} q^{-\frac{r}{a_1}} r^{n-1}.
\]
Let $N$ be the number of $O^\times$ orbits in $\mathcal{A}_r$ (note that it does not depend on $r$). The density of $(x^A)_*\left( x^B \lambda ^{\boxtimes n}\right)$ at each point of $\mathcal{A} _{r}$ is bounded from above by
\[
\frac{N\cdot  (x^A)_*\left( x^B \lambda ^{\boxtimes n}\right) (\mathcal{A} _r)}{\lambda(\mathcal{A} _r)} \leq c^{n-1}N q^{-\frac{r}{a_1}}r^{n-1},
\]
which tends to 0 as $r$ tends to 0.
}
 \end{enumerate}
}
\end{proof}
}

\subsection{Continuous push forward Implies \FRS} \lbl{subsec:continuous.to.FRS}

In this section we prove the implication $(c) \implies (a)$ in Theorem \ref{thm:push.forward.detailed}. Suppose that $X,Y$ are smooth varieties over a finitely generated field $k$, that $\varphi :X \rightarrow Y$ is a map, and that $x\in X(k)$ such that, for any finite extension $k' \supset k$, there is a local field $F \supset k'$ and a non-negative Schwartz measure $m$ on $X(F)$ that does not vanish at $x$ such that $\varphi_* m$ has a continuous (and hence bounded) density. Let $X^S$ be the smooth locus of $\varphi$ and let $Z=\varphi ^{-1} (\varphi (x))$.

\begin{claim} \label{cla:Z.dense}After passing to a Zariski neighborhood of $x$, $Z \cap X^S$ is a dense subvariety of $Z$.
\end{claim}

In order to prove the claim, we use the following lemma:

\begin{lemma} \label{lem:Zariski.is.dense} Let $V$ be an irreducible variety over a non-archimedean local field $F$ and let $U \subset V^{sm}$ be a Zariski open dense set. Suppose that $U(F) \neq \emptyset$. Then $U(F)$ is dense in $V(F)$ in the analytic topology.
\end{lemma}

\begin{proof} We first show the lemma in the case where $V$ is smooth. If $x\in V(F)$ is a smooth point, then the implicit function theorem implies that $x$ has a basis of neighborhoods $\mathcal{N}_i$ consisting of subsets diffeomorphic to $O^{\dim V}$. Since $\dim (V \smallsetminus U) < \dim V$,  it is impossible that $\mathcal{N}_i \subset (V \smallsetminus U)(F)$. Hence $\mathcal{N}_i \cap U(F) \neq \emptyset$ for all $i$.

If $V$ is not smooth, let $\pi:\widetilde{V} \rightarrow V$ be a strong resolution of singularities. By the above, $\pi ^{-1} (U)(F)$ is dense in $\widetilde{V}(F)$. Since $\pi$ is continuous, $U(F)$ is dense in $V(F)$.
\end{proof}

\begin{proof}[Proof of Claim \ref{cla:Z.dense}] Without loss of generality, we can assume that $X$ is irreducible. Let $Z_1,\ldots,Z_n \subset Z$ be the absolute irreducible components of $Z$ containing $x$. Note that $Z_i$ might not be defined over $k$. By passing to a Zariski neighborhood, it is enough to show that $Z_i \cap X^S$ is Zariski dense in $Z_i$ for all $i$. Since $X^S$ is open, it is enough to show that $Z_i \cap X^S$ is non-empty for all $i$.

Fix some $i$, and assume, by contradiction, that $Z_i\cap X^S=\emptyset$. Then $\dim \ker d \varphi |_z > \dim X-\dim Y$ for all $z\in Z_i(\overline{k})$. There is an open set $W_i \subset Z_i$ and an integer $r \geq 1$ such that $\dim \ker d \varphi |_z = \dim X-\dim Y+r$ for all $z\in W_i(\overline{k})$.

Let $k' \supset k$ be a finite extension such that $Z_i,W_i$ are defined over $k'$ and $W_i^{sm}(k') \neq \emptyset$. By the assumption, we get a local field $F \supset k'$ and a Schwartz measure $m$ on $X(F)$ that does not vanish at $x$ and such that $\varphi_*m$ has continuous density. Applying Lemma \ref{lem:Zariski.is.dense}, there is a point $p\in W_i^{sm}(F)\cap \supp(m)$.

Denote the ring of integers of $F$ by $O$. By the implicit function theorem, there are neighborhoods $U_X \subset X(F)$ and $U_Y \subset Y(F)$ of $p$ and $\varphi(x)=\varphi(p)$ respectively, analytic differomorphisms $\alpha_X: U_X \rightarrow O^{\dim X}$, $\alpha_Y: U_Y \rightarrow O^{\dim Y}$, $\alpha_Z:U_X \cap W_i^{sm}(F) \rightarrow O^{\dim Z}$ such that $\alpha_X(p)=0$, $\alpha_{Y}(\varphi(p))=0$, and an analytic map $\psi:O^{\dim X} \rightarrow  O^{\dim Y}$ such that the diagram

$$
\xymatrix
{\,\quad U_{X} \cap W_i(F) \ar[r]^{i} \ar[d]_{\alpha_Z}\ & U_{X} \ar[r]^{\phi|_{U_{X}}} \ar[d]_{\alpha_X}& U_{Y}  \ar[d]_{\alpha_Y}  \\
O^{\dim Z} \ar[r]_{\nu}\ &  O^{\dim X} \ar[r]_{\psi} & O^{\dim Y} }
$$
is commutative, where the inclusion $\nu:O^{\dim Z} \rightarrow O^{\dim X}$ is the first coordinate inclusion and $i$ is the obvious inclusion. Applying a linear map, we can assume that $\ker d \psi |_0= \linspan \left\{ e_1,\ldots,e_{\dim X-\dim Y+r} \right\}$. After an additional analytic differomorphism of $O^{\dim X}$ and $O^{\dim Y}$, we can assume further that $\ker d \psi |_z=\linspan \left\{ e_1,\ldots,e_{\dim X-\dim Y+r} \right\}$ for all $z\in O^{\dim Z}$. Denoting $\mu=(\alpha_X)_* (1_{U_X}m)$, we get that $\mu$ is a Schwartz measure that does not vanish at $0$, and that $\psi_*(\mu)$ has continuous density at $0$. By restricting to a small enough ball around 0 and applying a homothety, we can assume that $\mu$ is the normalized Haar measure.

For every $0<\epsilon<1$, let
\[
A_\epsilon=\left\{ (x_1,\ldots,x_{\dim X}) \in O^{\dim X} \mid |x_i|<\epsilon ^{n_i} \right\}
\]
where
\[
n_i=\left\{ \begin{matrix}  0 & i=1,\ldots,\dim Z \\ 1 & i=\dim Z+1,\ldots,\dim X-\dim Y+r \\ 2 & i=\dim X-\dim Y+r+1\ldots, \dim X \end{matrix} \right . .
\]
We have
\[
\mu(A_\epsilon)=\epsilon ^{\dim X-\dim Y+r-\dim Z+2(\dim Y -r)}=\epsilon ^{\dim X +\dim Y-\dim Z -r} \geq \epsilon ^{2\dim Y-r}.
\]
There is a constant $C$ such that, for any $\epsilon$, $\psi (A_\epsilon)$ is contained in the ball of radius $C\epsilon ^2$ around $0$. The measure of such ball is less than or equal to $(C\epsilon) ^{2\dim Y}$. If we denote the normalized Haar measure on $O^{\dim Y}$ by $\lambda$, we get that $\frac{\psi_* \mu (B(0,C \epsilon))}{\lambda (B(0,C \epsilon))}> \frac{\epsilon ^{-r}}{C^{2\dim Y}}$, so $\psi_* \mu$ does not have a bounded density at $0$ with respect to $\lambda$, contradicting the assumptions.

\end{proof}

As a corollary, we get that $\dim_x Z=\dim X-\dim Y$, so, after passing to a Zariski neighborhood of $x$, the map $\varphi$ is flat at $x$, $Z$ is reduced, and is a local complete intersection (and, therefore, is Gorenstein).

Next, we prove that $x$ is a rational singularity of $Z$. Fix invertible top forms $\omega_X\in \Gamma(X,\Omega_X)$ and $\omega_Y\in \Gamma(Y,\Omega_Y)$ (we may need to pass to a smaller Zariski neighborhoods of $x$ and $\varphi(x)$). Let $t_1,\ldots,t_{\dim Y}$ be local coordinates at $\varphi(x)$, and let $s_i=t \circ \varphi$. Under the isomorphism $\Omega_X \rightarrow \left( N_{Z} ^{X}\right) ^* \otimes \Omega_Z$ from \S\S\S\ref{sec:pi_shrik}, the invertible section $\omega_X$ is mapped to an element of the form $(s_1 \wedge \cdots \wedge s_{\dim Y}) \otimes \eta$ for some invertible section $\eta$ of $\Omega_Z$. By Proposition \ref{prop:shrik_prop} part \ref{prop:shrik_prop:5}, the restriction $\omega_X |_{X^S}$ is mapped to $(s_1 \wedge \cdots \wedge s_{\dim Y}) \otimes \frac{\omega_X}{\varphi ^*(d s_{\dim Y} \wedge \cdots \wedge d s_1)}$ under the isomorphism $\Omega_{X^S} \rightarrow \left( N_{Z\cap X^S}^{X^S} \right)^* \otimes \Omega_{Z \cap X^S}$. Since the isomorphism $\Omega_X \rightarrow \left( N_{Z}^{X} \right) ^* \otimes \Omega_Z$ is compatible with restrictions to open sets, we get $\eta |_{Z \cap X^S}=\frac{\omega_X}{\varphi ^*(d s_{\dim Y} \wedge \cdots \wedge d s_1)}$; denote this last top form by $\omega_Z$. We will apply Corollary \ref{lem:Gorenstein.finite.integral.implies.RS} with the invertible section $\eta$. Fix a finite extension $k'$ of $k$. By assumption, there are a local field $F$ and a Schwartz function $f$ on $X(F)$ that does not vanish at $x$, such that $\varphi_*(f| \omega_X |)$ has a continuous density with respect to $| \omega_Y |$, which we denote by $g:Y(F) \rightarrow \mathbb{R}$. By Corollary \ref{lem:Gorenstein.finite.integral.implies.RS}, it is enough to show that the integral of $f| \omega_Z |$ over $(Z\cap X^S)(F)$ is finite.

Fix some embedding of $X$ into an affine space, and let $d$ be the metric on $X(F)$ induced from the valuation metric. Define a function $h_\epsilon : X(F) \rightarrow \mathbb{R}$ by
\[
h_\epsilon(p)= \left\{ \begin{matrix} 1 & d(p,X^S(F)) \geq \epsilon \\ 0 & d(p,X^S(F)) < \epsilon \end{matrix} \right. ,
\]
and let $g_\epsilon: Y(F) \rightarrow \mathbb{R}$ be the density of the measure $\varphi_*(fh_\epsilon | \omega_X |)$ with respect to $| \omega_Y |$. By Proposition \ref{prop:push.forward.smooth.map}, $g_\epsilon$ is a continuous function and $g_\epsilon(\varphi(x))=\int_{Z(F)} fh_\epsilon | \omega_Z |$. Since $fh_\epsilon | \omega_X | \leq f| \omega_X |$, we get that $g_\epsilon \leq g$. Hence,
\[
\int_{(Z\cap X^S)(F)} f| \omega_Z|=\lim_{\epsilon \rightarrow \infty} \int_{Z(F)} fh_\epsilon | \omega_Z | = g_\epsilon (\varphi(x)) \leq g(\varphi(x)),
\]
so the integral $\int_{(Z\cap X^S)(F)} f| \omega_Z|$ converges.

\section{Representation Growth} \lbl{sec:rep.growth}

\subsection{Special Values of Representation Zeta Functions} \lbl{subsec:special.values}
\setcounter{theorem}{0}

{
In this section, we prove Proposition \ref{prop:Frobenius.Formula}. Using the results of Section \ref{sec:push.forward}, we deduce Theorem \ref{thm:def.and.representation.growth} which relates the representation growths of compact open subgroups of a semi-simple algebraic group over a local field and rational singularities of the deformation variety of the same group. We then deduce Theorem \ref{thm:intro.SLd.FRS} from Theorem \ref{thm:def.and.representation.growth} and the results of Section \ref{sec:singularities.def}.  These theorems imply Theorems \ref{thm:intro.SLd.representation.growth} and \ref{thm:intro.SLd.singularities.deformation}.
}
\begin{proof}[Proof of Proposition \ref{prop:Frobenius.Formula}] Let $\Gamma(i)$ be a decreasing sequence of finite-index normal subgroups of $\Gamma$ that form a basis of neighborhoods of 1. Denote the quotients $\Gamma/\Gamma(i)$ by $\Gamma_i$. Any representation of $\Gamma_i$ gives rise to a representation of $\Gamma$, and $\Irr(\Gamma)$ is the increasing union of the sets $\Irr(\Gamma_i)$. For any $g\in \Gamma$ and $i$, the set $\Phi_{\Gamma,n} ^{-1} (g \Gamma(i))$ is a union of cosets of $\left( \Gamma(i) \right) ^{2n}$. By Frobenius' Theorem (\ref{thm:Frobenius}), the number of these cosets is
\[
| \Gamma_i |^{2n-1} \sum_{\Irr(\Gamma_i)} \frac{\chi(g)}{\chi(1)^{2n-1}}.
\]
We get that
\begin{equation} \lbl{eq:Frobenius.pro.finite}
\frac{\lambda(\Phi_{\Gamma,n} ^{-1} (g \Gamma(i)))}{\lambda(g \Gamma(i))}=\sum_{\Irr(\Gamma_i)} \frac{\chi(g)}{\chi(1)^{2n-1}}.
\end{equation}

We now prove the equivalences: \begin{itemize}
\item[(1) $\Rightarrow$ (2)] is clear.
\item[(2) $\Rightarrow$ (3)] Assuming that $\frac{\lambda(\Phi_{\Gamma,n} ^{-1} (\Gamma(i)))}{\lambda(\Gamma(i))}$ are bounded, we get that the sums $\sum_{\Irr(\Gamma_i)} \frac{1}{\chi(1)^{2n-2}}$ are bounded, and, hence, the sum $\sum_{\Irr(\Gamma)} \frac{1}{\chi(1)^{2n-2}}$ converges.
\item[(3) $\Rightarrow$ (1)] Assume that  the sum $\sum_{\Irr(\Gamma)} \frac{1}{\chi(1)^{2n-2}}$ converges. Since $\left | \frac{\chi(g)}{\chi(1)^{2n-1}} \right | \leq \left | \frac{1}{\chi(1)^{2n-2}}\right |$, we get that $$\Sigma(g):=\sum_{\Irr(\Gamma)} \frac{\chi(g)}{\chi(1)^{2n-1}}$$ is a continuous function. By Corollary \ref{cor:push.forward.dominant.map}, the measure $\left( \Phi_{\Gamma,n} \right)_* \lambda_{\gamma}^{2n}$ has an $L_1$-density, which we denote by $f$. Lebesgue's Density Theorem implies that, for almost all $g\in \Gamma$,
\[
f(g)=\lim_{i \rightarrow \infty}  \frac{\lambda(\Phi_{\Gamma,n} ^{-1} (g \Gamma(i)))}{\lambda(g \Gamma(i))}.
\]

Equation (\ref{eq:Frobenius.pro.finite}) implies that $f(g)=\Sigma(g)$ for almost all $g$, so the continuous function $\Sigma(g)$ is a density for $\left( \Phi_{\Gamma,n} \right)_*\lambda_\Gamma ^{2n}$.
\end{itemize}
\end{proof}

We move on to the proof of Theorem \ref{thm:def.and.representation.growth}. We first show that the map $\Phi_{G,n}$ is flat.

\nir{
The following lemma is a consequence of the elementary properties of restriction and induction:

\begin{lemma} \lbl{lem:commensurability} (\cite[Lemma 3.3]{A}) Suppose that $\Gamma$ is a FAb topological group and $\Delta$ is a finite index subgroup of $\Gamma$. For every $s\in \mathbb{R}_{>0}$, the series $\zeta_{\Gamma}(s)$ converges if and only if the series $\zeta_\Delta(s)$ converges. Moreover,
\[
\frac{1}{[\Gamma : \Delta]^{1+s}}\zeta_\Delta(s) \leq \zeta_\Gamma(s) \leq [\Gamma : \Delta] \zeta_\Delta(s).
\]
\end{lemma}

\begin{theorem}[{\cite[Theorem 1.1\RamiA{,} cf. Section 2]{LS3}}] \lbl{thm:LS} For any simply connected simple algebraic group $G$ defined over a finite field $\mathbb{F}_q$ \RamiA{and any real $s >1$,}
\[
\lim_{m \rightarrow \infty} \zeta_{G(\mathbb{F}_{q^m})}(s) = 1.
\]
\end{theorem}

\begin{corollary} For any simple algebraic group $G$ defined over a finite field $\mathbb{F}_q$, there is a constant $C$ such that, for any integer $m \geq 1$ and any real $s >1$,
\[
\zeta_{G(\mathbb{F}_{q^m})}(s) \leq C.
\]
\end{corollary}

\begin{proof} Let $f:G^{sc} \rightarrow G$ be the simply connected cover of $G$. By Theorem \ref{thm:LS}, there is a constant $D$ such that, for every  \RamiA{$m \geq 1$ and $s >1$},
we have
$\zeta_{f(G^{sc}(\mathbb{F}_{q^m}))}(s) \leq \zeta_{G^{sc}(\mathbb{F}_{q^m})}(s) \leq D$. Since $f(G^{sc}(\mathbb{F}_{q^m}))$ has bounded index in $G(\mathbb{F}_{q^m})$, the result follows from Lemma \ref{lem:commensurability}.
\end{proof}
}

{

\begin{corollary}[{cf. \cite[Corollary 1.11]{LS}}] \lbl{cor:Phi.flat} Let $G$ be a simple algebraic group over a field $E$ and let $n \geq 2$ be an integer. Then  \begin{enumerate}
\item The map $\Phi_{G,n}$ is flat.
\item If $G$ is simply connected, then all the fibers $\Phi_{G,n} ^{-1} (g)$ are irreducible.
\end{enumerate}
\end{corollary}

\begin{proof} It is enough to prove the claims assuming $E=\mathbb{F}_q$ is a finite field. \begin{enumerate}
\item By Theorem \ref{thm:flat_dim}, we need to prove that $\dim \, \Phi_{G,n} ^{-1}(g) \leq (2n-1)\cdot \dim \, G$ for every $g\in G(\overline{\mathbb{F}_q})$.

Fixing $g$, we can enlarge $q$ and assume that $g\in G(\mathbb{F}_q)$. By Frobenius' Theorem (Theorem \ref{thm:Frobenius}), we get that, for any $m$,
\begin{equation} \lbl{eq:size.fibers.flat}
| \Phi_{G,n} ^{-1}(g)(\mathbb{F}_{q^m}) | = | G(\mathbb{F}_{q^m}) |^{2n-1} \cdot \sum_{\chi \in \Irr G(\mathbb{F}_{q^m})} \frac{\chi(g)}{\chi(1)^{2n-1}} \leq
\end{equation}
\[
\leq | G(\mathbb{F}_{q^m}) |^{2n-1} \cdot \sum_{\chi \in \Irr G(\mathbb{F}_{q^m})} \frac{1}{\chi(1)^{2n-2}} \leq C2^{(2n-1)\dim \, G}q^{m(2n-1)\dim \, G},
\]
where the first inequality follows from the inequality $|\chi(g)| \leq \chi(1)$, and the second inequality follows from Theorem \ref{thm:LS} and the inequality $| G(\mathbb{F}_{q^m}) | < 2^{\dim \, G}q^{m \cdot \dim \, G}$ (see \cite[Lemma 3.5]{Nor}). Equation \eqref{eq:size.fibers.flat}, together with the Lang--Weil estimates (\cite[Theorem 1]{LW}), implies that $\dim \, \Phi_{G,n} ^{-1}(g) \leq (2n-1)\dim \, G$.
\item Assuming $G$ is simply connected, the proof above gives
\begin{equation} \label{eq:cor:Phi.flat.2}
\lim_{m \rightarrow \infty} \frac{| \Phi_{G,n} ^{-1}(g)(\mathbb{F}_{q^m}) |}{| G(\mathbb{F}_{q^m}) |^{2n-1}} =1.
\end{equation}
Let $c$ be the number of absolutely irreducible components of $\Phi_{G,n} ^{-1} (g)$. By the Lang--Weil estimates, there are infinitely many natural numbers $m$ such that
\begin{equation} \label{eq:cor:Phi.flat.3}
|\Phi_{G,n} ^{-1} (g)(\mathbb{F}_{q^m})|=c \cdot \left( q^m \right) ^{(2n-1)\cdot \dim \, G}\left( 1+o_m(1) \right) .
\end{equation}
Since $G$ is connected, the Lang--Weil estimates imply that
\begin{equation} \label{eq:cor:Phi.flat.4}
|G(\mathbb{F}_{q^m})|=\left( q^m \right) ^{\dim\, G} \left( 1+ o_m(1) \right).
\end{equation}
Combining \eqref{eq:cor:Phi.flat.2}, \eqref{eq:cor:Phi.flat.3}, and \eqref{eq:cor:Phi.flat.4}, we get that $c=1$.
\end{enumerate}
\end{proof}
}

\begin{proof} [Proof of Theorem \ref{thm:def.and.representation.growth}] $2. \implies 3. \implies 4. \implies 1.$ follow from Theorem \ref{thm:push.forward.detailed} and Proposition \ref{prop:Frobenius.Formula}. We prove $1. \implies 2.$

By Corollary \ref{cor:Phi.flat}, the map $\Phi_{G,n}$ is always flat. It remains to prove that, for any $g\in G(\overline{k})$, the variety $\Phi_{G,n} ^{-1}(g)$ has rational singularities.

Let $g\in G(\overline{k})$, and let $R \subset \overline{k}$ be a subring which is finitely generated over $\mathbb{Z}$ such that $G$ has a model $\underline{G}$ over $R$ and $g\in \underline{G}(R)$. Find a local field $F \supset k$ with ring of integers $O$ such that $R \subset O$. Since $\Phi_{G,n}$ is flat, Elkik's Theorem \ref{thm:elk} implies that $\Phi_{G,n}$ is \FRS\ at $(1,\ldots,1)$, and, therefore, it is \FRS\ in some Zariski neighborhood of $(1,\ldots,1)$. By Theorem \ref{thm:push.forward.detailed}, there is some congruence subgroup $\underline{G}^N(O)$ for which the push forward of the normalized Haar measure on $\left( \underline{G}^N(O) \right) ^{2n}$ under the map $\Phi_{\underline{G}(O),n}$ has continuous density. By Proposition \ref{prop:Frobenius.Formula}, the series $\zeta_{\underline{G}^N(O)}(2n-2)$ converges, and, by Lemma \ref{lem:commensurability}, the series $\zeta_{\underline{G}(O)}(2n-2)$ converges. Applying Proposition \ref{prop:Frobenius.Formula} again, the push forward of the normalized Haar measure on $\left( \underline{G}(O) \right) ^{2n}$ under the
map $\Phi_{\underline{G}(O),n}$ has continuous density. Theorem \ref{thm:push.forward.detailed} implies that $\Phi_{G,n}$ is \FRS\ at $g$.
\end{proof}
{We move on to the proof of Theorem \ref{thm:intro.SLd.FRS}. We start with some preparations. The following theorem follows from the proof of \cite[Corollary 4.4]{LM}:
\begin{theorem} \label{thm:LM} Let $F$ be a local field of characteristic 0, let $G$ be a connected and simply connected algebraic group defined over $F$, and let $\Gamma$ be a compact open subgroup of $G(F)$. Then there is a constant $C$ such that $R_n(\Gamma) \leq Cn^{3\dim(G)}$
\end{theorem}

\begin{lemma} \label{lem:isogeny} Let $G$ and $H$ be algebraic groups over $k$, and let $\varphi:G \to H$ be a homomorphism with finite kernel and open image. For every integer $n \geq 1$, the following conditions are equivalent: \begin{enumerate}
\item $\Phi_{G,n}$ is \FRS.
\item $\Phi_{H,n}$ is \FRS.
\end{enumerate}
\end{lemma}

\begin{proof} Without loss of generality, we can assume that $k$ is finitely generated. Using Theorem \ref{thm:def.and.representation.growth}, it is enough to show that, for every non-archimedean local field $F$ containing $k$, there is a compact open subgroup $\Gamma \subset G(F)$ such that the restriction of $\varphi$ to $\Gamma$ gives an isomorphism from $\Gamma$ to an open subgroup of $H(F)$. By assumption, $\varphi$ is etale. By the implicit function theorem, the restriction of $\varphi$ to some open neighborhood $U$ of 1 in $G(F)$ is an open and injective map. Now let $\Gamma$ be any compact open subgroup that is contained in $U$.
\end{proof}

The following lemma is immediate:
\begin{lemma} \label{lem:FRS.product} Let $G$ and $H$ be algebraic groups over $k$, and let $n \geq 1$ be an integer. If $\Phi_{G,n}$ and $\Phi_{H,n}$ are \FRS, then $\Phi_{G \times H,n}$ is \FRS.
\end{lemma}

\begin{proof}[Proof of Theorem \ref{thm:intro.SLd.FRS}] Without loss of generality, we can assume that $k$ is algebraically closed. Applying Lemma \ref{lem:isogeny}, we can assume that $G$ is simply connected, and, thus, is a product of simple factors. Applying Lemma \ref{lem:FRS.product}, we can assume that $G$ is simple. If $G$ is classical, the assertion follows from Theorems \ref{thm:local.FRS} and \ref{thm:def.and.representation.growth}, as well as Lemma \ref{lem:isogeny}. If $G$ is exceptional, the assertion follows from Theorem \ref{thm:LM} and Theorem \ref{thm:def.and.representation.growth}.
\end{proof}

}
\subsection{Representation Growth in Positive Characteristic} \lbl{subsec:pos.char}
\setcounter{theorem}{0}

In this subsection, we prove Theorem \ref{thm:intro.pos.zero.local.zeta} from the introduction. Recall the formulation:

\begin{theorem*} Let $G$ be an affine group scheme over a localization of $\mathbb{Z}$ by finitely many primes whose generic fiber is semi-simple, and fix an integer $n \geq 1$. There is a constant $p_0$ such that, if $F_1,F_2$ are local fields with isomorphic residue fields of characteristic greater than $p_0$, and if $O_1,O_2$ denote the rings of integers of $F_1$ and $F_2$, then $\alpha(G(O_1))<2n$ if and only if $\alpha(G(O_2))<2n$. Moreover, in this case, $\zeta_{G(O_1)}(2n)=\zeta_{G(O_2)}(2n)$.
\end{theorem*}

\begin{proof} We will use the model theory of Henselian valued fields. See for example \cite{CL}. Choose a translation-invariant top differential form $\omega$ on $G$. If $p$ is large enough, then, for every local field $F$ of residue characteristic $p$, the measure $| \omega |_F$ is a Haar measure on $G(F)$. If we denote the $2n$ projections from $G^{2n}$ to $G$ by $pr_1,\ldots,pr_{2n}$ and define
\[
\omega ^{\fbox{$\wedge$}2n}:=pr_1^* \omega \wedge \ldots \wedge pr_{2n}^* \omega,
\]
then, if $p$ is large enough, the measure $\left | \omega ^{\fbox{$\wedge$}2n} \right |_F$ is a Haar measure on $G(F)^{2n}$.

Consider $G \subset \GL_d$ as a definable set in the theory of Henselian valued fields, and define $G_O$ to be the intersection of $G$ with the definable set $\left\{ (x_{i,j})\in \GL_d \mid \val(x_{i,j}) \geq 0 \right\}$. Under this definition, $G_O(F_1)=G(O_1)$ and $G_O(F_2)=G(O_2)$. Consider the first-order formula
\[
\phi(g_1,h_1,\ldots,g_n,h_n,\gamma):= \left( g_i,h_i\in G_\mathcal{O} \wedge \val \left( \Phi_{G,n}(g_1,\ldots,h_n) \right) \geq \gamma \right)
\]
on $G_O^{2n} \times \Gamma$, where $\Gamma$ is the value group. By \cite[Theorem 7.6]{CL}, if $p$ is large enough, then \begin{enumerate}
\item $\int_{G^m(O_1)} |\omega|_{F_1} = \int_{G^m(O_2)} |\omega|_{F_2}$.
\item For every $m\in \mathbb{Z}$,
\[
\int_{\left\{ (g_1,\ldots,h_{n}) \mid F_1 \models \phi(g_1,\ldots,h_n,m) \right\} } \left | \omega ^{\fbox{$\wedge$} 2n} \right |_{F_1}=\int_{\left\{ (g_1,\ldots,h_{n}) \mid F_2 \models \phi(g_1,\ldots,h_n,m) \right\} } \left | \omega ^{\fbox{$\wedge$} 2n} \right |_{F_2} .
\]
\end{enumerate}

If we denote the normalized Haar measure of $G(O_1)$ by $\mu$, and the normalized Haar measure of $G(O_2)$ by $\nu$, then the equalities above imply that, for every $m$,
\[
\frac{\left( \Phi_{G,n} \right)_* \mu ^{2n}(G^m(O_1))}{\mu(G^m(O_1))}=\frac{\left( \Phi_{G,n} \right)_* \nu ^{2n}(G^m(O_2))}{\nu(G^m(O_2))}
\]
By Proposition \ref{prop:Frobenius.Formula}, the claim follows.
\end{proof}

\subsection{Volumes of moduli spaces of local systems} \lbl{subsubsec:localization}
\setcounter{theorem}{0}
In this section, we recall the definition of the Atiyah--Bott--Goldman form, and prove Theorem \ref{thm:intro.volume.formula}. We fix a natural number $n$ and a semi-simple algebraic group $G$ defined over $k$ (assumed to have characteristic 0).

\begin{lemma} \lbl{lem:G.open} 
$ $ \begin{enumerate}
\item \label{lem:G.open:1}There is a (unique) Zariski-open subset $(G^{2n})^{open} \subset G^{2n}$ such that $(G^{2n})^{open}(\overline{k})$ is the set of all $2n$-tuples that generate a Zariski dense subgroup.
\item \label{lem:G.open:2}The restriction of $\Phi_{G,n}$ to $(G^{2n})^{open}$ is smooth.
\item \label{lem:G.open:3} The action of $G/Z(G)$ on $(G^{2n})^{open}$ is free.
\item \label{lem:G.open:4} Any $G$-orbit in $(G^{2n})^{open}$ is closed in $G^{2n}$
\end{enumerate}
\end{lemma}

For the proof of (4) we will need the following standard lemma:
\begin{lem}\label{lem:geo.Frob}
Let $H$ be an algebraic group over $k$, and let $\phi:X\to Y$ be an $H$-equivariant map of $H$-algebraic varieties. Assume that the action of $H$ on $Y$ is transitive. Let $y \in Y(\overline{k})$ be a point and $Z \subset X(\overline{k})$ be an $H(\overline{k})$-invariant subset. Then, the following conditions are equivalent:
\begin{enumerate}
\item $Z$ is Zariski closed in $X(\overline{k})$.
\item $Z \cap \phi^{-1}(y)$ is Zariski closed in $\phi^{-1}(y)(\overline{k})$.
\end{enumerate}
\end{lem}

\begin{proof}$ $
\begin{enumerate}[Step 1.]
\item The case when $Y=H.$\\
In this case, $X \cong  \phi^{-1}(y) \times H$. Under this identification, $Z$ corresponds to $Z \cap \left(\phi^{-1}(y) \times H\right)(\overline{k})$, so the assertion is obvious.
\item The general case.\\
Let $\pi:H \to Y$ be the map given by the action on $y$. Consider the Cartesian square
$$\xymatrix{
X' \ar@{->}_{\phi'}[d] \ar@{->}^{\pi'}[r]& X \ar@{->}^\phi[d]\\
H\ar@{->}^\pi[r]& Y
.}$$

Since $\pi$ is a submersion, so is $\pi'$. Therefore it is enough to prove the lemma for the map $\phi'$. This follows from the previous step.
\end{enumerate}
\end{proof}

\begin{proof}[Proof of Lemma \ref{lem:G.open}] Without loss of generality, we can assume that $k$ is algebraically closed. Let $\mathfrak{g}$ be the Lie algebra of $G$. \begin{itemize}
\item[Proof of \eqref{lem:G.open:1}] Let $(G^{2n}(k))^{open}\subset G^{2n}(k)$ be the set of all $2n$-tuples that generate a Zariski-dense subgroup. {We need to show that $(G^{2n}(k))^{open}$ is Zariski open.}  Let $(G^{2n})^{irr} \subset G^{2n}$ be the collection of tuples $(g_1,\ldots,g_{2n})$ such that the group generated by the $g_i$ acts irreducibly on $\mathfrak{g}$. Clearly, $(G^{2n})^{open} \subset (G^{2n})^{irr}$. We claim that $(G^{2n})^{irr}$ is Zariski open. Indeed, let $\Grass(\mathfrak{g})$ be the union of all Grassmannian varieties of $\mathfrak{g}$, and let $Y \subset G^{2n} \times \Grass(\mathfrak{g})$ be the closed subvariety
\[
Y=\left\{ (g_1,\ldots,g_{2n},V)\in G^{2n} \times \Grass(\mathfrak{g}) \mid (\forall i)\, g_i \cdot V=V \right\}.
\]
$(G^{2n})^{irr}$ is the complement of the image of the projection of $Y$ to $G^{2n}$, so it is open.

To finish the proof, we will find a Zariski closed subset $X \subset G^{2n}$ such that $(G^{2n})^{open}=(G^{2n})^{irr} \smallsetminus X$. \nir{Fix an embedding $G \hookrightarrow \SL_d$ for some $d$.} A theorem of Jordan (\cite[Theorem 36.13]{CR}) states that there is a constant $C=C(d)$ such that, if $H \subset \SL_d(k)$ is a finite subgroup, then $H$ has a normal abelian subgroup of index at most $C$, and, therefore, there is a homomorphism $H \rightarrow S_{C!}$ with an abelian kernel. Denote the free group on $2n$ generators by $F_{2n}$. For any homomorphism $\rho:F_{2n} \rightarrow S_{C!}$, let
\[
X_{\rho}= \left\{ (g_1,\ldots,g_{2n}) \mid \forall w_1,w_2\in \Ker(\rho) \quad [w_1(g_1,\ldots,g_{2n}),w_2(g_1,\ldots,g_{2n})]=1 \right\}.
\]
Each $X_\rho$ is a closed subvariety of $G^{2n}$ and, since there are only finitely many possibilities for $\rho$, the union of all $X_\rho$ is also a closed subvariety of $G^{2n}$. We denote this union by $X$. The discussion above implies that if $(g_1,\ldots,g_{2n})\in G^{2n}(k)$ generates a finite subgroup, then $(g_1,\ldots,g_{2n})\in X(k)$. We claim that $(G^{2n})^{open}=(G^{2n})^{irr} \smallsetminus X$, which will prove (1).

If $g\in X_\rho$, then the subgroup generated by $w_i(g_1,\ldots,g_{2n})$ for $w_i\in \Ker(\rho)$, is abelian and has finite index in the subgroup generated by the $g_i$s. In particular, the subgroup generated by the $g_i$s cannot be Zariski open. Together with the inclusion $(G^{2n})^{open} \subset (G^{2n})^{irr}$, we get $(G^{2n})^{open} \subset (G^{2n})^{irr} \smallsetminus X$.

Conversely, if $(g_1,\ldots,g_{2n})\in (G^{2n})^{irr} \smallsetminus X$, let $H$ be the Zariski closure of the subgroup generated by the $g_i$. The Lie algebra of $H$ is invariant under all $g_i$, so it is equal to either $\mathfrak{g}$ or $0$. If $\Lie(H)=0$ then $H$ is finite, contradicting the assumption that $(g_1,\ldots,g_{2n})\notin X$. If $\Lie(H)=\mathfrak{g}$ then $H=G$, so $(g_1,\ldots,g_{2n})\in (G^{2n})^{open}$. This shows that $(G^{2n})^{open} \supset (G^{2n})^{irr} \smallsetminus X$.

\item[Proof of \eqref{lem:G.open:2}] After identifying the tangent space to $G$ at $g$ with the Lie algebra $\mathfrak{g}$ via the map $X \in \mathfrak{g} \mapsto gX \in T_gG$, the differential of the commutator map is
\[
d[\cdot,\cdot]|_{(g,h)}: (X,Y) \mapsto X^{hgh ^{-1} }-X^{hg}+Y^{h g}-Y^h,
\]
where we denote $X^a=aXa ^{-1}$. Therefore, the derivative of $\Phi_{G,n}$ at a point $(g_1,h_1,\ldots,g_n,h_n)$ is the map
\begin{equation} \lbl{eq:derivative.of.Phi}
(X_1,Y_1,\ldots,X_n,Y_n) \mapsto \sum_{i=1}^n \left( X^{g_ih_i ^{-1} }-X^{g_i} +Y^{g_i}-Y\right) ^{P_i ^{-1} h_i},
\end{equation}
where we denote $P_i=[g_{i+1},h_{i+1}]\cdots[g_n,h_n]$.

Consider the Killing form on $\mathfrak{g}$. The image of the map $X \mapsto X^g-X $ is the orthogonal complement to the centralizer of $g$ in $\mathfrak{g}$. Assume that $Z\in \mathfrak{g}$ is orthogonal to the image of $d \Phi_{G,n}$. Taking the summand $i=n$ in (\ref{eq:derivative.of.Phi}), we get that {$Z^{h_n ^{-1}}$} commutes with both $g_n$ and $g_nh_n ^{-1}$. Therefore, $Z$ commutes with both $g_n$ and $h_n$. Continuing by decreasing induction, we get that the orthogonal complement to the image of $d \Phi_{G,n}$ is the common centralizer of $g_i,h_i$. By assumption, it is trivial. This proves (2).

\item[Proof of \eqref{lem:G.open:3}] The assertion follows from the fact that the stabilizer of a tuple $(g_1,\ldots,g_{2n})\in G^{2n}$ is equal to the centralizer of the Zariski closure of the subgroup generated by the $g_i$.

{
\item[Proof of \eqref{lem:G.open:4}] Let $x=(g_1,\ldots,g_{2n}) \in (G^{2n})^{open}$. {Recall that we denote the free group on $2n$ generators by $F_{2n}$. For every $\gamma \in F_{2n}$, let $\phi_\gamma :G^{2n} \rightarrow G$ be the substitution map. The tuple $x$ induces a map $\psi_x: F_{2n} \to G$ sending $\gamma$ to $\phi_\gamma(x)$, and the image of $\psi_x$ is Zariski dense.} Choose $\gamma_1$ such that $\psi_x(\gamma_1)$ is regular semi-simple. Without loss of generality, we may assume that $\psi_x(\gamma_1)$ is diagonal matrix. Choose $\gamma_2$ such that all entries of $\psi_x(\gamma_2)$  are non-zero.

Define $$\phi_{\gamma_1,\gamma_2}:=\phi_{\gamma_1}\times \phi_{\gamma_2}:G^{2n} \to G\times G,$$
and let $\cO:=G \cdot x$ be the orbit of $x$.
\begin{enumerate}[Step 1.]
\item $\phi_{\gamma_1}(\cO)$ is closed.\\
$\phi_{\gamma_1}(\cO)=G \cdot \phi_{\gamma_1}(x)=G \cdot \psi_{x}(\gamma_1)$ which is closed, since it is the conjugacy class of the semi-simple element $\psi_{x}(\gamma_1)$.
\item $\phi_{\gamma_1,\gamma_2}(\cO)$ is closed. \\
Let $y:=(a,b):=\phi_{\gamma_1,\gamma_2}(x)$. We have $\phi_{\gamma_1,\gamma_2}(\cO)=G \cdot y$. Lemma \ref{lem:geo.Frob} implies that, in view of the previous step, the fact that $G \cdot y$ is closed in $G \times G$ follows from   the fact that $G \cdot y \cap \{a\} \times G$ is closed in $\{a\} \times G$. For the later, we use the equality  $$G \cdot y \cap \{a\} \times G=T\cdot b=T\cdot \psi_{x}(\gamma_2),$$ where $T:=Z_G(a)$ is the standard torus. The assertion now follows from the fact that all the entries of $\psi_x(\gamma_2)$  are non-zero.

\item $\cO$ is closed.\\
By lemma \ref{lem:geo.Frob}, the previous step implies that it is enough to show that $\cO \cap \phi_{\gamma_1,\gamma_2}^{-1}(a,b)$ is closed in  $\phi_{\gamma_1,\gamma_2}^{-1}(a,b)$. This is obvious, since $\cO \cap \phi_{\gamma_1,\gamma_2}^{-1}(a,b)=\{x\}.$
\end{enumerate}
}

\end{itemize}
\end{proof}

\begin{definition} Denote
\[
\Def_{G,n}^{open}:=(G^{2n})^{open} \cap \Phi_{G,n} ^{-1}(1).
\]
\end{definition}

By Lemma \ref{lem:G.open}\eqref{lem:G.open:2}, $\Def_{G,n}^{open}$ is a smooth variety.
{
By Lemma \ref{lem:G.open}\eqref{lem:G.open:3} and Luna Slice Theorem (see, e.g. \cite{DR}), the geometric quotient $$\mathcal{M}_G:=\left( (G^{2n})^{open} \cap \Phi_{G,n} ^{-1}(1) \right) / G$$ exists, and is a smooth subvariety of the categorical quotient $\left( G^{2n} \cap \Phi_{G,n} ^{-1} (1)\right) / G$.

Let $\pi$ be the fundamental group of a compact surface of genus $n$. We can think of elements of $\mathcal{M}_G$ as morphisms in  $\Hom(\pi,G)$  up to conjugacy.  Standard deformation theory shows that the tangent space of $\mathcal{M}_G$ at a point corresponding to $\rho \in \Hom(\pi,G)$ is the cohomology group $H^1(\pi,\mathfrak{g}_\rho)$, where $\mathfrak{g}_\rho$ denotes the representation of $\pi$ on the space $\mathfrak{g}$ given by $\Ad \circ \rho:\pi \rightarrow \GL(\mathfrak{g})$. Assuming that $\rho \in ((G ^{2n})^{open} \cap \Phi_{G,n} ^{-1} (1))(k)$, both $\Ad \circ \rho$ and its dual are irreducible, so the cohomology groups $H^0(\Sigma,\mathfrak{g}_\rho)$ and $H^2(\pi,\mathfrak{g}_\rho)$ vanish. We get that $\dim \left( H^1(\pi,\mathfrak{g}_\rho) \right)$ is equal to negative the Euler characteristic of the local system $\mathfrak{g}_\rho$, which is $(2n-2)\dim\ \mathfrak{g}$. We conclude that $\dim \mathcal{M}_G=(2g-2)\dim\ \mathfrak{g}$.

\begin{definition} Let $\langle \cdot,\cdot \rangle$ be the Killing form on $\mathfrak{g}$.
\begin{enumerate}
\item
The \emph{Atiyah--Bott--Goldman form} on $\mathcal{M}_G$  is the differential 2-form whose value at a point corresponding to $\rho$ is
\[
\eta_{ABG}:H^1(\pi,\mathfrak{g}_\rho) \times H^1(\pi,\mathfrak{g}_\rho) \rightarrow H^2 \left( \pi,\mathfrak{g}_\rho\otimes \mathfrak{g}_\rho \right) \rightarrow H^2(\pi,\mathfrak{g}_\rho)=k,
\]
where the first arrow is the cup product and the second arrow is composition with $\langle \cdot,\cdot \rangle$.
\item Define  a top form on $\cM_G$:
$$v_{ABG}:=\frac{\eta_{ABG}^{\wedge \dim \mathcal{M}_G/2}}{(\dim \mathcal{M}_G / 2)!}.$$
\end{enumerate}

\end{definition}
}

We now define a volume form on $\Def_{G,n}^{open}$. Let $Q:\Def_{G,n}^{open} \to \cM_{G}$ be the quotient map. For any $\rho\in\Def_{G,n}^{open}$, identify the tangent space to the fiber $Q^{-1}(Q(\rho))$ with $\mathfrak{g}$ via the action map. We get a short exact sequence
\[
0 \to \mathfrak{g} \to T_{\rho}\Def_{G,n}^{open} \to T_{Q(\rho)}\cM_{G} \to 0.
\]
This gives us an identification of the sheaf of relative top forms $\Omega_{\Def_{G,n}^{open} / \cM_{G}}$ with the free sheaf on $\Def_{G,n}^{open}$ with fiber $\bigwedge^{\dim \mathfrak{g}}\mathfrak{g}^*$.

\begin{definition} Let $\omega\in \bigwedge^{\dim \mathfrak{g}}\mathfrak{g}^*$, and let $\omega'\in\Omega_{\Def_{G,n}^{open} / \cM_{G}}$ be the corresponding section. Let
\[
\nu_{\omega}:=\omega' \otimes Q^* (v_{ABG})\in \Omega_{\Def_{G,n}^{open} / \cM_{G}} \otimes Q^*(\Omega_{\cM_{G}})\cong \Omega_{\Def_{G,n}}.
\]
\end{definition}

{Assume now that $G$ is defined over a local field $F$ of characteristic zero. If $F$ is archimedean, assume in addition that $F=\mathbb{R}$ and that $G(\mathbb{R})$ is compact. Let $\Gamma$ be an open compact group in $G(F).$}
Let $(\Gamma ^{2n})^{open}=(G^{2n})^{open}(F)\cap \Gamma^{2n}$, and let $\Def_{\Gamma,n}^{open}=(\Gamma^{2n})^{open}\cap\Phi_{G,n}^{-1}(1)$. In order to show that the quotient $\Def_{\Gamma,n}^{open}/\Gamma$ is an 
analytic
 manifold, we recall a general criterion for a quotient to be a manifold.

Let $(X,\mathcal{O})$ be an analytic manifold, and let $K$ be a compact analytic group acting freely on $X$. Consider the quotient $X/K$, the quotient map $\pi : X \rightarrow X/K$, and the sheaf $\mathcal{Q}$ on $X/K$ given by $\mathcal{Q}(A)=\mathcal{O}( \pi ^{-1} (A))^K$. The following lemma is standard, see e.g. \cite[Proposition 2.2.2]{CDS} for $C^ \infty$ manifolds. We sketch a proof {which is valid for the real and p-adic analytic cases.}
\begin{lemma} \label{lem:Rami.made.me} The ringed space $(X/K,\mathcal{Q})$ is an analytic manifold.
\end{lemma}

\begin{proof} Since $K$ is compact, the quotient $X/K$ is Hausdorff. We need to show that the sheaf $\mathcal{Q}$ is locally isomorphic to the sheaf of analytic functions on an analytic manifold. Fix $x\in X$ and consider the orbit map $a: K \rightarrow X$, $a(k)=k \cdot x$. The map $a$ is an immersion, is one-to-one, and its domain is compact. Therefore, its image $K \cdot x$ is a submanifold. Let $S \subset X$ be a $(\dim X-\dim K)$-dimensional submanifold that contains $x$ and is transversal to $K \cdot x$ at $x$. By compactness, there is a neighborhood $U$ of $x$ such that, for any $y\in U \cap S$, the following two properties hold: \begin{enumerate}
\item The intersection $K \cdot y \cap S \cap U$ is equal to $\left\{ y \right\}$.
\item $K \cdot y$ is transverse to $S$ at $y$.
\end{enumerate}
It follows that the action map $A:K \times (S \cap U) \rightarrow X$ is an analytic diffeomorphism onto an open submanifold of $X$. The map $p\in S\cap U \mapsto A(1,p)$ induces an isomorphism between the restriction of $\mathcal{Q}$ to $\pi(K \cdot (U \cap S))$ and the sheaf of analytic functions on $U \cap S$.
\end{proof}
{
By Lemma \ref{lem:Rami.made.me} and Lemma \ref{lem:G.open}\eqref{lem:G.open:3}, the quotient $$\mathcal{M}_\Gamma:=\Def_{\Ga,n}^{open} / \Gamma=\Def_{\Ga,n}^{open} / (\Gamma/Z(\Gamma))$$ is an analytic manifold. {We have a natural \et map  $r:\mathcal{M}_\Gamma \rightarrow \mathcal{M}_G(F)$. Thus we can consider the Atiyah--Bott--Goldman form $\eta_{ABG}$ and the forms $v_{ABG},\nu_{\omega}$ as forms on $\mathcal{M}_\Gamma,\cM_{\Ga}$, and $\Def_{\Ga,n}^{open}$.}


{Any compact semi-simple (real) Lie group $\Gamma$ is the real points of some semi-simple algebraic group. Therefore, we can discuss $\mathcal{M}_\Gamma$ and the Atiyah--Bott--Goldman form on it. Originally, the Atiyah--Bott--Goldman form was introduced in this setting, and the following results are known:
\begin{theorem}[\cite{Gol}] \label{thm:Gol.real}
Let $\Gamma$ be a compact semi-simple Lie group. Then $\eta_{ABG}$ is a symplectic form on $\cM_\Gamma$.
\end{theorem}
}


\nir{We say that a volume form on a $k$-vector space $V$  is compatible with a non-degenerate symmetric bilinear form $\langle \cdot,\cdot \rangle$ if, for any basis $e_i$ of $V\otimes \overline k$ such that $\langle e_i,e_j \rangle = \delta_{i,j}$, the volume of $e_1 \wedge \cdots \wedge e_n$ is $\pm 1$. Such volume form need not exist in general, but, if $k$ is algebraically closed there are exactly two such volume forms.}


\begin{theorem}[\cite{Wi}] \label{thm:Wit.real}
{Let $\Gamma$ be a compact semi-simple Lie group. }
\nir{Let $\omega$ be a translation invariant $\overline{F}$-valued top differential form on $\Gamma$  whose value at 1 is compatible with the Killing form.} Then
\[
\nu_{\omega}= \pm \left. \frac{\omega ^{\fbox{$\wedge$}2n}}{\Phi_{G,n} ^* \omega }\right|_{\Def_{\Ga,n}^{open}}.
\]
\end{theorem}
{The above results (Theorems \ref{thm:Gol.real} and \ref{thm:Wit.real}) are also valid  in the algebraic case:

\begin{theorem} \label{thm:symp.algebraic} Let $G$ be a semi-simple group defined over $k$. Then
\begin{enumerate}
\item $\eta_{ABG}$ is a symplectic form on $\cM_G$.
\item\label{thm:symp.algebraic:2} Let $\omega$ be a translation invariant top differential form on $G$  whose value at 1 is compatible with the Killing form. Then
$$
\nu_{\omega}= \pm \left. \frac{\omega ^{\fbox{$\wedge$}2n}}{\Phi_{G,n} ^* \omega}\right|_{\Def_{G,n}^{open}}$$
\end{enumerate}
\end{theorem}
}

The proofs in \cite{Gol,Wi} are easily adaptable to the algebraic case. For brevity, we will not repeat the argument, but rather deduce Theorem \ref{thm:symp.algebraic} from Theorems \ref{thm:Gol.real} and \ref{thm:Wit.real} when $G$ is connected and simply connected. From this point on, we assume this is the case. We will use the following preparations:

\begin{lemma} \label{lem:dense.points.local.field}
Let $X$ be a smooth irreducible algebraic variety defined over a (possibly archimedean) local field $F$ of characteristic 0. Assume that $X(F)$ is non empty. Then $X(F)$ is Zariski dense in $X$
  \end{lemma}
  \begin{proof}
Let $x \in X(F)$. Passing to an open neighborhood we can assume that $X$ is affine and we have an \et map $\phi:X\to \A^n$ that maps $x$ to $0$. Let $Z$ be the Zariski closure of $X(F)$. Using the implicit function theorem, we see that $\phi(Z)(F)$ contains some open neighborhood of $0$ in $F^n=\A^n(F)$. Thus, $\phi(Z)$ is dense in $\A^{n}$, so $\dim Z=\dim(\phi(Z))=n=\dim X.$ This implies $Z=X$.
  \end{proof}

\begin{lemma}(\cite[Proposition 3.3]{GM}) \label{lem.M.Gamma.non.empty}
Let $\Gamma$ be a compact semi-simple Lie group. Then $\cM_\Gamma$ is not empty
\end{lemma}

\begin{proof}[Proof of Theorem \ref{thm:symp.algebraic} for simply connected groups] \nir{Without loss of generality, $k=\mathbb{C}$. We choose a real structure on $G$ so that $G(\mathbb{R})$ is compact.} Since we assume that $G$ is simply connected, Corollary \ref{cor:Phi.flat} implies that $\mathcal{M}_G$ is irreducible. By Lemma \ref{lem.M.Gamma.non.empty}, $\mathcal{M}_{G(\mathbb{R})}$ is non-empty, and, hence, $\mathcal{M}_G(\mathbb{R})$ is non-empty. By Lemma \ref{lem:dense.points.local.field}, the set $\mathcal{M}_G(\mathbb{R})$ is Zariski dense in $\mathcal{M}_G$. The second statement now follows from Theorem \ref{thm:Wit.real}, as well as the claim that the form $\eta_{ABG}$ is closed. Since it is also anti-symmetric and non-degenerate (which follows from the properties of the Poincare duality map), we get the first claim as well.
\end{proof}

We can now prove Theorem \ref{thm:intro.volume.formula}

\begin{proof}[Proof of Theorem \ref{thm:intro.volume.formula}]
By Theorem \ref{thm:def.and.representation.growth}, the series $\zeta_{\Gamma}(2n-2)$ converges.
Let $\lambda_{\Gamma}$ be the normalized Haar measure on $\Gamma$.  By Proposition \ref{prop:Frobenius.Formula},
\[
\zeta_{\Gamma}(2n-2)=\frac{(\Phi_{G,n})_*\lambda_{\Gamma^{2n}}}{\lambda_\Gamma}(1).
\]

The restriction of $|\omega|$ to $\Gamma$ is equal to $\left( |\omega|(\Gamma) \right) \cdot \lambda_{\Gamma }$. Thus,

$$\left(|\omega|(\Gamma)\right)^{2n-1}\frac{(\Phi_{G,n})_*\lambda_{\Gamma^{2n}}}{\lambda_\Gamma}(1)=\frac{(\Phi_{G,n})_*\left(1_{\Ga^{2n}}\left | \omega ^{\fbox{$\wedge$}2n}\right|\right)}{|\omega|}(1) $$
Let $\Def_{G,n}^S \subset \Def_{G,n}$ be the intersection of $\Def_{G,n}$ with the smooth locus of the map $\Phi_{G,n}$.
By Theorem \ref{thm:push.forward.detailed}, we have
\[
\frac{(\Phi_{G,n})_*\left(1_{\Ga^{2n}}\left | \omega ^{\fbox{$\wedge$}2n}\right|\right)}{|\omega|}(1)=\int_{\Def_{G,n}^S(F) \cap \Ga^{2n}} \left| \frac{\omega ^{\fbox{$\wedge$}2n}}{\Phi_{G,n} ^* \omega}\right|.
\]
By Lemma \ref{lem.M.Gamma.non.empty}, $\Def_{G,n}^{open}$ is non-empty, and, since $\Def_{G,n}$ is irreducible (Corollary \ref{cor:Phi.flat}), we get that
$$\dim\ \left( \Def_{G,n}\smallsetminus\Def_{G,n}^{open} \right)<\dim\ \Def_{G,n}.$$
Since $\Def_{G,n}^{open} \subset \Def_{G,n}^S$, we get that
\[
\int_{\Def_{G,n}^S(F) \cap \Ga^{2n}} \left| \frac{\omega ^{\fbox{$\wedge$}2n}}{\Phi_{G,n} ^* \omega}\right|=\int_{\Def_{\Ga,n}^{open}}\left| \frac{\omega ^{\fbox{$\wedge$}2n}}{\Phi_{G,n} ^* \omega}\right|.
\]
By Theorem \ref{thm:symp.algebraic}\eqref{thm:symp.algebraic:2},
\[
\int_{\Def_{\Ga,n}^{open}}\left| \frac{\omega ^{\fbox{$\wedge$}2n}}{\Phi_{G,n} ^* \omega}\right|=\int_{\Def_{\Gamma,n}^{open}}\left|\nu_{\omega}\right|=\left(\int_{\Gamma / Z(\Gamma)}|\omega|\right)\left(\int_{\cM_{\Gamma}}|v_{ABG}|\right).
\]

We conclude that
\[
\int_{\cM_{\Gamma}}|v_{ABG}|=\frac{\left(|\omega|(\Gamma)\right)^{2n-1}\zeta_{\Gamma}(2n-2)}{\int_{\Gamma / Z(\Gamma)}|\omega|}=|Z(\Gamma)|(|\omega|(\Gamma))^{2n-2}\zeta_{\Gamma}(2n-2).
\]

\end{proof}

\appendix
\section{Continuity Along Curves} \lbl{sec:cont.along.curves}
\setcounter{theorem}{0}

In this appendix, we prove the following theorem:

\begin{theorem} \label{thm:continuity.along.curves} Let $\pi :X \rightarrow Y$ be a morphism between algebraic varieties over a local field $F$, and let $U \subset X$ be an open set on which $\pi$ is smooth. Suppose that $\omega \in \Gamma(U,\Omega_{U/Y})$, and that $\varphi :X(F) \rightarrow \mathbb{C}$ is a Schwatz function, and that, for any $y\in Y(F)$, the integral $I(y)=\int_{\pi ^{-1} (y)\cap U(F)} \varphi \cdot | \omega |$ is finite. Assume that, for any curve $C \subset Y$, the restriction $I|_{C(F)}$ is continuous. Then $I$ is continuous.
\end{theorem}

The proof is divided into two parts. In the first, we define a notion of constructible function, and show that $I$ is constructible (Proposition \ref{prop:integration.is.constructible}). In the second, we show that a constructible function is continuous if and only if its restriction to any curve is continuous (Proposition \ref{prop:continuity.along.curves}).

\subsection{Constructible Functions}

In this section, we will use notions and results from the model theory of valued fields. For the standard notions of Model Theory, see \cite{CK}. Fix a finitely generated field $k$ of characteristic 0. Let $\mathcal{L}$ be the first-order language with two sorts, which we denote by $\VF$ (for valued field) and $\VG$ (for value group) and \begin{enumerate}
\item Four function symbols: $+_{\VF} : \VF \times \VF \rightarrow \VF$, $\cdot_{\VF} : \VF\times \VF \rightarrow \VF$, $+_{\VG} : \VG \times \VG\rightarrow \VG$, and $\val: \VF \rightarrow \VG$.
\item Relation symbols: a binary relation $\leq$ on $\VG$, and, for each $n$, unary relations $P_n$ and $p_n$ on $\VF$ and $\VG$.
\item A constant in $\VF$ for each element of $k$.
\end{enumerate}

For each local field $F$ containing $k$, we interpret the sort $\VF$ as $F$, the sort $\VG$ as $\mathbb{Z}$, the functions $+,\cdot,\val$ as addition, multiplication, and valuation\footnote{The value $\val(0)$ is irrelevant and can be taken to be arbitrary since every two choices will give rise to bi-interpretable models.}, the relation $\leq$ as order, and the relations $P_n(x)$ (respectively, $p_n$) as saying that $x$ is an $n$-th power (respectively, divisible by $n$). We denote the theory of $F$, i.e., the set of all sentences that are true in this interpretation, by $T_{\mathcal{L},F}$. A theorem of Macintyre states that $T_{\mathcal{L},F}$ has elimination of quantifiers, see \cite[1.3]{De-Int}. {This implies that any definable set is equivalent to a Boolean combination of sets of one of the following forms: \begin{enumerate}
\item $\left\{ x \mid f(x)=0 \right\}$.
\item $\left\{ x \mid \val(f(x)) \leq \val(g(x)) \right\}$.
\item $\left\{ x \mid ( \exists z) f(x)=z^m \right\}$.
\end{enumerate}
where $f,g$ are polynomials and $m$ is a positive integer. In fact, only sets of the form (3) are needed: the set in (2) is equivalent to the set $\left\{ x \mid (\exists z)\  f^2(x)+pg^2(x)=z^2 \right\}$, and the set in (1) is equivalent to the set $\left\{ x \mid (\exists z)\ pf^2(x)=z^2 \right\}$.}

Recall that a {\emph{definable set}} in a theory $T$ is an equivalence class of formulas under the equivalence relation for which $\psi_1(x_1,\ldots,x_n)$ and $\psi_2(x_1,\ldots,x_n)$ are equivalent if the sentence $\left( \forall x_1,\ldots,x_n \right) \left( \psi_1(x_1,\ldots,x_n) \leftrightarrow \psi_2(x_1,\ldots,x_n)\right)$ follows from $T$. If $X$ is a definable set and $M$ is a model of $T$, the set of all tuples in $M$ that satisfy the formula $X$ is denoted by $X(M)$. We can perform the usual operations of union, intersection, and Cartesian product on definable sets. Similarly, a {\emph{definable function}} between two definable sets $X$ and $Y$ is a definable set $f$ in $X \times Y$ such that the sentence $\left( \forall x \in X \right) \left( \exists ! y\in Y \right) \left( (x,y) \in f\right)$ follows from $T$. If $f$ is a definable function from $X$ to $Y$ and $M$ is a model of $T$, the set $f(M)$ is the graph of a function between $X(M)$ and $Y(M)$. As a consequence of elimination of quantifiers, we have

\begin{proposition} \lbl{prop:definable.to.VG} (see \cite[Theorem 2.7]{De-Int}) Let $X \subset \VF^n$ be a definable set, and let $f:X \rightarrow \VG$ be a definable function. There is a partition $X=X_1 \cup \cdots \cup X_m$ into definable sets, and, for each $i=1,\ldots,m$, a rational function $g_i\in \mathbb{Q}(x_1,\ldots,x_n)$ and an integer $e_i$ such that the restriction of $f$ to $X_i$ is equal to $\frac{1}{e_i}\val(g_i(x))$.
\end{proposition}

\begin{definition} Fix a local field $F$ containing $k$, and let $q$ be the size of the residue field of $F$. Let $f:\VF^n \rightarrow \VG$ be a definable function. \begin{enumerate}
\item Let $A_f:F^n\rightarrow \mathbb{C}$ be the function $x\mapsto f(F)(x)\in \mathbb{Z} \subset \mathbb{C}$.\item Let $E_f:F^n \rightarrow \mathbb{C}$ be the function $x \mapsto q^{f(F)(x)}\in q^{\mathbb{Z}} \subset \mathbb{C}$.
\item We say that a function from $F^n$ to $\mathbb{C}$ is {\emph{constructible}} if it belongs to the algebra generated by all functions of the form $A_f,E_f$.
\end{enumerate}
\end{definition}

\begin{theorem} \lbl{thm:Denef.integration} (see \cite[Theorem 3.1]{De-Int} for a weaker version and \cite[Theorems 4.4 and 6.9]{CL} for a stronger version) Suppose that $f:\VF ^n \times \VF^m \times \VG^m \rightarrow \VG$ is a definable function. Assume that, for every $a\in F^m$ and $\lambda \in \mathbb{Z} ^m$, the integral $I(a,\lambda)=\int_{F^n} q^{f(x,a,\lambda)} d x$ converges. Then the function $I(a,\lambda)$ is constructible.
\end{theorem}

\begin{lemma} \lbl{lem:sum.over.fibers.constructible} Suppose that $X$ and $Y$ are algebraic varieties over $F$, that $\pi :X \rightarrow Y$ is a morphism with finite fibers, and that $f: X \rightarrow \mathbb{A} ^1$ is regular. Then there is a finite definable partition of $Y$ such that, an each part, the function $y\in Y(F) \mapsto \sum_{x\in \pi ^{-1} (y)(F)}|f(x)|$ is a linear combination of functions of the form $y \mapsto q^{g(y)}$, where $g:Y \rightarrow \VG$ is definable.
\end{lemma}

\begin{proof} We will show that there are definable functions $g_i:Y \rightarrow \VG$ such that, for every $y\in Y(F)$, we have $\left\{ \val(f(x)) \mid x\in \pi ^{-1}(y)(F) \wedge f(x)\neq 0\right\} \subset \left\{ g_1(y),\ldots,g_m(y) \right\}$, and, after passing to a finite definable partition, the number of points $x\in \pi ^{-1}(y)(F)$ such that $\val(f(x))=g_i(y)$ is a constant $n_i$. Given this claim, we have that $\sum_{x\in \pi ^{-1} (y)(F)}|f(x)|=\sum n_i q^{-g_i(y)}$ on each piece.

Let $N$ be the maximum size of a fiber of $\pi$. We will show, by induction on $i=1,\ldots,N$, that there is a definable partition of $Y$, and, for each part $A$, there are definable functions $g_1,\ldots,g_i$ and natural numbers $n_1,\ldots,n_i$ such that, for each $y\in A(F)$, \begin{enumerate}
\item $\#\left\{ x\in \pi ^{-1} (y) \mid \val(f(x))=g_i(y) \right\} =n_i$.
\item $\# \left\{ x\in \pi ^{-1}(y) \mid f(x)\neq0 \wedge \val(f(x))\notin \left\{ g_1(y),\ldots,g_i(y) \right\} \right\} \leq N-i$.
\end{enumerate}
The claim for $i=N$ is what we want to show.

Suppose that $g_j,n_j$ were chosen for $j<i$, and let
\[
g_i(y)=\min \left\{ \val(f(x)) \mid \pi(x)=y \wedge \val(f(x)) \neq g_1(y),\ldots,g_{i-1}(y) \right\}.
\]
Clearly, $g_i$ is a definable function. Since, for every $n$, the set $\left\{ y \in Y \mid \left( \exists x_1,\ldots,x_n \in X \right) \pi(x_j)=y \wedge \val(f(x_j))=g_i(y) \wedge x_j\neq x_l \right\}$ is definable, we get that there is a partition as in the assertion of the lemma.
\end{proof}

\begin{proposition} \lbl{prop:integration.is.constructible} Let $\pi:X \rightarrow Y$ be a morphism between algebraic varieties over $F$, and let $U \subset X$ be an open set such that $\pi|_U$ is smooth. Suppose that $\omega \in \Gamma(U,\Omega_{U/Y})$ and that $\varphi :X(F) \rightarrow \mathbb{C}$ is a Schwartz function. Then the function $y \mapsto \int_{\pi ^{-1} (y) \cap U}| \omega |$ is constructible.
\end{proposition}

\begin{proof} Let $O$ be the ring of integers of $F$. Since the claim is local, we can assume that $Y \subset \mathbb{A} ^M$ is affine, that $X \subset Y \times \mathbb{A}^N$, and that $\pi$ is the projection. By decomposing into finitely many pieces, translating, and dilating, we can assume that $\varphi$ is the characteristic function of a set of the form $O^{M+N}$. Finally, decomposing $A$ further, we can assume in addition that there is a subset $I \subset \left\{ 1,\ldots,N \right\}$ such that, for any $y\in A$, the projection $\pi_I: \pi ^{-1}(y) \cap U \cap O^{M+N} \rightarrow O^I$ is etale. Let $\nu$ be the standard volume form on $O^I$. For any $y\in Y(F)$,
\[
\int_{\pi ^{-1}(y) \cap U \cap O^{M+N}}| \omega |=\int_{z \in O^I}\left( \sum_{x\in \pi_I ^{-1}(z)\cap \pi ^{-1}(y) \cap U\cap O^{M+N}}\left| \frac{\omega }{\pi_I ^* \nu} (x)\right| \right) | \nu |.
\]
By Lemma \ref{lem:sum.over.fibers.constructible}, there is a definable partition such that, on each part, the integrand is a linear combination of functions of the form $q^{-f(x)}$, where $f: Y \rightarrow \VG$ is definable. By Theorem \ref{thm:Denef.integration}, the integral is a constructible function of $y$.
\end{proof}

\subsection{Continuity of Constructible Functions}

{We start by considering constructible functions from $\VG^n$ to $\mathbb{C}$.}

\begin{lemma} \lbl{lem:linear.independence} Let $\varphi_1,\ldots,\varphi_N \in (\mathbb{Q} ^n)^*$ be different linear functionals and let $f_1,\ldots,f_N\in \mathbb{Q}[x_1,\ldots,x_n]$ be non-zero polynomials. Then the functions $f_i(x)q^{\varphi_i(x)}:\mathbb{Z}_{\geq 0} ^n \rightarrow \mathbb{R}$, where $x=(x_1,\ldots,x_n)$, are linearly independent.
\end{lemma}

{
\begin{proof} Suppose that $\sum \alpha_i f_i(x) q^{\varphi_i(x)}=0$, where all coefficients $\alpha_i$ are nonzero. Choose $v\in \mathbb{Z}_{\geq 0}^n$ for which $\varphi_i(v) \neq \varphi_j(v)$ for all $i,j$. By rearranging, we can assume that $\varphi_1(v) > \varphi_i(v)$ for all $i>1$. For every $x$ we have
\[
0=\lim_{n \rightarrow \infty} \frac{\sum \alpha_i f_i(x+nv)q^{\varphi_i(x+nv)}}{f_1(x+nv)q^{\varphi_1(nv)}} = \alpha_1 q^{\varphi_1(x)},
\]
so $\alpha_1=0$, a contradiction.
\end{proof}
}
\begin{definition} A {\emph{V-function}} is a function $R: \mathbb{Z}_{\geq 0} ^n \rightarrow \mathbb{C}$ of the form $R(x)=\sum \alpha_i f_i(x) q^{\varphi_i(x)}$, where $\varphi_i\in (\mathbb{Q} ^n)^*$ are linear functionals, and $f_i(x)\in \mathbb{Q}[x_1,\ldots,x_n]$ are polynomials.
\end{definition}

In other words, the collection of $V$-functions is the algebra generated by the functions of the form $\varphi(x)$ and $q^{\varphi(x)}$, where $\varphi$ is a rational linear functional.

\begin{lemma} \lbl{lem:no.infinite.ray.limits} Let $R(x)=\sum \alpha_i f_i(x) q^{\varphi_i(x)}$ be a V-function. Suppose that {the $\varphi_i$ are all distinct and} there are no $p,v\in \mathbb{Z}_{\geq 0} ^n$ such that $\lim_{n \rightarrow \infty} |R(p+nv)|= \infty$. Then, for any $e\in \mathbb{Z}_{\geq 0}^n$, \begin{enumerate}
\item $\varphi_i(e) \leq 0$, and
\item If $\varphi_i(e)=0$, then $\frac{\partial f_i}{\partial e}=0$.
\end{enumerate}
Conversely, conditions 1. and 2. imply that there are no $p,v\in \mathbb{Z}_{\geq 0} ^n$ such that $\lim_{n \rightarrow \infty} |R(p+nv)|=  \infty$.
\end{lemma}

\begin{proof} Suppose that there is $e\in \mathbb{Z}_{\geq 0} ^n$ and $i$ such that $\varphi_i(e)>0$. By reordering the terms, we can assume that $\varphi_1(e)>0$ and that $\varphi_1(e)>\varphi_i(e)$ for $i>0$. Choosing $p\in \mathbb{Z}_{\geq 0}$ such that $f_1(p) \neq 0$, we get that $f_1(p+n e)q^{\varphi_1(p+n e)}$ tends to infinity much faster than any other term, so $\lim |F(p+n e)|=\infty$, a contradiction.

Now assume that $\varphi_i(e)=0$. Denote the set of indices $j$ such that $\varphi_j(e)=0$ by $S$. Then, for any $p\in \mathbb{Z}_{\geq 0} ^n$,
\[
F(p+n e)=\sum_{j\in S} \alpha_j f_j(p+n e)q^{\varphi_j(p)} + o_n(1).
\]
By assumption, the polynomial $t \mapsto \sum_{j\in S} \alpha_j f(p+ t e)q^{\varphi_j(p)}$ is constant. Taking the derivative we get $\sum_{j\in S} \alpha_j \frac{\partial f_j}{\partial e}(p) q^{\varphi_j(p)}=0$ for all $p$. By Lemma \ref{lem:linear.independence}, $\frac{\partial f_j}{\partial e}=0$ for all $j\in S$.

The converse implication is clear.
\end{proof}

For the next lemma, we topologize the set $\mathbb{Z}_{\geq 0} \cup \left\{ \infty \right\}$ as the one-point compactification of the discrete space $\mathbb{Z}_{\geq 0}$.

\begin{lemma} \lbl{lem:V.function.extension} Suppose that $R: \mathbb{Z}_{\geq 0} ^n \rightarrow \mathbb{C}$ is a V-function such that there are no $p,v\in \mathbb{Z}_{\geq 0} ^n$ such that $\lim_{n \rightarrow \infty} |R(p+nv)|= \infty$. Then $R$ extends to a continuous function on $(\mathbb{Z}_{\geq 0} \cup \left\{ \infty \right\})^n$.
\end{lemma}

\begin{proof} By Lemma \ref{lem:no.infinite.ray.limits}, we can assume without loss of generality that $R(x)=f(x)q^{\varphi(x)}$ for a polynomial $f(x)$ and a linear functional $\varphi$ that satisfy conditions 1. and 2. By decomposing $f(x)$ into a sum of monomials, we can assume that $f(x)$ is a monomial. In this case, the function is $\prod \left( x_i^{a_i} q^{b_ix_i} \right)$, and the claim is reduced to the one dimensional case, where it is clear.
\end{proof}

\begin{definition} \lbl{defn:simple.function} Let $X$ be a definable set. A {\emph{simple function}} $f:X(F) \rightarrow \mathbb{C}$ is a composition of two functions $V: X(F) \rightarrow \mathbb{Z}_{\geq 0} ^n$ and $R: \mathbb{Z}_{\geq 0} ^n \rightarrow \mathbb{C}$, such that: \begin{enumerate}
\item $V$ is a surjection.
\item Each of coordinates of $V$ has the form $\frac{1}{e}\val(g(x))$, where $g$ is a rational function and $e\in \mathbb{Z}_{> 0}$
\item $R$ is a V-function.
\end{enumerate}
\end{definition}

\begin{proposition} \lbl{prop:constructible.to.simple} Let $f:X(F) \rightarrow \mathbb{C}$ be a constructible function. There is a finite definable partition $X=X_1 \cup \cdots \cup X_m$ such that the restriction of $f$ to each $X_i(F)$ is simple.
\end{proposition}

\begin{proof} There are definable functions $\alpha_1,\ldots,\alpha_m: X \rightarrow \VG$ such that $f$ is in the algebra generated by the functions $\alpha_i(x)$ and $q^{\alpha_i(x)}$. After passing to a definable partition of $X$, we can assume that each $\alpha_i$ has the form $\frac{1}{e}\val(g_i(x))$, where $g_i$ is a rational function, and $e\in \mathbb{Z}_{\geq 0}$. Consider the image of $X(F)$ under the map $(\alpha_1,\ldots,\alpha_m)$. It is a definable subset of $\mathbb{Z} ^n$, and, by \cite[Section 2.6]{Fu}, is a disjoint union of simple cones. Dividing $X$ further, we can assume that the image of $X(F)$ under $(\alpha_1,\ldots,\alpha_m)$ is a simple cone $C$. There is an integral linear transformation $A$ that induces an isomorphism between $C$ and $a\mathbb{Z}_{\geq 0} ^n$, for some $a\in \mathbb{Z}_{\geq 0}$. Let $V$ be the composition
\[
X \stackrel{(\alpha_i)}{\rightarrow}C \stackrel{A}{\rightarrow}a \mathbb{Z}_{\geq 0} ^n \stackrel{1/a}{\rightarrow} \mathbb{Z}_{\geq 0} ^n.
\]
$V$ is a surjection, and it is easy to see that each coordinate of $V$ has the form $\frac{1}{ae}\val(h(x))$, for some rational function $h$. Finally, each $\alpha_i$ is a $\mathbb{Q}$-linear functional in the coordinates of $V$. This implies that there is a $V$-function $R$ such that $f=R\circ V$.
\end{proof}

\begin{proposition} \lbl{prop:piecewise.continuity} Let $X \subset \VF^n$ be a bounded constructible set and $f:X(F) \rightarrow \mathbb{C}$ be a constructible function. Then one of the following holds \begin{enumerate}
\item There is a definable subset $Y \subset X$ and a point $y_0 \in \overline{Y(F)}$ such that
\[
\lim_{\underset{y\in Y(F)}{ y \rightarrow y_0}} f(y) = \infty.
\]
{ \item There are two definable subsets $Y_1,Y_2 \subset X$ such that $\overline{Y_1(F)}\cap \overline{Y_2(F)} \neq \emptyset$ and
\[
\inf_{y_i\in Y_i(F)} |f(y_1)-f(y_2)|>0.
\]
}
\item There is a definable partition $X=X_1 \cup \cdots \cup X_m$ such that, for each $i$, the restriction of $f$ to $X_i(F)$ extends to a continuous function on $\overline{X_i(F)}$.
\end{enumerate}
Here, the closures of $Y(F)$, $Y_i(F)$, and $X_i(F)$ are taken inside $F^n$.
\end{proposition}

\begin{proof} After passing to a definable partition, we can assume that $f$ is simple. Write $f=R\circ V$, with $R$ and $V$ as in Definition \ref{defn:simple.function}.

Assume first that there are $p,v\in \mathbb{Z}_{\geq 0} ^n$ such that $R(p+nv) \rightarrow \infty$. We will show that the first alternative holds. Let $Y=V ^{-1} (p+\mathbb{Z}_{\geq 0} v)$. Choose $x_n\in V ^{-1} (p+nv)$, and let $y_0 \in F^n$ be any accumulation point of $\left\{ x_n \right\}$. If $y_n \in Y$ and $y_n \rightarrow y_0$, then $V(y_n)$ tends to infinity on the ray $p+\mathbb{Z}_{\geq 0} v$, and so $f(y_n)=R(V(y_n)) \rightarrow \infty$.

{Assume now that there are no $p,v\in \mathbb{Z}_{\geq 0} ^n$ such that $R(p+nv) \rightarrow \infty$ and assume that third alternative does not hold. We will show that the second alternative holds. By Lemma \ref{lem:V.function.extension}, the function $R$ extends continuously to a function $\overline R$ on $\left( \mathbb{Z}_{\geq 0} \cup \left\{ \infty \right\} \right) ^n$. This implies that $R$ is bounded, and, hence, so is $f$.

By assumption, $f$ cannot be extended continuously to $\overline{X(F)}$, so we can find a point $p \in \overline{X(F)}$ and two sequences of points $x_i,y_i \in X(F)$ such that  $p=\lim x_i=\lim y_i$ and both limits $\lim f(x_i)$ and $\lim f(y_i)$ exist and are not equal. Let $a_{i}=V(x_{i})$ and $b_{i}=V(y_{i})$. Choose subsequences $a_{i_{j}}$ and $b_{i_{j}}$  which converge to points $a,b \in \left( \mathbb{Z}_{\geq 0} \cup \left\{ \infty \right\} \right) ^n$. Clearly, $\overline R(a) \neq \overline R(b)$.

We can choose non-intersecting neighborhoods $I \ni a$ and  $J \ni b$ in $\left( \mathbb{Z}_{\geq 0} \cup \left\{ \infty \right\} \right) ^n$, and non-intersecting closed neighborhoods $A \ni \overline R(a)$, $B \ni \overline R(b)$ in $\C$, such that both $I$ and $J$ are products of (possibly infinite) intervals in $ \mathbb{Z}_{\geq 0} \cup \left\{ \infty \right \}$, that $\overline{R}(I) \subset A$, and that $\overline{R}(J) \subset B$. The definable subsets $Y_1:= V^{-1}(I)$ and $Y_2:= V^{-1}(J)$ satisfy the requirements. }

\end{proof}

\begin{lemma} \lbl{lem:curve.through.set} Let $W$ be an affine variety over a local field $F$, $X \subset W$ be an $F$-definable subset, and $x\in W(F)$ be a point. Assume that $x$ is contained in the closure of $X(F)$. Then there is a curve $C \subset W$ defined over $F$ such that $x$ is in the closure of $(C \cap X)(F)$.
\end{lemma}

\begin{proof} If $x\in X(F)$, the claim is trivial; we assume that this is not the case. $X$ is given by a Boolean combination of conditions of the form $f_i(x)=0$ or $( \exists z) g_i(x)=z^m$ for some regular functions $f_i$ and $g_i$. By intersecting $W$ with the locus of vanishing of $f_i$, we can assume that $X$ is given by a Boolean combination of conditions of the form $( \exists z) g_i(x)=z^{m_i}$. Let $G=\prod g_i$ and let $Z \subset W$ be the zero locus of $G$. Let $\pi :\widetilde{W} \rightarrow W$ be a strict resolution of singularities of the pair $(W,Z)$. Thus, $\widetilde{W}$ is smooth and the composition $G \circ \pi$ is locally monomial in some local coordinate system. {After replacing $X$ by the intersection of $\pi ^{-1} (X)$ and a small ball around some $\widetilde{x} \in \pi ^{-1} (x)$,} it is enough to prove the lemma in the case where $W=\mathbb{A} ^n$, $x=0$, and the functions $g_i$ are monomial. By assumption, there is a non-zero point $p$ in $X(F) \cap O^n$. We will show that the line $C$ passing through $0$ and $p$ satisfies the required property. Let $\alpha \in O$. Since $g_i$ is monomial and $g_i(p)$ is an $m_i$th power, it follows that $g_i \left( \alpha ^{\prod m_i} p\right)$ is also an $m_i$th power. Thus, the point $\alpha ^{\prod m_i}p$ belongs to $X(F)$. Taking $\alpha \rightarrow 0$, we get that $0$ is an accumulation point of $C \cap X(F)$.
\end{proof}

\begin{proposition} \lbl{prop:continuity.along.curves} Suppose that $X \subset \VF^n$ is a {closed and} bounded definable set, and that $f:X(F) \rightarrow \mathbb{C}$ is a constructible function which is discontinuous. Then there is a curve $C\subset \mathbb{A} ^n$ such that the restriction of $f$ to $(C \cap X)(F)$ is discontinuous.
\end{proposition}

\begin{proof} We apply Proposition \ref{prop:piecewise.continuity}.
\begin{enumerate}[{Case} 1:]
\item
There is a definable subset $Y \subset X$ and a point $y_0\in\overline{Y(F)} \subset X(F)$ such that
$$
\lim_{\underset{y\in Y}{y \rightarrow y_0}} f(y)=\infty.
$$
Choose, using Lemma \ref{lem:curve.through.set} a curve $C$ such that $y_0\in \overline{(C \cap Y)(F)}$. The restriction of $f$ to $(C\cap X)(F)$ is discontinuous.
\item
There are definable subsets $Y_1,Y_2 \subset X$ and a point $p\in \overline{Y_1(F)}\cap\overline{Y_2(F)} \subset X(F)$ such that
\[
\inf_{y_i\in Y_i(F)} |f(y_1)-f(y_2)|>0.
\]
Choose, using Lemma \ref{lem:curve.through.set} curves $C_1$ and $C_2$ such that $p\in \overline{(C_i \cap Y_i)(F)}$. Denoting $C=C_1 \cup C_2$, the restriction of $f$ to $(C \cap X)(F)$ is discontinuous.
\item
There is a definable partition $X=X_1 \cup \cdots \cup X_m$ such that, for each $i$, the restriction of $f$ to $X_i(F)$ extends to a continuous function on $\overline{X_i(F)}$.\\
Since $f$ was assumed to be discontinuous, there is $i$ and a point $x_0 \in \overline{X_i(F)}$ such that
\[
\lim_{\underset{x\in X_i}{x \rightarrow x_0}}f(x) \neq f(x_0).
\]
Using Lemma \ref{lem:curve.through.set}, choose a curve $C$ such that $x_0\in \overline{(C \cap X_i)(F)}$. The restriction of $f$ to $(C\cap X)(F)$ is discontinuous.
\end{enumerate}
\end{proof}

\section{Recollections in Algebraic geometry} \lbl{sec:alg.geom}
\setcounter{theorem}{0}
In this section, we review some well-known notions and facts from algebraic geometry that we refer to in the paper. Throughout this section, $k$ is a field of characteristic $0$ and all schemes and algebras are of finite type over $k$. By a variety, we mean a reduced scheme over $k$.

\subsection{Flat Maps} \lbl{ssec:flat}
\begin{theorem} [{cf.  \cite[11.3.10]{EGA}}]\lbl{thm:flat.locus}
Let $S$ be a scheme and let $\phi:X\to Y$  be a morphism of flat $S$-schemes. Denote the structure maps $X \to S$ and $Y \rightarrow S$ by $\pi_X$ and $\pi_Y$ respectively. Then
\begin{enumerate}
 \item $\{x\in X(k) \mid  \phi \text{ is flat at } x\}=\{x \in X(k) \mid \phi_{\pi_X(x)}:\pi_X ^{-1}(\pi_X(x)) \rightarrow \pi_Y ^{-1} (\pi_X(x)) \text{ is flat at } x\}$.
 \item The set in (1) is Zariski open.
\end{enumerate}
\end{theorem}

We will call the set in Theorem \ref{thm:flat.locus} the flat locus of $\phi$.

\begin{theorem}[{cf. \cite[Exercise 10.9]{Har_ag}}] \label{thm:flat_dim}
Let  $\phi:X\to Y$  be a morphism of smooth irreducible algebraic varieties. Let $x \in X(k)$.
 Then the following are equivalent:
\begin{enumerate}
 \item  $\phi$ is flat at $x$.
 \item $\dim(\phi^{-1}(\phi(x)))=\dim X -\dim Y$.
 \item $\dim(\phi^{-1}(\phi(x))) \leq \dim X -\dim Y$.
\end{enumerate}
\end{theorem}

\begin{prop}[cf. {\cite[III.10.2]{Har_ag}}]\label{prop:sm_crit}
Let $\phi:X \to Y$ be a flat morphism of schemes, and let $x \in X(k)$. The following conditions are equivalent
\begin{enumerate}
 \item $\phi$ is smooth at $x$
 \item $x$ is a smooth point of $\phi^{-1}(\phi(x))$.
\end{enumerate}
\end{prop}

\begin{prop}[{\cite[Proposition 9.7]{Har_ag}}]\label{prop:flat_zero_div}
Let $X$ be an affine  scheme and let $f \in\Gamma(X,O_X)$. Consider $f$ as a map $X \to \A^1$. Then the following conditions are equivalent
\begin{enumerate}
 \item $f$ is flat at any point of $f ^{-1} (0)$.
 \item $f$ is not a zero divisor.
\end{enumerate}

\end{prop}

\subsection{LCI  Maps} \lbl{ssec:LCI}

For convenience, we will assume that all the schemes we discus from this point on are equidimensional. All the constructions and results that we discus below can be easily modified for locally  equidimensional schemes, and, usually, also for arbitrary schemes.

\begin{defn}[{cf. \cite[Introduction]{Con}}]$ $
\begin{enumerate}
\item Let $X$ be an affine scheme, and let $Y \subset X$ be a subscheme. We say that $Y$ is a \emph{complete intersection} in $X$ if there are $f_1,\dots f_n\in k[X]$ such that $f_i$ is not a zero divisor in  $k[X]/( f_1, \dots f_{i-1})$ for all $i$ and $Y$ is the zero locus of the $f_i$.
\item Let $X$ be a scheme. A closed subscheme $Y \subset X$ is called a \emph{local complete intersection} (LCI for short) inside $X$ if there exists an affine open cover $X = \bigcup{U_i}$ such that $Y \cap U_i$ is CI inside $U_i$.
\item A morphism  $Y \to X$ between two schemes is called LCI if it is a closed embedding and its image is a LCI inside $X$.
\item Given a LCI map $i: Y \to  X$, we define its normal  bundle  $$N_{Y}^{{X}}:=\mathcal{H}om_{\mathcal{O}_Y}(i^*( \cI_Y/\cI_Y^2),O_Y),$$ where $\cI_Y$  is the sheaf of regular functions on $X$ that vanish on $Y$, and $\mathcal{H}om$ is sheaf of homomorphisms.
\end{enumerate}
\end{defn}

\begin{prop} Let $i: Y \to  X$ be a LCI map. Then $N_{Y}^{{X}}$ is a locally free sheaf and $\rk (N_{Y}^{{X}})=\dim \ i:= \dim X- \dim Y$.
\end{prop}

{
\begin{defn}$ $
\begin{enumerate}
\item An affine scheme $X$ is a \term{complete intersection}{CI}\ (CI for short) if it a CI inside some affine space.
\item A scheme $X$ is a \term{local complete intersection}{LCI}\ (LCI for short) if, locally, it is an affine CI scheme. Note that this does not mean that there is an LCI embedding of  $X$  into an affine space.
\end{enumerate}
\end{defn}
}
Theorem \ref{thm:flat_dim} implies:
\begin{cor}
Let $X$ be a scheme. The following conditions are equivalent:
\begin{enumerate}
\item $X$ is a {LCI}.
\item Locally (on  $X$), one can find a flat map of smooth varieties $\phi:X' \to Y'$ such that $X=\phi^{-1}(y)$  for some $y \in Y(k)$.
\end{enumerate}
\end{cor}

\subsection{Grothendieck Duality} \lbl{ssec:G.duality}
For a full treatment of the machinery of  Grothendieck Duality in the coherent category, see \cite{Har_du} and \cite{Con}. We will sum up the ingredients of this theory that we need. Since some sign errors from \cite{Har_du} were corrected in \cite{Con}, we will try to follow the corrected version when there is a difference. This, however, is not essential for our purposes, since, in our applications, we will be interested only at the absolute values of differential forms (for purposes of integration).

For a scheme $X$, we consider the derived category of sheaves on $X$, and denote the subcategory consisting of complexes with coherent cohomologies by $D(X)$. Similarly, we denote the subcategory of $D(X)$ consisting of bounded (respectively, bounded from below, respectively, bounded from above) complexes by $D^b(X)$ (respectively, $D^+(X)$, respectively, $D^-(X)$).

\subsubsection{The Functor $\pi^!$}\label{sec:pi_shrik}

Before defining the functor $\pi^!$, we will define two special cases of it: $\pi ^b$ and $\pi ^\#$.
For these, we will need the following notation:
\begin{notn} $ $
\begin{enumerate}
\item For a smooth map $\pi:X \to Y$, we denote the line bundle of relative top differential forms on $X$ with respect to $Y$ by $\Omega_{X/Y}$. If $X$ is smooth, we denote $\Omega_X=\Omega_{X/\Spec(k)}$.
\item For a LCI map $i: X \to  Y$, we define $\Omega_{X/Y}$  to be the top exterior power of $N_X^Y$.
\end{enumerate}
\end{notn}

\begin{defn}[cf. {\cite[(2.2.7)--(2.2.8)]{Con}}] Let  $\pi:X\to Y$ be a morphism of schemes.
\begin{enumerate}
\item Assume that $\pi$ is smooth. Define the functor  $$\pi^\#:D^+_{}(Y) \to D^+_{}(X)$$ by $$\pi^\#(\cF):=\Omega_{X/Y}[(\dim X -\dim Y )] \overset{L}{\otimes} \pi^*(\cF).$$
\item Assume that $\pi$ is finite. Define the functor $$\pi^b:D^+_{}(Y) \to D^+_{}(X)$$ by
$$\pi^b(\cF)= \overline \pi^*(\mathcal{RH}om_{Y}(\pi_{*}(O_X),\cF)).$$
Here, $\mathcal{RH}om_{Y}$ is the internal hom in $D^+(Y)$, the object $\mathcal{RH}om_{Y}(\pi_{*}(O_X),\cF)$ is considered with the natural action of the sheaf of algebras $\pi_{*}(O_X)$, and the ringed space morphism $\overline \pi:(X,O_{X})\to (Y,\pi_*(O_{X}))$ is the one induced by $\pi$.
\item Note that the functor $\pi^b$ is a right adjoint to $\pi_*$. The adjunction map $R\pi_* \circ \pi^b \to Id_{D^+(Y)}$ is called the trace map and denoted by $Tr^{{b}}_\pi$. 
See {\cite[III, Proposition  6.5]{Har_du}} for more details.
\end{enumerate}
\end{defn}

In \cite{Har_du,Con}, the following functors and natural transformations were constructed:
\begin{enumerate}[(i)]
\item For any morphism of schemes $\pi:X \to Y$, a functor $$\pi^!:D^+_{}(Y) \to D^+_{}(X).$$
\item For any morphisms of schemes $X \overset{\pi}{\to} Y \overset{\tau}{\to} Z$, an isomorphism  $$c_{\pi,\tau}: (\tau\circ\pi)^! \to \pi^!\circ\tau^! .$$
\item For any finite morphism of schemes $\pi:X \to Y$, an isomorphism
$$d_{\pi }: \pi^!\to\pi^b.$$
\item For any smooth morphism of schemes $\pi:X \to Y$, an isomorphism
$$e_{\pi }: \pi^!\to\pi^\#.$$
\item For any proper morphism of schemes $\pi:X \to Y$, a morphism
$$Tr_{\pi }: R\pi_* \circ \pi^! \to Id.$$
 \item For any cartesian square of schemes
$$\xymatrix{
V \ar@{->}_\tau[d] \ar@{->}^i[r]& X \ar@{->}^\pi[d]\\
U\ar@{->}^j[r]& Y
}$$
 where $i,j$ are open embeddings\footnote{$b_{j,\pi}$ is defined, more generally, for any flat $i,j$ but we will not need this fact.}, an isomorphism $$b_{j,\pi}:i^*\circ \pi^! \to \tau^!\circ j^*.$$
\item { For any LCI morphism of schemes $\pi:X \to Y$, an isomorphism $\eta_\pi$ between the functor $\pi ^b$ and the functor $\mathcal{F} \mapsto L\pi^*(\cF) \overset{L}{\otimes} {\Omega}_{X/Y}[\dim X-\dim
 Y]$.}
\end{enumerate}
The properties of these constructions are summed up in \cite[VII, Corollary 3.4 {and III, Corollaries 7.3--7.4}]{Har_du}. We include the properties that we will use in the next proposition.

\begin{prop}\label{prop:shrik_prop}
$ $
\begin{enumerate}
\item\label{prop:shrik_prop:1} For any scheme $X$, the functor $Id_X^!$ is the identity functor. For any morphism of schemes, $\pi: X \to Y$ the natural transformations $c_{{Id},\pi}$ and $c_{\pi,{Id}}$ are the identity.
\item\label{prop:shrik_prop:2} Coherence relations for $c$: For any diagram $$X \overset{\pi}{\to} Y \overset{\tau}{\to} Z \overset{\mu}{\to} W$$ of schemes and an object $\cF\in D^+_{}(W)$, the following diagram is commutative
 $$\xymatrix{
\pi^!(\tau^! ( \mu^!(\cF))) \ar@{->}_{\pi^!(c_{\mu,\tau})}[d] \ar@{->}^{c_{\pi,\tau}}[r]& (\tau \circ \pi)^! (\mu^!(\cF)) \ar@{->}^{c_{\tau \circ \pi,\mu}}[d]\\
\pi^!((\mu\circ \tau)^!  (\cF))\ar@{->}^{{c_{\pi,\mu \circ \tau }}}[r]& (\mu \circ \tau \circ \pi)^!(\cF)
}$$

\item\label{prop:shrik_prop:3} Let
$$\xymatrix{
V \ar@{->}_\tau[d] \ar@{->}^i[r]& X \ar@{->}^\pi[d]\\
U\ar@{->}^j[r]& Y
}$$
be  {a cartesian square of schemes such that $i,j$ are open  embeddings, }
  and let $\mathcal{F} \in D^+(Y)$.
 \begin{enumerate}
\item \label{prop:shrik_prop:3a}
If $\pi$ is finite, the following diagram is commutative

$$\xymatrix{
\tau^b(\cF|_{U}) \ar@{->}^{}[r]& (\pi^b(\cF))|_{V} \\
\tau^!(\cF|_{U})\ar@{->}^{b_{j,\pi}}[r] \ar@{->}_{d_\tau}[u] & \pi^!(\cF)|_V \ar@{->}^{d_\pi}[u]
}$$
where the upper row is the natural isomorphism.

\item \label{prop:shrik_prop:3b}
If $\pi$ is smooth, the following diagram is commutative
$$\xymatrix{
\tau^{\#}(\cF|_{U}) \ar@{->}^{}[r]& (\pi^{\#}(\cF))|_{V} \\
\tau^!(\cF|_{U})\ar@{->}^{b_{j,\pi}}[r] \ar@{->}_{e_\tau}[u] & \pi^!(\cF)|_V \ar@{->}^{e_\pi}[u]
}$$
where the upper row is the natural isomorphism.

\item \label{prop:shrik_prop:3c} If $\pi$ is proper, the morphism $Tr_{\tau}$ coincides with $Tr_{\pi}|_U$ under the identification $\tau_*({b_{j,\pi}}): \tau_*(\tau^!(\cF|_{U})) \to  \tau_*((\pi^!(\cF))|_{V})$.
\end{enumerate}

\item  \label{prop:shrik_prop:4} If $\pi$ is proper, the morphism $Tr_{\pi}$ exhibits $\pi^!$ as a right adjoint to the functor $R\pi_*$.

\item \label{prop:shrik_prop:41}  Let $\pi:X \rightarrow Y$ be a finite morphism and $\cF\in D^+_{}(Y)$ be an object. Then the following diagram is commutative:
$$\xymatrix{
R\pi_*(\pi^!(\cF)) \ar@{->}^{\qquad Tr_\pi}[r] \ar@{->}^{\pi_*(d_{\pi})}[d]& \cF \\
R\pi_*(\pi^b(\cF)) \ar@{->}_{Tr^b_\pi}[ru] }$$

\item\label{prop:shrik_prop:42} For any diagram$$X \overset{\pi}{\to} Y \overset{\tau}{\to} Z$$ of schemes and an object $\cF\in D^+_{}(Z)$, the following diagram is commutative: $$\xymatrix{
R(\tau\circ \pi)_*((\tau\circ \pi)^!(\cF)) \ar@{->}^{\qquad \qquad Tr_{\tau\circ \pi}}[r]& \cF \\
R\tau_*(R\pi_*(\pi^!(\tau^!(\cF)))) \ar@{->}_{\qquad R\tau_*(Tr_\pi)}[r] \ar@{->}^{R(\tau\circ \pi)_*(c_{\pi,\tau})\circ a_{\pi,\tau}}[u]& R\tau_*(\tau^!(\cF)),\ar@{->}_{Tr_\tau}[u]}$$
where $a_{\pi,\tau}:R\tau_* \circ R\pi_* \to  R(\tau\circ \pi)_{*}$ is the natural isomorphism.

\item  \label{prop:shrik_prop:5} Let $i: Y \hookrightarrow X$ be an embedding of smooth  algebraic varieties. Let $pr_X:X\to spec(k)$ and $pr_Y:Y\to spec(k)$ be the projections to a point. Denote the dimensions of $X$ and $Y$ by $n$ and $m$. Consider the composition:
\begin{multline*}
\zeta_i:\Omega_Y=pr_Y^\#(k)[-m]\overset{e_{pr_{Y}}^{-1}}{\to}pr_Y^!(k)[-m]= (pr_X\circ i)^{!}(k)[-m]\overset{c_{i,pr_x}^{-1}}{\to}i^!( pr_X^!(k))[-m]\overset{i^!(e_{pr_{X}})}{\to} \\ i^!( pr_X^\#(k))[-m]=i^!(\Omega_X)[n-m]\overset{d_{i}}{\to}i^b(\Omega_X)[n-m]\overset{\eta_i}{\to}i^*(\Omega_X)\otimes \Omega_{Y/X}.
\end{multline*}
 Then, in local coordinates, we can write $\zeta_i$ as follows:
 $$\zeta_i(dx_{1}\wedge \dots \wedge dx_{k}\wedge dt_1 \wedge \dots \wedge dt_{n})=(dt_1^{\vee} \wedge \dots \wedge dt_k^{\vee})\otimes i^*(dt_n \wedge \dots \wedge dt_{1}\wedge dx_{1}\wedge \dots \wedge dx_{k}),$$
where $$x_1,\dots,x_k, t_1, \dots, t_n$$ are local coordinates on $X$, the subvariety  $Y$ is given by $$t_1= \dots= t_n=0,$$  and $dt_1^{\vee} , \dots , dt_k^{\vee}$ are the basis of the normal bundle $N_Y^X$ dual to the basis  $dt_1 , \dots , dt_k$ of the conormal bundle $(N_Y^X)^* \subset T^*(X)$.

\item \label{prop:shrik_prop:6} Let $i: Y \to  X$ be a  LCI map and let $\cF \in D^+_{}(X)$. Then,
\begin{enumerate}
\item\label{prop:shrik_prop:6b} The isomorphism $\eta_i$  is compatible with restrictions to open sets, in a similar way to \eqref{prop:shrik_prop:3a} above.
\item\label{prop:shrik_prop:6c} If $\cF$  is a locally free sheaf and $\dim X-\dim Y=1$, the trace map $Tr^b_i$ can be described as follows: identifying $Ri_*(i^b(\cF))$ and $i_*(i^*(\cF) \otimes {\Omega}_{Y/X})[-1]$ via $\eta_i$, the trace map
$$Ri_*(i^b(\cF))\cong i_*(i^*(\cF) \otimes {\Omega}_{Y/X})[-1]\to \cF$$
is represented by an extension $$0\to\cF \to \cF(Y) \to i_*(i^*(\cF) \otimes {\Omega}_{Y/X}) \to 0,$$
where $\cF(Y)$ is the sheaf of rational sections of $\cF$  that become regular after multiplication by a function that vanishes on $Y$, and the map $\cF \to \cF(Y)$ is the natural embedding.
\end{enumerate}

\end{enumerate}
\end{prop}

\begin{rem}
$ $
\begin{itemize}
\item
Parts  \eqref{prop:shrik_prop:1} and \eqref{prop:shrik_prop:2} follow from   \cite[VII,Corollary 3.4. (a) VAR1]{Har_du}.
  \item  Parts \eqref{prop:shrik_prop:3a} and \eqref{prop:shrik_prop:3b} follow from \cite[VII,Corollary 3.4. (a) VAR6]{Har_du}.
\item Part \eqref{prop:shrik_prop:3c}  follows from \cite[VII, Corollary 3.4. (b) TRA4]{Har_du}.
\item Part \eqref{prop:shrik_prop:4}  follows from \cite[VII, Corollary 3.4. (c)]{Har_du}.
\item Part \eqref{prop:shrik_prop:41} follows from \cite[VII, Corollary 3.4. (b) TRA2]{Har_du}.
\item Part \eqref{prop:shrik_prop:42} follows from \cite[VII, Corollary 3.4. (b) TRA1]{Har_du}.
\item {Part \eqref{prop:shrik_prop:5}  follows from \cite[VII,Corollary 3.4. (a) VAR5]{Har_du}. Note that in \cite[III, Definition 1.5]{Har_du} the explicit description of $\zeta_i$  sometimes differs by a sign from the one we use here following \cite[page 30 case (c)]{Con}.} 
\item {Part \eqref{prop:shrik_prop:6b}  follows from \cite[III, Proposition 7.4(b)]{Har_du}.}
\item {Part \eqref{prop:shrik_prop:6c}  follows from the proof of \cite[III, Corollary 7.3]{Har_du}.} 
\end{itemize}
\end{rem}
\begin{rem} We considered many isomorphisms between various functors. In this appendix, we will try to be carefull and not identify these functors, but rather use the appropriate isomorphism each time. However, outside the appendix we will identify isomorphic functors and object more freely in order to make the text more readable.
\end{rem}

\subsubsection{Duality}
\begin{notn}
Let $\phi:X \to Y$ be a morphism of schemes.
\begin{enumerate}
\item The \term{relative dualizing object}{RelDuOb}\ is $\mathcal{D}_{X/Y}:=\phi^!(O_Y) \in D^+_{}(X)$.
\item The \term{dualizing object}{DuOb}\ of $X$ is $\mathcal{D}_X :=\mathcal{D}_{X/\spec(k)}\in D_{}^{+}(X)$.\item The shifted  dualizing objects are $\Omega_{X/Y}:=\mathcal{D}_{X/Y}[\dim Y- \dim X]$ and $\Omega_{X}:=\Omega_{X/\spec(k)}.$
\end{enumerate}
\end{notn}

\begin{remark} There is an abuse of notations in the definition of the shifted dualizing sheaves: for smooth or CI maps, we use the notation $\Omega$ to denote the sheaves of (relative) top differential forms. However, in these cases, $\Omega_{X}$ and $\Omega_{X/Y}$ are canonically identified with the sheaves of top differential forms and relative top differential forms. For simplicity, we will use this notation.
\end{remark}

\begin{rem} There is no standard definition of $\mathcal{D}_X$ in the generality discussed in \cite{Har_du,Con} since the schemes there are not of finite type over a field. Instead, in \cite{Har_du,Con}, there is a more general notion of a ``dualizing object", which is unique only up to a shift and twist by a line bundle (see \cite[V,\S3]{Har_du}). In order to define the derived pull back for a morphism of finite type $\pi$ between two schemes, they choose dualizing objects which are compatible via $\pi ^!$.

In the generality discussed in this article, we choose the standard normalization of this object. The fact that the object that we described is a ``dualizing object" follows from the fact that $\pi^!$ maps a dualizing object to a dualizing object, see \cite[V,\S10]{Har_du}. In particular, $\mathcal{D}_X$ satisfies the properties of a dualizing object described in \cite{Har_du,Con}. We will state the ones that are important to us in Theorems \ref{thm:fid} and \ref{thm:do}.

\end{rem}

\begin{theorem}[{cf. \cite[V,\S2, The definition on page 83 and I, Proposition 7.6]{Har_du}}]\label{thm:fid}
 The object $\mathcal{D}_X$ is  in $D^b(X).$ Moreover, it has a representative which is  a bounded complex consisting of injective sheaves.
\end{theorem}

\begin{notn} Let $X$ be a scheme. Using Theorem \ref{thm:fid}, we define the the \term{duality}{Du}\ (contravariant) functor $\bD:D(X) \to D(X)$ by $\bD(\cF):=\mathcal{RH}om(\cF, \mathcal{D}_X)$, where $\mathcal{RH}om$ is the internal hom (see e.g. \cite[II,\S3]{Har_du}).
\end{notn}

\begin{thm}[{cf. \cite[V,\S2]{Har_du}}] \label{thm:do}
The natural morphism $Id \to \bD \circ \bD$ is isomorphism.
\end{thm}

\begin{prop}
\label{prop:star_shrik}
Let $\phi:X \to Y$ be a morphism of schemes. Then
\begin{enumerate}
\item\label{thm:do:9}  The restrictions of the functors $\bD \circ L\phi^* \circ \bD$ and $\phi^!$ to their joint domains of definition are isomorphic.
\item\label{thm:do:10} Assuming that $\phi$ is proper, the restrictions of the functors $\bD \circ R\phi_* \circ \bD$ and $R\phi_*$ to their joint domains of definition are isomorphic.
\end{enumerate}
\end{prop}
\begin{proof}$ $
\eqref{thm:do:9}  is part of the construction of $\phi^!$ (see  \cite[(3.3.6)]{Con}). \eqref{thm:do:10} follows from \eqref{thm:do:9} and Theorem \ref{thm:do} or alternativly from \cite[VII,Corollary 3.4.,(c)]{Har_du} and Theorem \ref{thm:do}.
\end{proof}

\subsection{Cohen--Macaulay Schemes} \lbl{ssec:CM}

\begin{defn}\label{def:CM}
Let $X$ be a scheme and $\cF$ be a coherent sheaf over $X$.
\begin{enumerate}
\item $\cF$ is said to be  \term{Cohen--Macaulay}{CMS}\ of dimension $i$ if $H^j(\bD(\cF))=0$  for all $j \neq -i$.
\item $X$ is  said to be  \term{Cohen--Macaulay}{CMA}\ if $\cO_X$ is \CMS. In other words,  $X$ is Cohen--Macaulay if $\Omega_X$ is a sheaf.
\end{enumerate}
\end{defn}
\begin{remark}
$ $
\begin{itemize}
\item This definition is taken from \cite[\S\S2.7.]{BBG}. It is explained there why it is equivalent to the classical one.
\item According to these definitions, the notion of Cohen--Macaulay {module} of dimension $i$ is local but the notion of Cohen--Macaulay is not. {The reason is that the support of a module can have different dimensions in different points.}
\end{itemize}
\end{remark}
\begin{thm}[{see eg. \cite[\S\S2.5]{BBG}}]\label{thm:CM_eq}
Let $A$ be a commutative algebra and let $M$  be a finitely generated module over $A$, considered also as a sheaf over $\spec(A)$. The following conditions are equivalent:
\begin{enumerate}
\item $M$ is \CMS\ of dimension $i$.
\item There exists a polynomial subalgebra $B \subset A$ of dimension $i$ such that $M$ is projective (equivalently\footnote{{The equivalence is by the Quillen--Suslin theorem \cite{Qui}.}}, free) and finitely generated over $B$.
\item For any polynomial subalgebra $B \subset A$ of dimension $i$ such that $M$ is finitely generated over $B$, we have that $M$ is projective (equivalently, free) over $B$.
\item There exists a regular subalgebra $B \subset A$ of (pure) dimension $i$ such that $M$ is projective and finitely generated over $B$.
\item For any regular subalgebra $B \subset A$ of (pure) dimension $i$ such that $M$ is finitely generated over $B$, we have that $M$ is projective over $B$.
\end{enumerate}
\end{thm}

\begin{rem}
$ $
\begin{itemize}
\item The formulation in \cite[\S\S2.5]{BBG} is slightly different from ours. In order to deduce our version from theirs, one should use the Noether normalization lemma.
\item As commented in \cite[\S\S2.7.]{BBG}, this theorem is an immediate corollary of Proposition \ref{prop:star_shrik}, when we use Definition \ref{def:CM} for the notion of \CMS\ modules.
\end{itemize}
\end{rem}
\begin{cor}\label{cor:cm_torfree_i}
$ $
Let $M$ be a \CMS\ module   of dimension $i$.
Then the dimension of the support of  any non-zero  $m \in M$  is $i$.
\end{cor}

Theorem \ref{thm:do} implies
\begin{proposition}
If $X$ is a \CMS\  scheme, then $\Omega_X$ is a \CMS\ module.
\end{proposition}

Together with Corollary \ref{cor:cm_torfree_i}, we get
\begin{cor}\label{cor:cm_torfree}
If $X$ is \CMS\  scheme, then $\Omega_X$ is a torsion free module.
\end{cor}

\begin{thm}[{see  \cite[Theorem 2.9 (3)]{BBG}\footnote{This is a simple corollary of  Theorem \ref{thm:CM_eq}}}]
Let $\mu: \cF \to \cG$ be a morphism of \CMS\ sheaves of dimension $i$. Suppose that $\mu$ is an isomorphism in dimension $i$, i.e. that the kernel and cokernel of $\mu$ have dimension less than $i$. Then $Ker (\mu)= 0$ and $Coker (\mu)$ is a \CMS\ of dimension $i - 1$. In particular, if $\mu$ is an isomorphism in dimension $i-1$, then $\mu$ is an isomorphism
\end{thm}

\begin{cor}[{cf. \cite[Corollary 2.3 (b)]{BL}}] \lbl{cor:codim.2.in.CM}
$ $
\begin{enumerate}
\item\label{cor:codim.2.in.CM:1} A \CMA\ variety which is regular outside a subvariety of codimension 2 is normal.
\item Let  $X$ be a \CMA\ scheme, and let $i:Y \hookrightarrow X$ be an open embedding such that $X\smallsetminus i(Y)$ has co-dimension $2$ in $X$. Then the natural morphism $\Omega_{X} \to i_*(\Omega_{Y})$ is an isomorphism.
\end{enumerate}
\end{cor}

\begin{defn}
A scheme $X$ is \term{Gorenstein}{Gor}\ if it is \CMS\ and $\Omega_X$ is  a line bundle.
\end{defn}

\S\S\S\ref{sec:pi_shrik} implies
\begin{cor}
{An LCI scheme is  \Gor.}
\end{cor}
\subsection{Resolution of Singularities} \label{sssec:res_sing}

$ $
We will use Hironaka's theorem on resolution of singularities in characteristic 0, proved in \cite{Hir}. A more recent overview can be found in \cite{Hir_mod}.
\begin{definition}
Let $X$ be an algebraic variety.
\begin{itemize}
\item A \emph{resolution of singularities} of $X$ is a proper map $p:Y \to X$ such that $Y$ is smooth and $p$ is a birational equivalence.
\item A \emph{strong resolution of singularities} of $X$ is a resolution of singularities $p:Y \to X$ which is isomorphism over the smooth locus of $X$.
\item A subvariety $D \subset X$ is said to be a \emph{normal crossings divisor}
(or \emph{NC divisor}) if,
for any $x\in D$, there exists an \et neighborhood $\phi: U \to X$ of $x$
and an \et map $ {\alpha}:U \to \A^n$ such that $\phi^{-1}(D)= {\alpha}^{-1}(D')$, where $D' \subset \A^n$
is a union of coordinate hyperplanes.
\item
A subvariety $D \subset X$ is said to be a \emph{strict normal crossings divisor}
(or \emph{SNC divisor}) if,
for any $x\in D$, there exists a Zariski neighborhood $U \subset X$ of $x$
and an \et map $ {\alpha}:U \to \A^n$ such that $D \cap U= {\alpha}^{-1}(D')$, where $D' \subset \A^n$
 is a union of coordinate hyperplanes.
\item We say that a resolution of singularities $p:Y \to X$ \emph{resolves} (respectively, \emph{strictly resolves}) a closed subvariety $D \subset X$ if $p^{-1}(D)$ is an NC divisor (respectively, an SNC divisor). In this case, we will also say that $p:Y \to X$ is a \emph{resolution} (respectively, a \emph{strict resolution}) of the pair $(X,D)$.
\item {Let $D \subset X$ be a subvariety of co-dimension 1}. A \emph{strong resolution} of  the pair $(X,D)$ is a strict resolution  $p:Y \to X$, which is isomorphism outside the union of  the singular locus of $X$ and the singular locus of $D$. 
\item {Let  $p:Y \to X$ be a resolution of singularities. Let $U\subset X$ be the maximal open set on which $p$ is an isomorphism. Let $D \subset X$ be a subvariety. The {\emph{strict transform}} of $D$ is defined to be $\overline{p^{-1}(D\cup U)}$ }
\end{itemize}
\end{definition}

\begin{theorem}[Hironaka] \label{thm:Hir}
Let $D \subset X$ be a pair of algebraic varieties. {Assume that $D \subset X$ is of co-dimension 1}, and let
$U\subset X$ be a smooth open subset such that $U \cap D$ is also smooth. Then there exists a resolution of singularities
$p:\tilde X \to X$ that resolves $D$, and such that the map $p^{-1}(U) \to U$ is an isomorphism.
\end{theorem}

There is a standard procedure to resolve a normal crossings divisor and obtain a strict normal crossings divisor,
 see e.g. \cite{Jon}.
This gives the following corollary.

\begin{corollary} Any  pair of algebraic varieties ${D \subset }X$ admits a strong resolution.
\end{corollary}

The following proposition is standard
\begin{prop} \label{prop:strict.transform}
Let $D \subset X$ be a pair of algebraic varieties such that $D$ is irreducible, has codimension one in $X$, and is not contained in the singular locus of $X$.  Let $p:Y \to X$ be a strong resolution of the pair   $(X,D)$, and let $D'$ be the strict transform of $D$. Then $p|_{D'}:D'\to D$ is a resolution of singularities.
\end{prop}
{
For completeness, we include the proof of this proposition. We will use the following:
\begin{lemma}\label{lem:SNC.irr.sm} An irreducible SNC divisor $D \subset X$ is smooth.
\end{lemma}
\begin{proof}
Let $x \in D$. Let $U \ni x$ be a Zariski open neighborhood and $\phi:U \to \A^n$ an etale map such that $D$ is a pre-image of a union $S$ of coordinate hyperplanes. Without loss of generality, we may assume that $\phi(x)=0$. This mean that $S$ has to consist of a unique hyperplane, so $D$ is smooth.
\end{proof}

\begin{proof}[Proof of proposition \ref{prop:strict.transform}]
We have to prove the following:
\begin{enumerate}
\item $p|_{D'}:D'\to D$ is proper.
\item $p|_{D'}:D'\to D$ birational equivalence.
\item $D'$ is smooth.
\end{enumerate}
Statement (1) is obvious. Let $U\subset X$ be the maximal open set on which $p$ is an isomorphism. Since $p$ is a strong resolution, $U \cap D$ is non-empty, so it is open and dense in $D$. This proves (2). Since $D$ is irreducible, $U \cap D$ is irreducible, so $p^{-1}(U \cap D)$ is irreducible. It follows that $D'=\overline{p^{-1}(U \cap D)}$ is irreducible. On the other hand, the fact that $p^{-1}(D)$ is an SNC divisor, implies that so is $D'$. Assertion (3) follows now from Lemma \ref{lem:SNC.irr.sm}
\end{proof}
}
\subsection{Grauert--Riemenschneider Vanishing Theorem} \lbl{ssec:GR}

\begin{thm}[{Grauert--Riemenschneider, cf. \cite[Theorem 4.3.9 and Notation 1.1.7]{laz}}]\label{thm:GR} $ $
Let $\phi:X \to Y$ be a resolution of singularities. Then $R\phi_*(\Omega_X)=\phi_*(\Omega_X)$, i.e. $R^i\phi_*(\Omega_X)=0$ for all $i \neq 0$
\end{thm}

\subsection{Rational Singularities} \lbl{ssec:RS}
The notion of rational singularities was introduced in \cite{Art} for surfaces, and in \cite{KKMS} in general.
\begin{defn}[{cf. \cite[I \S 3, page 50-51]{KKMS}}] \label{defn:rational.singularities} Let $X$ be an algebraic variety.
\begin{enumerate}
\item We say that $X$ has \term{rational singularities}{RS}\footnote{A more accurate expression would be `the singularities of $X$ are all rational', but `has rational singularities' is more commonly used.} if, for any (or, equivalently\footnote{The equivalence is via Hironaka's theorem.} for some) resolution of singularities $p:\tilde X \to X$, the natural morphism  $Rp_*(\cO_{\tilde X}) \to \cO_{X}$ is an isomorphism.
\item A (usually singular) point $x \in X(k)$ is a \term{rational singularity}{RSPt}\ if there is a Zariski neighborhood $U\subset X$ of $x$ that has \RS.
\end{enumerate}
\end{defn}
The properties of duality (Proposition \ref{prop:star_shrik}), the  Grauert--Riemenschneider Theorem (Theorem \ref{thm:GR}), and Corollary \ref{cor:codim.2.in.CM} imply the following:
\begin{prop} \label{prop:equviv_rat}
Let $X$ be an algebraic variety. Then
\begin{itemize}
\item The following are equivalent:
\begin{enumerate}
\item \label{prop:equviv_rat:1}
$X$ has \RS.
 \item \label{prop:equviv_rat:2} For any (or, equivalently, for some) resolution of singularities $p:\tilde X \to X$, the trace map $Tr_{p}: R p_*(\Omega_{\tilde X})\to \Omega_{X}$ is an isomorphism.
 \item \label{prop:equviv_rat:3} $X$ is \CMA\ and, for any (or, equivalently, for some) resolution of singularities $p:\tilde X \to X$, the trace map $Tr_{p}:p_*(\Omega_{\tilde X}) \to \Omega_{X}$ is an isomorphism.
 \item \label{prop:equviv_rat:4}$X$ is \CMA\ and, for any (or, equivalently, for some) resolution of singularities $p:\tilde X \to X$, the trace map $Tr_{p}:p_*(\Omega_{\tilde X}) \to \Omega_{X}$ is onto.
 \item \label{prop:equviv_rat:5} $X$ is \CMA, normal (or, equivalently, regular outside co-dimension $2$), and, for any (or, equivalently, for some) resolution of singularities $p:\tilde X \to X$, the {morphism $p_*(\Omega_{\tilde X})\to i_*(\Omega_{X^{sm}})$ which is the composition of the trace map and the isomorphism from Corollary \ref{cor:codim.2.in.CM},}is an isomorphism. Here, $i:X^{sm} \to X$ is the embedding of the regular locus.
\end{enumerate}

\item
If $X$ is affine, these conditions are also equivalent to the following:
\begin{itemize}
 \item[{\ref{prop:equviv_rat:4}'.}]\label{prop:equviv_rat:4p} $X$ is \CMA\ and, for
 \begin{itemize}
\item Any (or, equivalently, some) strong resolution of singularities $p:\tilde X \to X$ and
\item {Any section $\omega \in \Gamma(X,\Omega_X)$},
 \end{itemize}
{there exists a top differential form $\tilde \omega \in \Gamma({\tilde X},\Omega_{\tilde X})$ such that} $$\omega|_{X^{sm}}=\tilde\omega|_{X^{sm}}.$$
Here, we consider $X^{sm}$ as a subset of both $X$ and $\tilde X$. More formally, the last equality mean that,  $e_{pr_{X^{sm}}}(b_{j,pr_{X}}(\omega|_{X^{sm}}))=(e_{{pr_{\widetilde X}}}\widetilde{\omega})|_{X^{sm}}$, where $pr_{X^{sm}},pr_{X},pr_{\widetilde X}$ are the projections to the point $\spec(k)$, and $j:{X^{sm}} \to X$ is the embedding.
\item[{\ref{prop:equviv_rat:5}'.}]\label{prop:equviv_rat:5p} $X$ is \CMA, normal (or, equivalently, regular outside co-dimension $2$), and for
 \begin{itemize}
\item Any (or, equivalently, some) strong resolution of singularities $p:\tilde X \to X$ and
\item {Any top differential form $\omega \in \Gamma(X^{sm},\Omega_{X^{sm}})$},
\end{itemize}
there exists a top differential form $\tilde \omega \in \Gamma({\tilde X},\Omega_{\tilde X})$ such that $$\omega=\tilde\omega|_{X^{sm}}.$$
 \end{itemize}

 \end{itemize}
\end{prop}
\begin{proof}$ $
\begin{itemize}
\item[\eqref{prop:equviv_rat:1}$\Leftrightarrow$ \eqref{prop:equviv_rat:2}:] See \cite[Page 141]{elk}. A proof can also be given using Proposition \ref{prop:star_shrik}.
\item[\eqref{prop:equviv_rat:2}$\Leftrightarrow$ \eqref{prop:equviv_rat:3}:] Follows from Grauert--Riemenschneider Theorem (Theorem \ref{thm:GR}).
\item[\eqref{prop:equviv_rat:3}$\Leftrightarrow$ \eqref{prop:equviv_rat:4}:]
We have to prove that, under the conditions of  \eqref{prop:equviv_rat:3}, the trace map ${Tr_{p}:}p_*(\Omega_{\tilde X}) \to \Omega_{X}$ is a monomorphism. The claim is local, so we can assume that $X$ is affine. Let $\omega \in \Gamma({\tilde X},\Omega_{\tilde X})=\Gamma(X,p_*(\Omega_{\tilde X}))$ be in the kernel of $Tr_{p}$. Let $U \subset \tilde X$  be the open dense subvariety on which $p$ is isomorphism. Let $q$ be the restriction $p|_U$ considered as a map to its image. By Proposition \ref{prop:shrik_prop}\eqref{prop:shrik_prop:3c}, $\omega|_{U}$ is in the kernel of $Tr_{Id_{U}}$. By Proposition \ref{prop:shrik_prop}\eqref{prop:shrik_prop:42}, this implies that $\omega|_{U}=0$, so we get $\omega=0$.
\item[\eqref{prop:equviv_rat:4}$\Rightarrow$ \eqref{prop:equviv_rat:5}:]
The normality follows from \eqref{prop:equviv_rat:1}, the rest from \ref{cor:codim.2.in.CM}.
\item[\eqref{prop:equviv_rat:4}$\Leftarrow$ \eqref{prop:equviv_rat:5}:]
Follows from \ref{cor:codim.2.in.CM}.
\item[\eqref{prop:equviv_rat:4}$\Rightarrow$ (\ref{prop:equviv_rat:4}'):] Follows from \ref{prop:shrik_prop}\eqref{prop:shrik_prop:3c}.
\item[\eqref{prop:equviv_rat:4}$\Leftarrow$ (\ref{prop:equviv_rat:4}'):] Let $\omega \in \Gamma(X,\Omega_X)$. Let $\tilde \omega \in \Gamma({\tilde X},\Omega_{\tilde X})=\Gamma(X,p_*(\Omega_{\tilde X}))$ be as in (\ref{prop:equviv_rat:4}'). We have to show that $Tr_p(\tilde \omega)= \omega$. By \ref{prop:shrik_prop}\eqref{prop:shrik_prop:3c}, $Tr_p(\tilde \omega)|_{X^{sm}}= \omega|_{X^{sm}}$. The assertion follows now from the fact that $\Omega_X$ is torsion free (see Corollary \ref{cor:cm_torfree}).
\item[(\ref{prop:equviv_rat:4}')$\Rightarrow$ (\ref{prop:equviv_rat:5}'):] The normality follows from \eqref{prop:equviv_rat:1}. We have to show that for any $\omega_0 \in \Gamma(X^{sm},\Omega_{X^{sm}})$, there is $\omega \in \Gamma(X,\Omega_{X})$ such that $b_{j,pr_{X}}(\omega|_{X^{sm}})=\omega_0$. This follows from Corollary \ref{cor:codim.2.in.CM}(2).
\item[(\ref{prop:equviv_rat:4}')$\Leftarrow$ (\ref{prop:equviv_rat:5}'):] This is obvious.
 \end{itemize}
\end{proof}

A fundamental property of \RS\ which is crucial to our paper is the following theorem
\begin{theorem}[\cite{elk}]\label{thm:elk}
Let $\phi:X\to Y$ be a flat morphism, and assume that $Y$ is smooth. Let $U:=\{x \in X(k)| x \text{ is a \RSPt\ of } \phi^{-1}(\phi(x))\}$. Then: \begin{enumerate}
\item $U$ is open.
\item Any $x \in U$ is a  \RSPt\ of $X$.
\end{enumerate}
\end{theorem}

\section{Examples for \S\S\ref{ssec:mult_FRS}--\ref{ssec:Comb.stat.sp}}\label{app:ex}
We present some examples of the graphs that we discus in \S\S\ref{ssec:mult_FRS}--\ref{ssec:Comb.stat.sp}.
\begin{center}\includegraphics[width=160mm]{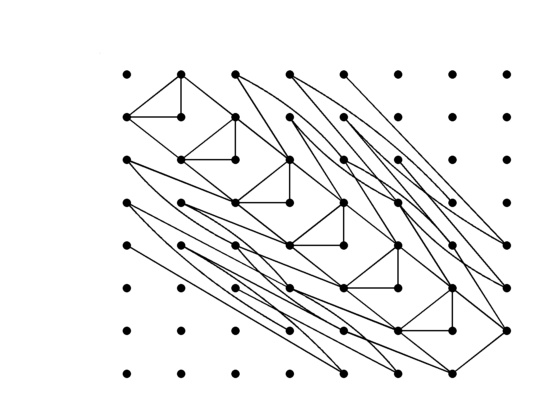}
The graph $\Gamma_3$ of \S\S\ref{ssec:mult_FRS} for the case $d=8$. The vertex $(i,j)$ is the dot in the $i$th row and $j$th column.

\includegraphics[width=160mm]{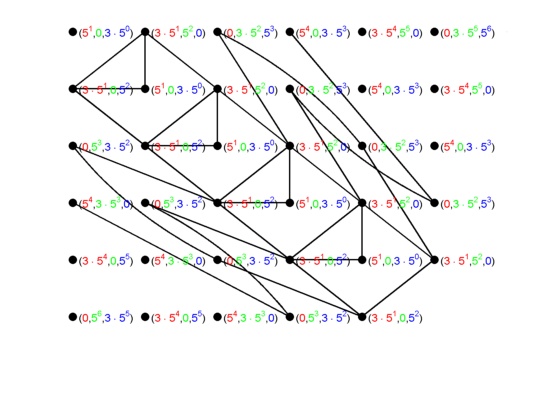}
The weights $w_3$ on the  graph $\Gamma_3$ from \S\S\ref{ssec:mult_FRS} for the case $d=6.$ Near each vertex we see its three weights in three different colors.
\end{center}

Using the procedure in {Definition \ref{def:col}, Corollary \ref{cor:col} and Remark \ref{rem:col}} we get the following coloring on $\Gamma_3$:
\begin{center}\includegraphics[width=160mm]{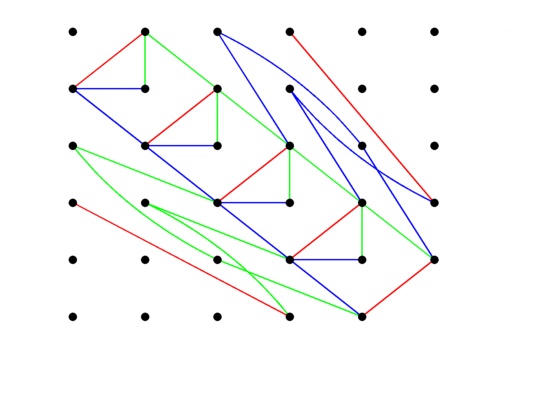}
The graph $\Gamma_3$ of \S\S\ref{ssec:mult_FRS} for the case $d=6,$ colored.

\includegraphics[width=160mm]{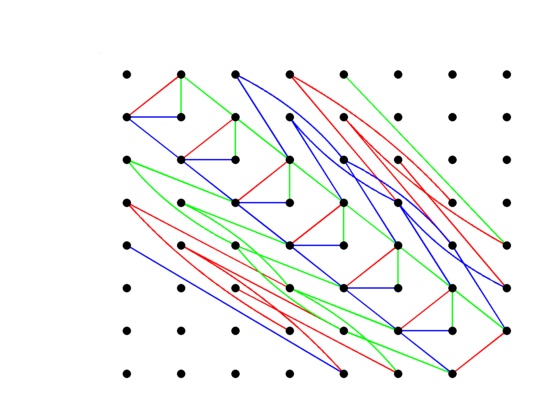}
The graph $\Gamma_3$ of \S\S\ref{ssec:mult_FRS} for the case $d=8,$ colored.
\end{center}
Each of the colors represent one level of the graph $\Gamma_4$. In other words, for each  $m$ the graph $\Gamma_4^m$ consist of the vertices of $\Gamma_3$ with only edges of one color.

\begin{center}
\includegraphics[width=160mm]{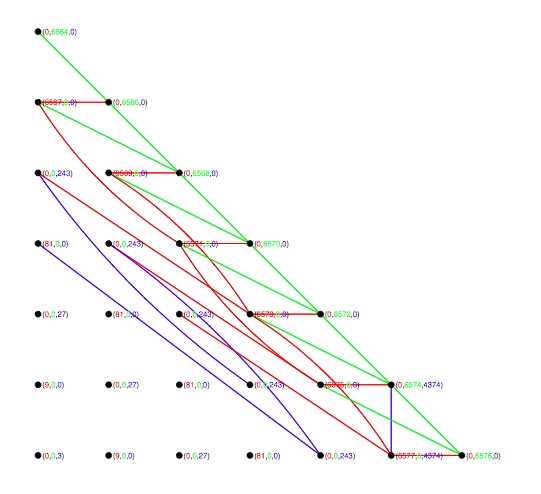}

The graph $\Gamma_3$ of \S\S\ref{ssec:Comb.stat.so} for the case $d=8$, with the weights \nir{$w_3$} of the vertices and colors of edges. The vertex corresponding to $\left\{ i,j \right\}$, $i<j$ is the dot in the $i$th row and $j$th column. The graphs $\Gamma_4^m$ are obtained by keeping the edges of a single color.

\includegraphics[width=160mm]{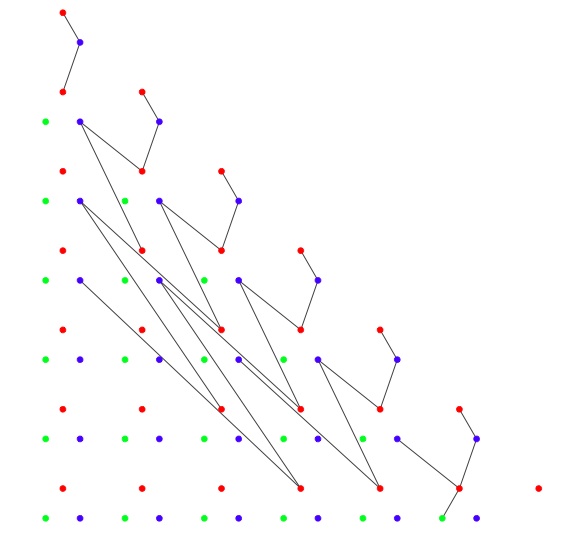}

The graph $\Gamma_8$ of \S\S\ref{ssec:Comb.stat.sp} for the case $d=7$. A red dot at coordinate $(i,j), i \leq j$ corresponds to the multiset $[i,j]\in \left\{ 1,\ldots,d \right\} ^{[2]}=I_2$. A blue dot at coordinate $(i,j), i \leq j$ corresponds to the ordered pair $(i,j)\in \left\{ 1,\ldots,d \right\} ^2\smallsetminus \left\{ (d,d) \right\}=I_0$. A green dot at coordinate $(i,j), i < j$ corresponds to ordered pairs $(j,i)\in I_0$.
\end{center}


\begin{thebibliography}{Kly844}

\bibitem[\href{ http://www.jstor.org/stable/2373050}{Art}]{Art} Artin, M. \emph{On isolated rational singularities of surfaces} American Journal of Mathematics      88/1 (1966).

\bibitem[Avn]{A} Avni, N.; \emph{Arithmetic groups have rational representation growth.} Ann. of Math. (2) 174 (2011), no. 2, 1009Ð1056.

\bibitem[AKOV]{AKOV} Avni, N.; Klopsch, B.; Onn, U., Voll, C. \emph{Representation zeta functions of compact p-adic analytic groups and arithmetic groups} Duke Math. Journal, to appear.

\bibitem[BLMM]{BLMM} Bass, H.; Lubotzky, A.; Magid, A. R.; Mozes, S. \emph{The proalgebraic completion of rigid groups.} Proceedings of the Conference on Geometric and Combinatorial Group Theory, Part II (Haifa, 2000). Geom. Dedicata 95 (2002), 19--58.

\bibitem[\href{http://www.math.tau.ac.il/~bernstei/Publication_list/publication_texts/cohen-macauley.pdf}{BBG}]{BBG}
Bernstein, J.N.; Braverman, A.; Gaitsgory, D. \emph{The Cohen-Macaulay property of the category of $(\fg,K)$-modules}, Selecta Mathematica, New Ser. 3 (1997).

\bibitem[\href{http://www.math.tau.ac.il/~bernstei/Publication_list/publication_texts/kostant.pdf}{BL}]{BL}
Bernstein, J. N.; Lunts, V. \emph{A simple proof of kostant's theorem that $u(\fg)$ is free over its center}, Amer. Jour. Math. 118/5 (1996).

\bibitem[Bou]{B} Boutout, J-F., \emph{Singularites rationnelles et quotients par les groupes reductifs} Inventiones Mathematicae, 88 (1987) 65--68.

\bibitem[Bur]{Bur} Burns, D., \emph{On rational singularities in dimensions $>2$} Math. Ann. 211 (1974), 237--244.

\bibitem[CK]{CK} Chang, C. C.; Keisler, H. J. \emph{Model theory.} Third edition. Studies in Logic and the Foundations of Mathematics, 73. North-Holland Publishing Co., Amsterdam, 1990. xvi+650 pp.

\bibitem[CL]{CL} Cluckers, R.; Loeser, F. \emph{Ax--Kochen--Ersov theorems for p-adic integrals and motivic integration.} Geometric methods in algebra and number theory, 109--137, Progr. Math., 235, Birkhauser Boston, Boston, MA, 2005.

\bibitem[{Con1}]{Con} B. Conrad, \emph{Grothendieck Duality and Base Chainge} Springer Lecture Notes. no. 1750 (2000).

\bibitem[Con2]{Con_Nag} Conrad, B. \emph{Deligne's notes on Nagata compactifications.} J. Ramanujan Math. Soc. {\bf 22/3}, 2007;
see also
B. Conrad,
\emph{Erratum for ``Deligne's notes on Nagata compactifications'',}
J. Ramanujan Math. Soc. {\bf 24/4}, 2009.

\bibitem[CR]{CR} Curtis, Charles W.; Reiner, \emph{Irving Representation theory of finite groups and associative algebras.} Reprint of the 1962 original. AMS Chelsea Publishing, Providence, RI, 2006. xiv+689 pp.

\bibitem[CDS]{CDS} Cushman, R.; Duistermaat, H.; Sniatycki, J. \emph{Geometry of nonholonomically constrained systems.} Advanced Series in Nonlinear Dynamics, 26. World Scientific Publishing Co. Pte. Ltd., Hackensack, NJ, 2010. xviii+404 pp.

\bibitem[Den]{De-Int} Denef, J. \emph{On the evaluation of certain p-adic integrals.} SŽminaire de thŽorie des nombres, Paris 1983Ð84, 25Ð47, Progr. Math., 59, BirkhŠuser Boston, Boston, MA, 1985.

\bibitem[Dre]{DR} Drezet, J. M. \emph{Luna's slice theorem and applications} available at http://www.math.jussieu.fr/~drezet/papers/Wykno.pdf .

\bibitem[Elk]{elk} R. Elkik,
\emph{Singularites rationnelles et deformations,}
Inventiones math. 47, (1978)

\bibitem[FQ]{FQ} Freed, D.; Quinn, F. \emph{Chern--Simons theory with finite gauge group.} Comm. Math. Phys. 156 (1993), no. 3, 435--472.

\bibitem[Ful]{Fu} Fulton, W. \emph{Introduction to Toric Varieties.} Princeton University Press (1993).

\bibitem[GM]{GM} Gelander, T.; Minsky, Y.N. \emph{The dynamics of Aut(Fn) on redundant representations.}
Groups Geom. Dyn. 7, No. 3, 557-576 (2013).

\bibitem[Gol]{Gol} Goldman, W. \emph{The symplectic nature of fundamental groups of surfaces.} Adv. Math. 54, 200 (1984).

\bibitem[\href{http://www.numdam.org/item?id=PMIHES_1966__28__5_0}{Gro}]{EGA}  A. Grothendieck
\emph{Elements de geometrie algebrique IV,}
Publications mathématiques de l' I.H.E.S., tome 28 (1966),

\bibitem[{Har1}]{Har_ag}
R. Hartshorne, \emph{Algebraic geometry},
Springer-Verlag, New York   (1977).


\bibitem[{Har2}]{Har_du}
R. Hartshorne, \emph{Residues and Duality} Springer Lecture Notes. no. 20 (1966)


\bibitem[Hes]{He} Hesselink, Wim H. \emph{Cohomology and the resolution of the nilpotent variety.} Math. Ann. 223 (1976), no. 3, 249--252.

\bibitem[Hin]{Hi} Hinich, V. \emph{On the singularities of nilpotent orbits.} Israel J. Math. 73 (1991), no. 3, 297--308.

\bibitem[\href{http://www.ams.org/mathscinet-getitem?mr=199184}{Hir}]{Hir}
 H. Hironaka, 
 \emph{Resolution of singularities of an algebraic variety over a field of characteristic zero.}
  I, II. Ann. of Math. {\bf 79/2}, 109--326, 1964. 

\bibitem[Jai]{Ja} Jaikin-Zapirain, A. \emph{Zeta function of representations of compact p-adic analytic groups.} J. Amer. Math. Soc. 19 (2006), no. 1, 91--118

\bibitem[\href{http://www.math.columbia.edu/~dejong/papers/ALTERATIONS.dvi}{Jon}]{Jon} J. de Jong,
\emph {Smoothness, semi-stability and alterations.}
 Inst. Hautes Etudes Sci. Publ. Math, {\bf 83}, 1996.



\bibitem[{KKMS}]{KKMS}
G. Kempf,  F. Knudsen, , D. Mumford, B. Saint-Donat, \emph{Toroidal Embeddings I.} Springer Lecture Notes. no. 339 (1973)


\bibitem[{Kol}]{Hir_mod}
J. Kollar,
\emph{Lectures on resolution of singularities.}
Annals of Math. Studies. Princeton University Press, 2007.

\bibitem[Kos]{Ko} Kostant, Bertram \emph{Lie group representations on polynomial rings.} Amer. J. Math. 85 1963 327--404.

\bibitem[LW]{LW} Lang, S.; Weil, A. \emph{Number of Points of Varieties in Finite Fields.} American Journal of Mathematics , Vol. 76, No. 4 (Oct., 1954), pp. 819-827.

\bibitem[LL]{LL} Larsen, Michael; Lubotzky, Alexander \emph{Representation growth of linear groups.} J. Eur. Math. Soc. (JEMS) 10 (2008), no. 2, 351--390.

\bibitem[{Laz}]{laz}
R.Lazarsfeld,  \emph{Positivity in Algebraic Geometry I.}
A Series of Modern Surveys in Math. Vol. 48, Springer, 2004.

\bibitem[Li]{Li} Li, Jun \emph{The space of surface group representations.} Manuscripta Math. 78 (1993), no. 3, 223--243.

\bibitem[LS]{LS}  Liebeck, Martin W.; Shalev, Aner \emph{Fuchsian groups, finite simple groups and representation varieties.} Invent. Math. 159 (2005), no. 2, 317--367.

\bibitem[LS2]{LS2} Liebeck, Martin W.; Shalev, Aner \emph{Fuchsian groups, coverings of Riemann surfaces, subgroup growth, random quotients and random walks.} J. Algebra 276 (2004), no. 2, 552--601.

\bibitem[LS3]{LS3} Liebeck, Martin W.; Shalev, Aner \emph{Character degrees and random walks in finite groups of Lie type.} Proc. London Math. Soc. (3) 90 (2005), no. 1, 61--86.

\bibitem[LM]{LM} Lubotzky, Alexander; Martin, Benjamin \emph{Polynomial representation growth and the congruence subgroup problem.} Israel J. Math. 144 (2004), 293--316.

\bibitem[Nor]{Nor} Nori, M.; \emph{On subgroups of $\GL_n(\mathbb{F}_p)$}. Invent. Math. 88 (1987), 257--275

\bibitem[{Qui}]{Qui}
D. Quillen,  \emph{Projective modules over polynomial rings}, Inventiones Mathematicae 36, (1976)


\bibitem[RR]{RR} Ranga Rao, R. \emph{Orbital integrals in reductive groups.} Ann. of Math. (2) 96 (1972), 505--510.

\bibitem[Ser]{Se} Serre, J.P. \emph{Lie Algebras and Lie Groups.} Lecture Notes in Mathematics, {\bf 1500}. Springer-Verlag, New York, 1964.


\bibitem[Sim]{Si} Simpson, C. \emph{Moduli of representations of the fundamental group of a smooth projective variety II}

\bibitem[Ste]{Ste} Steinberg, R. \emph{Endomorphisms of linear algebraic groups.} Memoirs of the AMS 80 (1968).

\bibitem[Wit]{Wi} Witten, Edward \emph{On quantum gauge theories in two dimensions.} Comm. Math. Phys. 141 (1991), no. 1, 153--209.

\end{thebibliography}
\end{document}